\documentclass[11pt,a4paper,twoside]{article}

\usepackage[a4paper]{geometry}
\usepackage{amsfonts}
\usepackage{amsthm}
\usepackage{amsmath}
\usepackage{parskip}
\usepackage{mathrsfs}
\usepackage{graphicx}
\usepackage{amssymb}
\usepackage{xtab}

\usepackage{amsmath}
\usepackage{amsfonts}
\usepackage{mathtools}
\usepackage{bbm}
\usepackage{amssymb}
\usepackage[dvipsnames]{xcolor}
\usepackage{graphicx}
\usepackage[numbib]{tocbibind}
\usepackage{array} % used for many of the table properties in this document
\usepackage{multirow} % used for the multirow commands
\usepackage{tabu} % used for the \begin{tabu}... \end{tabu} commands
\usepackage{bm}
\usepackage{enumitem}
\usepackage{mathrsfs}
\usepackage{etoolbox}
\usepackage[numbers]{natbib}
\usepackage{cite}
\usepackage{quiver}
\usepackage{tikz}
\usepackage{tikz-cd}
\usepackage{subcaption}

\usepackage{amsthm} %  package used to make the theorem environments work.
\theoremstyle{plain}
\newtheorem{thm}{Theorem}[section] 
\newtheorem{lem}[thm]{Lemma} 
\newtheorem{prop}[thm]{Proposition} 
\newtheorem{cor}{Corollary}

\newtheorem*{conj*}{Conjecture}

\theoremstyle{definition}
\newtheorem{defn}{Definition}[section]
\newtheorem{eg}{Example}[section]

\theoremstyle{remark}

\renewcommand{\refname}{Bibliography}
\DeclareMathOperator{\SU}{SU}

\DeclareMathOperator{\SL}{SL}
\DeclareMathOperator{\GL}{GL}
\DeclareMathOperator{\Orth}{O}
\DeclareMathOperator{\SO}{SO}
\DeclareMathOperator{\Spin}{Spin}
\DeclareMathOperator{\Pin}{Pin}
\newcommand{\R}{\mathbb{R}}
\newcommand{\N}{\mathbb{N}}
\newcommand{\C}{\mathbb{C}}

\newcommand{\Span}[1]{\left<{#1}\right>}

\newcommand{\git}{\mathbin{
		\mathchoice{\mkern-3mu /\mkern-6mu/ \mkern-3mu}% \displaystyle
		{\mkern-3mu /\mkern-6mu/ \mkern-3mu}% \textstyle
		{/\mkern-5mu/}% \scriptstyle
		{/\mkern-5mu/}}}% \scriptscriptstyle

\newcommand{\tikzcircle}[2][fill=red]{\tikz[baseline=-0.5ex]\draw[#1,radius=#2] (0,0) circle ;}%
\newcommand{\DLI}{\tikzcircle[fill=SkyBlue]{3pt}}
\newcommand{\DLII}{\tikzcircle[fill=BurntOrange]{3pt}}
\newcommand{\DLIII}{\tikzcircle[fill=Yellow]{3pt}}
\newcommand{\DLIV}{\tikzcircle[fill=YellowGreen]{3pt}}
\newcommand{\DLV}{\tikzcircle[fill=Salmon]{3pt}}
\newcommand{\DLVI}{\tikzcircle[fill=OliveGreen]{3pt}}
\newcommand{\DLVII}{\tikzcircle[fill=BlueViolet]{3pt}}
\newcommand{\DLVIII}{\tikzcircle[fill=OrangeRed]{3pt}}
\newcommand{\DLIX}{\tikzcircle[fill=Plum]{3pt}}
\newcommand{\DLX}{\tikzcircle[fill=Gray]{3pt}}

\newcommand{\ep}{\varepsilon}

\newcommand{\rhoN}{\rho_{\text{Nat}}}
\newcommand{\Spec}{\mathrm{Spec}}
\newcommand{\GHilb}{G\text{-Hilb}}

\usepackage{url}
\usepackage{hyperref}

\newcommand{\TheAuthor}{Jon Cheah}
\newcommand{\TheTitle}{Real McKay Correspondence: KR-Theory of Graded Kleinian Groups}

\newcommand{\TheModule}{\bf MA4K9 Dissertation}

\newcommand{\TheUni}{The University of Warwick}
\newcommand{\TheDept}{Mathematics Institute}

\newcommand{\TheSubDate}{\monthyear \formatdate{5}{4}{2022}}

\usepackage[numbers]{natbib}% round braces, sort multiple citations
\usepackage{hyperref}% creates hypertext links in pdf files (natbib compatible)
\usepackage{setspace}
\usepackage[nodayofweek]{datetime} % change the format of printed dates (no american style!) 
\newdateformat{monthyear}{\monthname[\THEMONTH], \THEYEAR}% new date format

\begin{document}

\pagenumbering{roman}
\begin{titlepage}
\begin{center}
\includegraphics[width=5cm]{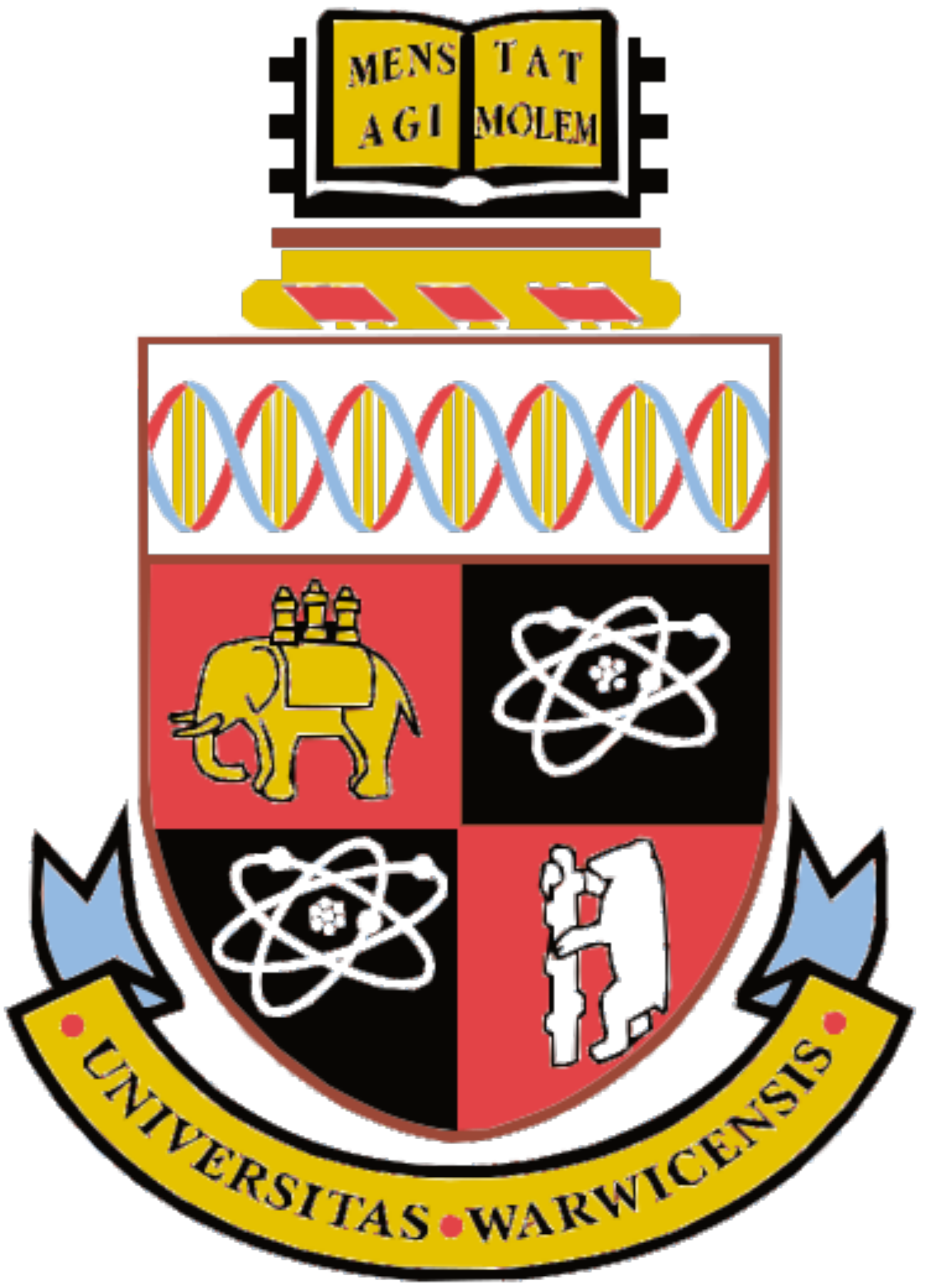} 
% this graphic file should live in the working directory
% the size of the following arguments to vspace were determined by 
% trial and error, and produce suitable output for the A4 pagesize.

\vspace*{20pt}
\begin{spacing}{1.5}
\begin{center}
{\Large \bf \TheTitle} % declared in preamble.tex

%\vspace*{10pt}

by

{\Large \bf \TheAuthor} % declared in preamble.tex

supervised by

{\Large \bf Dmitriy Rumynin}

%\vspace*{16pt}

{\large \bf \TheModule} % declared in preamble.tex

Submitted to \TheUni % declared in preamble.tex

\vspace*{36pt}
{\Large \bf \TheDept} % declared in preamble.tex

\TheSubDate % declared in preamble.tex

\vspace*{36pt}
\includegraphics[width=5cm]{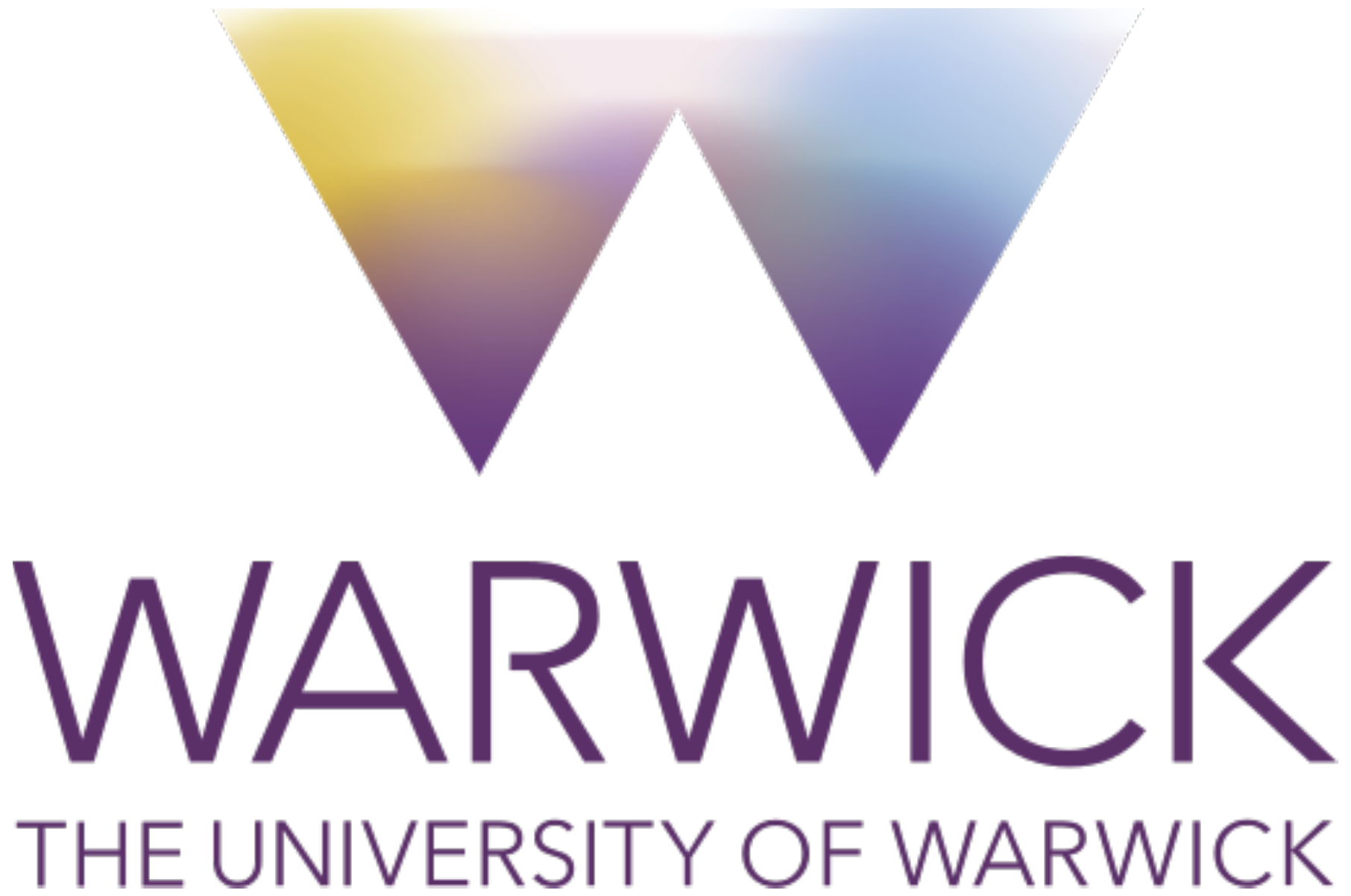} 
% this graphic file should live in the working directory

\end{center}
\end{spacing}
\end{center}
\end{titlepage}

\setcounter{page}{2}
\tableofcontents
\cleardoublepage
\onehalfspacing
\pagenumbering{arabic}
\onehalfspacing
\raggedbottom

% Now starts the main part of the report. 
% Each section/chapter is in a separate file. Two examples are include.
% 
% This part (between TOC and Bibliograpy) must obey:
% for MA4K8 must not exceed 30 pages.
% for MA4K9 the target is approx 30 pages and must not exceed 40 pages.

\section{Introduction}

This project considers the finite symmetry subgroups of the orthogonal group $\Orth(3)\subset \GL(3,\R)$ and how these can be embedded into one another. Of particular interest are the index $2$ containments, which were classified by Conway and Smith in \citep{CS}. The special orthogonal group $\SO(3) \subset \SL(3,\R)$ admits a double cover from the spinor group $\Spin(3) \cong \SU(2) \subset \SL(2,\C)$, and lifting our subgroups up preserves our network of containments. Those subgroups not contained in $\SO(3)$ are lifted to the pinor groups $\Pin_{\pm}(3)$ of which there are two choices. We then explore $KR$-theory as introduced by Atiyah \citep{Atiyah66} in 1966, which is a variant of topological $K$-theory when dealing with a topological space equipped with an involution. In the case of the index 2 containments $G\lhd \widehat{G}$, the quotient spaces $\C^2 \git G$, can be equipped by an involution via the action of $\widehat{G} / G$. In 1983, Gonzalez-Sprinberg and Verdier \citep{GSV} showed how one can view the McKay correspondence from \citep{McKay80} as an isomorphism between the $G$-equivariant $K$-theory $K_G(\C^2)$ and the $K$-theory of $\widetilde{\C^2 \git G}$, the minimal resolution of the singularity.

In section 2, we go over the finite subgroups of $\Orth(3)$ including explicit matrix generators and construct a graph of the index $2$ containments. These are then lifted by their double covers to their respective spinor or pinor groups. This lifting preserves their network of containments yielding an analogous graph. We also produce MAGMA code for each of the subgroups of $\Orth(3)$ and $\Pin_{\pm}(3)$ which can be found in \citep{joncheahMAGMA}.

In section 3, we take a brief diversion into topological $K$-theory, mainly citing \citep{HatcherVBKT} and \citep{Atiyah67}. We follow Atiyah and Segal \citep{AtiyahSegal} in constructing $G$-equivariant $K$-theory for use in section 4. In section 4, we give a brief expository overview of the McKay correspondence observed by McKay in \citep{McKay80} and the statement in terms of $K$-theory from Gonzalez-Sprigberg and Verdier \citep{GSV}, as well as an example in the explicit case of the binary dihedral group $\mathrm{BD}_{16}$.

In section 5, we use the previously constructed index 2 containments as $C_2$ graded subgroups, and calculate the Real and complex Frobenius-Schur indicators. Applying Dyson's classification of antilinear block structures (see \citep{JTDR}, \citep{Dyson}, or \citep{GM}), we produce decorated McKay graphs for each of the containments. Our MAGMA code for the calculation of the indicators can also be found in \citep{joncheahMAGMA}.

In section 6, we build $KR$-theory as was introduced in \citep{Atiyah66} and apply it to the case of our Kleinian singularities and $C_2$-graded groups. This allows us to state the final conjecture in section 7, a form of the McKay Correspondence for $KR$-theory in the case of $C_2$-graded groups.

We use the following notation throughout. Finite groups $G \lhd \widehat{G}$ give an index 2 containment. The symbol $\mathbbm{1}$ is used to represent the identity transformation, or the identity matrix of appropriate dimension. The group $C_2$ is multiplicative and might be written as $\left\{ 1,-1\right\}$, $\left\{\mathbbm{1}, -\mathbbm{1}\right\}$, or $\left\{ 1,\bold{x}\right\}$.

\pagebreak

\section{Finite Subgroups}

\subsection{Finite Subgroups of $\Orth(3)$}

We begin with a classical result.

\begin{thm}\label{thm:GLsubgroupsConjugateO3}
	Every finite subgroup of $\GL(3,\R)$ is conjugate to a finite subgroup of $\Orth(3)$.
\end{thm}
\begin{proof}
	Let $G\subset\GL(3,\R)$ be finite, and $\Span{\cdot,\cdot}$ be the usual inner product on $\R^3$. We construct $\Span{\cdot,\cdot}_G:\R^3\times\R^3\to\R$ by,	
	\begin{equation*}
		\Span{u,v}_G := \frac{1}{\left|G\right|} \sum_{g\in G} \Span{g \cdot u, g\cdot v}.
	\end{equation*}
	This is well defined as $G$ is finite, and we note that $\Span{\cdot,\cdot}_G$ inherits symmetry and bilinearity from $\Span{\cdot,\cdot}$. Furthermore, for any $u\in \R^3 \setminus \{0\}$
	\begin{equation*}
		\Span{u,u}_G = \frac{1}{\left|G\right|} \sum_{g\in G} \Span{g \cdot u,g \cdot u} > 0 
	\end{equation*}	
	and for any $h\in G$,
	\begin{equation*}
		\Span{h\cdot u, h\cdot v}_G = \frac{1}{\left|G\right|} \sum_{g\in G} \Span{g h\cdot u,g h\cdot v} = \frac{1}{\left|G\right|} \sum_{gh\in G} \Span{ g h\cdot u, h\cdot v} = \Span{ u,  v}_G,
	\end{equation*}
	so $\Span{\cdot,\cdot}_G$ is positive definite and $G$-invariant. Thus up to conjugation (or a change of coordinates in $\R^3$), $G$ is a subgroup of $\Orth(3)$.
	
\end{proof}

The finite subgroups of $\Orth(3)$ are well studied and can be found in \citep{CS} or \citep{FarleyOrtiz}. Up to conjugacy, there are 14 finite subgroups of $\Orth(3)$, and we list these in Table \ref{Table:O3Subgroups}. These are be split into 7 infinite families which each leave an axis of $\R^3$ invariant, and 7 sporadic groups which only leave the origin invariant. The latter manifest themselves as symmetry groups of polyhedra centred on the origin, whereas the former are the symmetry groups of various prisms/antiprisms \citep{CS}. Following the work of these authors, we have reproduced these groups in Magma code with explicit generators.

\renewcommand{\arraystretch}{1.5}
\begin{table}[!h]
	\centering
	\setlength\tabcolsep{3.5pt}
	\begin{tabular}{|cc|cc|}
		\hline
		\multicolumn{2}{|c|}{Axial groups}                                                  & \multicolumn{2}{c|}{Polyhedral groups}                                           \\ \hline
		\multicolumn{1}{|c|}{Cyclic }           & $C_n= \Span{A}$                      & \multicolumn{1}{c|}{Tetrahedral }       & $T_{12}=\Span{I^2, Y}$            \\ \hline
		\multicolumn{1}{|c|}{Diplo-cyclic }     & $2C_{2n} = \Span{A, -\mathbbm{1}}$ & \multicolumn{1}{c|}{Diplo-tetrahedral } & $2T_{24}=\Span{I^2,Y,-\mathbbm{1}}$ \\ \hline
		\multicolumn{1}{|c|}{Cyclo-cyclic }     & $CC_{2n}=\Span{-A'}$                  & \multicolumn{1}{c|}{Octahedral }        & $O_{24}=\Span{I,Y}$               \\ \hline
		\multicolumn{1}{|c|}{Dihedral }         & $D_{2n}=\Span{A,B}$                  & \multicolumn{1}{c|}{Tetra-octahedral }  & $TO_{24}=\Span{-I,Y}$             \\ \hline
		\multicolumn{1}{|c|}{Cyclo-dihedral }   & $CD_{2n}=\Span{A,-B}$                & \multicolumn{1}{c|}{Diplo-octahedral }  & $2O_{48}=\Span{I,Y, -\mathbbm{1}}$ \\ \hline
		\multicolumn{1}{|c|}{Dihedro-dihedral } & $DD_{4n}=\Span{-A',B}$               & \multicolumn{1}{c|}{Icosahedral }       & $I_{60}=\Span{X,Z}$               \\ \hline
		\multicolumn{1}{|c|}{Diplo-Dihedral }   & $2D_{8n}=\Span{A,B, -\mathbbm{1}}$    & \multicolumn{1}{c|}{Diplo-icosahedral } & $2I_{120}=\Span{X,Z,-\mathbbm{1}}$ \\ \hline
	\end{tabular}
	
	\caption{The finite subgroups of $\Orth(3)$ and their matrix generators. The axial groups are generated by rotation matrices $A=\begin{psmallmatrix}
			\cos{\frac{2\pi}{n}} & -\sin{\frac{2\pi}{n}} & 0 \\
			\sin{\frac{2\pi}{n}} & \cos{\frac{2\pi}{n}} & 0 \\
			0 & 0 & 1
		\end{psmallmatrix}$, $A'=\begin{psmallmatrix}
			\cos{\frac{\pi}{n}} & -\sin{\frac{\pi}{n}} & 0 \\
			\sin{\frac{\pi}{2n}} & \cos{\frac{\pi}{n}} & 0 \\
			0 & 0 & 1
		\end{psmallmatrix}$, and $B=\begin{psmallmatrix}
			-1 & 0 & 0 \\
			0 & 1 & 0 \\
			0 & 0 & -1
		\end{psmallmatrix}$. The polyhedral groups are generated by rotation matrices $Y=\begin{psmallmatrix}
			0 & 0 & 1 \\
			1 & 0 & 0 \\
			0 & 1 & 0
		\end{psmallmatrix}$, $I=\begin{psmallmatrix}
			1 & 0 & 0 \\
			0 & 0 & 1 \\
			0 & -1 & 0
		\end{psmallmatrix}$, $X=\frac{1}{\sqrt{5}}\begin{psmallmatrix}
			-\sqrt{5} & 0 & 0 \\
			0 &-{1} & 2 \\
			0 & 2 & 1
		\end{psmallmatrix}$, and $Z=\begin{psmallmatrix}
			\cos{\frac{2\pi}{5}} & -\sin{\frac{2\pi}{5}} & 0 \\
			\sin{\frac{2\pi}{5}} & \cos{\frac{2\pi}{5}} & 0 \\
			0 & 0 & 1
		\end{psmallmatrix}$.}
	
	\label{Table:O3Subgroups}
\end{table}
\renewcommand{\arraystretch}{1}

The groups $C_n, D_{2n}, T_{12}, O_{24}, I_{60}$ have all of their elements with determinant equal to $1$, and are precisely the orientation preserving groups contained in the special orthogonal group $\SO(3)$.

To use Conway and Smith's terminology, the four groups $CC_{2n}$, $CD_{4n}$, $DD_{4n}$ and $TO_{24}$ are called \emph{hybrid groups}. Given an index 2 containment of groups $H\lhd G\subset \Orth(3)$, we let the set $HG$ consist of all $h\in H$, and $-g$ for all $g\in G \setminus H$, and make this into a group by the standard composition of orthogonal transformations.

The groups with labels of the form $2G$ are the direct sums of a group $G\subset \SO(3)$ and the inversion group $\{\pm\mathbbm{1}\}$ where $\mathbbm{1}$ is the identity transformation. The inversion $-\mathbbm{1}$ is central in $\Orth(3)$. Conway and Smith call these diploid groups and denote these by $\pm G$. We avoid this for the sake of cleaner notation when lift up  to their covering groups. The \emph{diplo-} groups $2G$ should not be confused with the \emph{binary} groups $\mathrm{BG}$ to be introduced in the next section.

The index 2 containments are given by Conway and Smith in \citep{CS}, and we recreate their graph of containments in Figures  \ref{fig:O3Polyhedral} and \ref{fig:O3Axial}.

\begin{figure}[p]
	\centering
	
	% https://q.uiver.app/?q=WzAsNyxbMiwzLCJUX3sxMn0iXSxbMSwyLCJPX3syNH0iXSxbMCwyLCJUT197MjR9Il0sWzQsMiwiSV97NjB9Il0sWzIsMSwiXFxwbSBUX3sxMn0iXSxbMCwwLCJcXHBtIE9fezI0fSJdLFs0LDAsIlxccG0gSV97NjB9Il0sWzAsMV0sWzAsMl0sWzIsMSwiXFxjb25nIiwxLHsic3R5bGUiOnsiYm9keSI6eyJuYW1lIjoibm9uZSJ9LCJoZWFkIjp7Im5hbWUiOiJub25lIn19fV0sWzAsMywiNSIsMl0sWzAsNF0sWzQsNV0sWzIsNV0sWzEsNV0sWzMsNl0sWzQsNiwiNSIsMl1d
	\[\begin{tikzcd}
		{2O_{48}} &&&& {2I_{120}} \\
		&& {2T_{24}} \\
		{TO_{24}} & {O_{24}} &&& {I_{60}} \\
		&& {T_{12}}
		\arrow[from=4-3, to=3-2]
		\arrow[from=4-3, to=3-1]
		\arrow["\cong"{description}, draw=none, from=3-1, to=3-2]
		\arrow["5"', from=4-3, to=3-5]
		\arrow[from=4-3, to=2-3]
		\arrow[from=2-3, to=1-1]
		\arrow[from=3-1, to=1-1]
		\arrow[from=3-2, to=1-1]
		\arrow[from=3-5, to=1-5]
		\arrow["5"', from=2-3, to=1-5]
	\end{tikzcd}\]
	\caption{Containments between the polyhedral subgroups of $\Orth(3)$. The index is 2 when not labelled.}
	\label{fig:O3Polyhedral}
\end{figure}
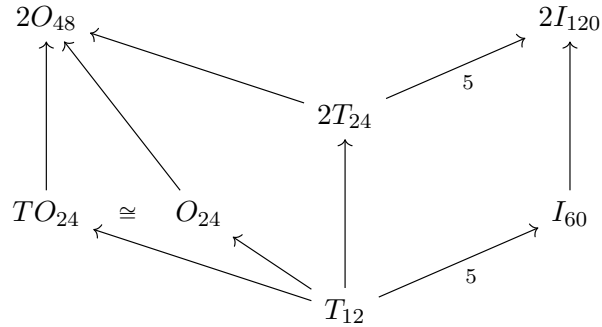

% https://q.uiver.app/?q=WzAsMTQsWzIsNCwiQ19uIl0sWzEsMiwiQ0RfezJufSJdLFswLDIsIkRfezJufSJdLFszLDIsIkNDX3sybn0iXSxbNCwyLCJcXHBtIENfbiJdLFs0LDMsIlxccG0gQ19tIl0sWzMsMywiQ0NfezJtfSJdLFsyLDUsIkNfbSJdLFswLDMsIkRfezJtfSJdLFsxLDMsIkNEX3sybX0iXSxbMSwwLCJERF97NG59Il0sWzMsMCwiXFxwbSBEX3sybn0iXSxbMywxLCJcXHBtIERfezJtfSJdLFsyLDEsIkREX3s0bX0iXSxbMCwxLCIiLDAseyJjdXJ2ZSI6MX1dLFswLDIsIiIsMix7Im9mZnNldCI6MSwiY3VydmUiOi0zfV0sWzEsMiwiXFxjb25nIiwxLHsic3R5bGUiOnsiYm9keSI6eyJuYW1lIjoibm9uZSJ9LCJoZWFkIjp7Im5hbWUiOiJub25lIn19fV0sWzAsMywiIiwyLHsiY3VydmUiOi0xfV0sWzUsNF0sWzYsNCwiIiwwLHsiY3VydmUiOi0xfV0sWzcsMF0sWzgsMiwiIiwwLHsib2Zmc2V0IjotMX1dLFs5LDEsIiIsMCx7Im9mZnNldCI6LTF9XSxbMiwxMCwiIiwwLHsiY3VydmUiOi0xfV0sWzEsMTBdLFs0LDExLCIiLDAseyJjdXJ2ZSI6MX1dLFsxLDExLCIiLDAseyJjdXJ2ZSI6NH1dLFsyLDExLCIiLDAseyJjdXJ2ZSI6LTJ9XSxbMTIsMTFdLFs4LDIsIiIsMCx7Im9mZnNldCI6MX1dLFsxMywxMSwiIiwwLHsib2Zmc2V0IjotMX1dLFszLDEwLCIiLDAseyJjdXJ2ZSI6LTR9XSxbMTMsMTEsIiIsMCx7Im9mZnNldCI6MX1dLFs5LDEsIiIsMCx7Im9mZnNldCI6MX1dLFswLDQsIiIsMix7Im9mZnNldCI6LTEsImN1cnZlIjozfV1d

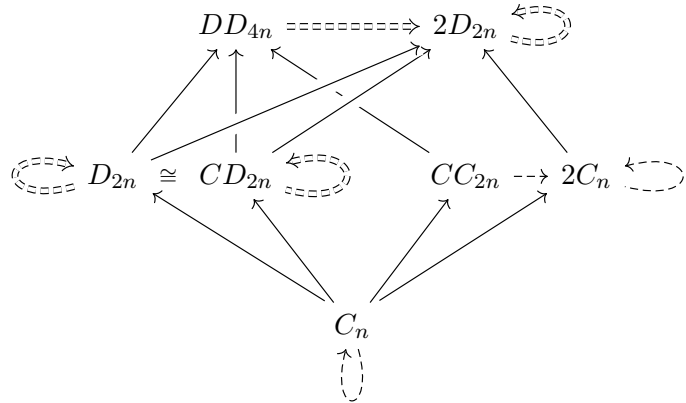
\begin{figure}[p]
	\centering
	\def\svgwidth{\columnwidth}
	% https://q.uiver.app/?q=WzAsMTQsWzIsNCwiQ19uIl0sWzEsMiwiQ0RfezJufSJdLFswLDIsIkRfezJufSJdLFszLDIsIkNDX3sybn0iXSxbNCwyLCJcXHBtIENfbiJdLFs0LDMsIlxccG0gQ19tIl0sWzMsMywiQ0NfezJtfSJdLFsyLDUsIkNfbSJdLFswLDMsIkRfezJtfSJdLFsxLDMsIkNEX3sybX0iXSxbMSwwLCJERF97NG59Il0sWzMsMCwiXFxwbSBEX3sybn0iXSxbMywxLCJcXHBtIERfezJtfSJdLFsyLDEsIkREX3s0bX0iXSxbMCwxLCIiLDAseyJjdXJ2ZSI6MX1dLFswLDIsIiIsMix7Im9mZnNldCI6MSwiY3VydmUiOi0zfV0sWzEsMiwiXFxjb25nIiwxLHsic3R5bGUiOnsiYm9keSI6eyJuYW1lIjoibm9uZSJ9LCJoZWFkIjp7Im5hbWUiOiJub25lIn19fV0sWzAsMywiIiwyLHsiY3VydmUiOi0xfV0sWzUsNF0sWzYsNCwiIiwwLHsiY3VydmUiOi0xfV0sWzcsMF0sWzgsMiwiIiwwLHsib2Zmc2V0IjotMX1dLFs5LDEsIiIsMCx7Im9mZnNldCI6LTF9XSxbMiwxMCwiIiwwLHsiY3VydmUiOi0xfV0sWzEsMTBdLFs0LDExLCIiLDAseyJjdXJ2ZSI6MX1dLFsxLDExLCIiLDAseyJjdXJ2ZSI6NH1dLFsyLDExLCIiLDAseyJjdXJ2ZSI6LTJ9XSxbMTIsMTFdLFs4LDIsIiIsMCx7Im9mZnNldCI6MX1dLFsxMywxMSwiIiwwLHsib2Zmc2V0IjotMX1dLFszLDEwLCIiLDAseyJjdXJ2ZSI6LTR9XSxbMTMsMTEsIiIsMCx7Im9mZnNldCI6MX1dLFs5LDEsIiIsMCx7Im9mZnNldCI6MX1dLFswLDQsIiIsMix7Im9mZnNldCI6LTEsImN1cnZlIjozfV1d
	\[\begin{tikzcd}[column sep=small]
		& {DD_{4n}} && {2D_{2n}} \\
		&&  		 &  \\
		{D_{2n}} & {CD_{2n}} && {CC_{2n}} & {2C_n} \\
		& 			 && 		  &  \\
		&& {C_n} \\
		\arrow[from=5-3, to=3-2]
		\arrow[shift right=1,from=5-3, to=3-1]
		\arrow["\cong"{description}, draw=none, from=3-2, to=3-1]
		\arrow[from=5-3, to=3-4]
		\arrow[from=3-4, to=3-5, dashrightarrow]
		\arrow[from=3-1, to=1-2]
		\arrow[from=3-2, to=1-2]
		\arrow[from=3-5, to=1-4]
		\arrow[from=3-4, to=1-2]
		\arrow[Rightarrow, from=1-2, to=1-4, dashed]
		\arrow[shift left=1, from=5-3, to=3-5]
		\arrow[from=3-2, to=1-4, crossing over]
		\arrow[from=3-1, to=1-4, crossing over]
		\arrow[loop left, from=3-1, to=3-1, Rightarrow, dashed, looseness=10]
		\arrow[loop right, from=3-2, Leftarrow, dashed]
		\arrow[loop below, from=5-3, to=5-3, dashed, looseness=15]
		\arrow[loop right, from=3-5, to=3-5, dashed, leftarrow, looseness=10]
		\arrow[loop right, from=1-4, dashed, Leftarrow]
	\end{tikzcd}\]
	\caption{Index 2 containments between the axial subgroups of $\Orth(3)$. Dashed arrows indicate maps from groups indexed by $m = n/2$ to groups indexed by $n$ exist in the case that $n$ is even. Double lines between groups indicate there are two different ways by which the smaller group is contained in the larger.}
	\label{fig:O3Axial}
\end{figure}

	\pagebreak
	
	\subsection{Finite Subgroups of $\SU(2,\C)$}
	
	In a similar argument to the proof of theorem \ref{thm:GLsubgroupsConjugateO3}, one can show that every finite subgroup of $\SL(2,\C)$, gives leaves invariant a particular hermitian inner product. From this, we conclude that our subgroup must be conjugate to some finite subgroup $G$ of the special unitary group $\SU(2,\C)$. The special unitary group admits a double covering over $\SO(3,\R)$, and is hence isomorphic to $\Spin(3)$ the double cover of $\SO(3,\R)$.

	The finite subgroups of $\SU(2)$ were calculated by Klein in \citep{FK} and are precisely
	\begin{align*}
		C_n, \quad 2C_{2n} \text{ (for n odd)}, \quad \mathrm{BD}_{4n}, \quad \mathrm{BT}_{24}, \quad \mathrm{BO}_{48}, \quad \mathrm{BI}_{120}.
	\end{align*}

	The last 4 are the binary dihedral, tetrahedral, binary octahedral, and binary icosahedral groups which are respectively the central extensions of the dihedral, tetrahedral, octahedral, and icosahedral subgroups of $\SO(3)$ by $C_2 =\left\{\pm\mathbbm{1}\right\}$. 
	Cyclic groups $C_n=C_{2m} \subset \SO(3)$ of even order are lifted to a cyclic group of order $2n$, which we will denote $\mathrm{BC}_{2n} \subset \SU(2)$ to help distinguish them from groups in $\SO(3)$. However, when $n$ is odd, the group $C_n = \Span{A\mid A^n=\mathbbm{1}} \subset \SO(3)$ can either be lifted to an isomorphic cyclic group or to a direct product $C_n\times C_2\cong 2C_{2n}$ isomorphic to a diplo-cyclic group.
	For a more detailed explanation of this classification see \citep{IVD}, \citep{JCS} or \citep{joncheahSU2}.
	
	\renewcommand{\arraystretch}{1.5}
	\begin{table}[!h]
		\centering
		\begin{tabular}{|c|c|}
			\hline
			Group                  & Matrix generators                   \\ \hline
			Cyclic group           & $C_n= \Span{A}$                     \\ \hline
			Diplo-cyclic ($n$ odd) & $2C_{2n} = \Span{A, -\mathbbm{1}}$ \\ \hline
			Binary dihedral        & $\mathrm{BD}_{4n}=\Span{A,B}$       \\ \hline
			Binary tetrahedral     & $\mathrm{BT}_{24}=\Span{I^2,Y}$     \\ \hline
			Binary octahedral      & $\mathrm{BO}_{48}=\Span{I,Y}$       \\ \hline
			Binary Icosahedral     & $\mathrm{BI}_{120}=\Span{X,Z}$      \\ \hline
		\end{tabular}
		
		\caption{The finite subgroups of $\Spin(3)\simeq \SU(2)$ and their matrix generators. The axial groups are generated by matrices $A=\begin{psmallmatrix}
				\ep^1 &  0 \\
				0 & \ep^{-1}
			\end{psmallmatrix}$, where $\ep$ is a primitive $n$-th root of unity, and $B=\begin{psmallmatrix}
				0 & 1 \\
				-1 & 0
			\end{psmallmatrix}$. The polyhedral groups are generated by matrices $Y=\frac{1}{2}\begin{psmallmatrix}
				-1+i & 1+i\\
				-1+i &-1-i
			\end{psmallmatrix}$, $I=\begin{psmallmatrix}
				i & 0 \\
				0 & -i
			\end{psmallmatrix}$, $X=\frac{1}{\sqrt{5}}\begin{psmallmatrix}
				-\mu+\mu^4 & \mu^2-\mu^3 \\
				\mu^2-\mu^3 & \mu-\mu^4
			\end{psmallmatrix}$, and $Z=\begin{psmallmatrix}
				-\mu^3 & 0 \\
				0 & -\mu^2
			\end{psmallmatrix}$, where $\mu$ is a primitive $5$th root of unity.}
		
		\label{Table:SU2Subgroups}
	\end{table}
	\renewcommand{\arraystretch}{1}
	
	We note that apart from the cyclic groups of odd order, all the finite subgroups $G \subset \SL(2,\C)$ are of even order and contain $-\mathbbm{1}$, the unique central element of order $2$.
	
	\begin{figure}[!hbtp]
		\centering
		\def\svgwidth{\columnwidth}
		% https://q.uiver.app/?q=WzAsOCxbMiwzLCJcXG1hdGhybXtCVH1fezI0fSJdLFswLDAsIlxcbWF0aHJte0JPfV97NDh9Il0sWzUsMCwiXFxtYXRocm17QlR9X3sxMjB9Il0sWzcsMiwiXFxtYXRocm17QkN9X3sybn0iXSxbNywwLCJcXG1hdGhybXtCRH1fezRufSJdLFs4LDMsIkNfbiJdLFs5LDIsIjJDX24iXSxbMTAsMF0sWzAsMV0sWzAsMiwiNSIsMl0sWzMsNF0sWzUsM10sWzUsNiwibiBcXHRleHR7IG9kZH0iLDJdLFs1LDNdXQ==
		\[\begin{tikzcd}[column sep=small]
			{\mathrm{BO}_{48}} &&&&& {\mathrm{BT}_{120}} && {\mathrm{BD}_{4n}} \\
			\\
			&&&&&&&\mathrm{BC}_{2n}& {2C_{2n}}\\
			&& {\mathrm{BT}_{24}} &&&&& {C_{n}}
			\arrow[from=4-3, to=1-1]
			\arrow["5"', from=4-3, to=1-6]
			\arrow[from=3-8, to=1-8]
			\arrow[from=4-8, to=3-8]
			\arrow[from=4-8, to=3-9,"n \text{ odd}"']
			\arrow[from=3-9, to=1-8,"n \text{ odd}"']
			\arrow[loop right, from=1-8, Leftarrow, dashed]
			\arrow[loop right, from=3-9, to=3-9, leftarrow, dashed, looseness=10]
		\end{tikzcd}\]
		\caption{Containments between the finite subgroups of $\SU(2)=\Spin(3)$. The index is 2 when not labelled.}
		\label{fig:Spin3Subgroups}
	\end{figure}
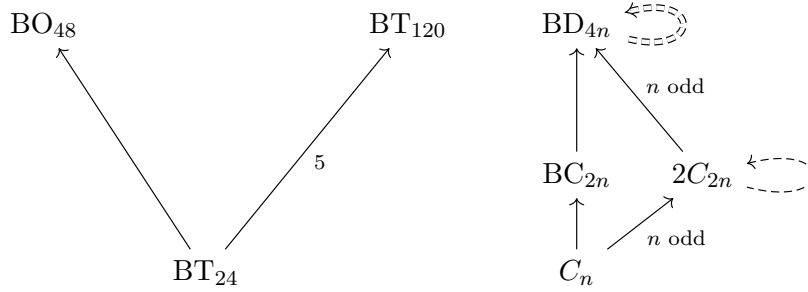

	\begin{defn}\label{def:McKayGraph}
		Fix $G$ a finite subgroup of $\SU(2)$. Let $\rho_{\text{Nat}} $ denote the natural representation arising from the inclusion $G \to \SU(2) \subset \GL(2,\C)$, and let $\mathrm{Irr}(G) := \left\{\rho_1, \dots  ,\rho_n\right\}$ denote the set of irreducible complex representations. For each $i$, consider the representation $\rho_{\text{Nat}} \otimes \rho_j$. As $G$ is reductive, there exist $m_{ij}$ such that,
		\begin{equation*}
			\rho_{\text{Nat}} \otimes \rho_j = \bigoplus_{i=1}^n m_{ij}\rho_j.
		\end{equation*}
		
		We construct the \emph{McKay graph} or \emph{McKay quiver} of $G$ to be the graph whose vertex set is $ \mathrm{Irr}(G) := \left\{\rho_1, \dots  ,\rho_n\right\}$, and has $m_{ij}$ directed edges from $\rho_i$ to $\rho_j$.
	\end{defn}
	Note that $m_{ij}$ is equal to the dimension of $\mathrm{Hom}_{\C G}(\rho_i,\rho_{\text{Nat}})$. When we have a pair of directed edges, $\rho_i$ to $\rho_j$ and $\rho_j$ to $\rho_i$, we instead draw a single undirected edge.

	\begin{eg}
		Let $G=\mathrm{BT}_{24}$. We calculate the character table of $G$ using Magma (see file \verb!Pin3--BT24.txt! \citep{joncheahMAGMA}). This could also be done by hand by considering the character table of $G /Z(G) = \mathrm{BT}_{24} / \left\{\pm \mathbbm{1}\right\} \cong T_{12} \cong A_4$.
		\begin{table}[!h]
			\centering
			\begin{tabular}{|c | c c c c c c c|}
				\hline
				Class	 & 1 & 2 & 3 & 4 & 5 & 6 & 7 \\
				Size	 & 1 & 1 & 4 & 4 & 6 & 4 & 4 \\
				Order	 & 1 & 2 & 3 & 3 & 4 & 6 & 6 \\
				\hline
				$\chi_1$ & 1 & 1 & 1 & 1 & 1 & 1 & 1 \\
				$\chi_2$ & 1 & 1 & $\omega$ & $\omega^2$ & 1 & $\omega^2$ & $\omega$ \\
				$\chi_3$ & 1 & 1 & $\omega^2$ & $\omega$ & 1 & $\omega$ & $\omega^2$ \\
				$\chi_4$ & 2 & -2 & -1 & -1 & 0 & 1 & 1 \\
				$\chi_5$ & 2 & -2 & $-\omega^2$ & $-\omega$ & 0 & $\omega$ & $\omega^2$ \\
				$\chi_6$ & 2 & -2 & $-\omega$ & $-\omega^2$ & 0 & $\omega^2$ & $\omega$ \\
				$\chi_7$ & 3 & 3 & 0 & 0 & 1 & 0 & 0 \\
				\hline
			\end{tabular}
			\caption{Character table of the binary tetrahedral group $\mathrm{BT}_{24}$. Here, $\omega$ is a primitive third root of unity.}
			\label{Table:BT24chartable}
		\end{table}
		
		By considering the traces of the explicit matrix generators of $\mathrm{BT}_{24}$, one can verify that $\rhoN = \rho_4$. One can then check that,
		\begin{align*}
			\chi_4 \cdot \chi_1 &= \chi_4 \\
			\chi_4 \cdot \chi_2 &= \chi_6 \\
			\chi_4 \cdot \chi_3 &= \chi_5 \\
			\chi_4 \cdot \chi_4 &= \chi_1 + \chi_7 \\
			\chi_4 \cdot \chi_5 &= \chi_3 + \chi_7 \\
			\chi_4 \cdot \chi_6 &= \chi_2 + \chi_7 \\
			\chi_4 \cdot \chi_7 &= \chi_4 + \chi_5 + \chi_6
		\end{align*}
		which allows us to draw the McKay graph of $\mathrm{BT}_{24}$.
		
		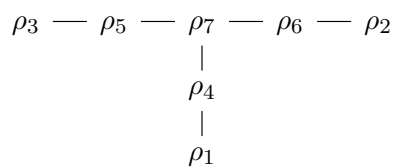
\begin{figure}[!h]
			\centering
			\def\svgwidth{\columnwidth}
			\[\begin{tikzcd}[column sep=small, row sep=small]
				\rho_3 & \rho_5 & \rho_7 & \rho_6 & \rho_2 \\
				&& \rho_4 \\
				&& \rho_1
				\arrow[no head, from=1-5, to=1-4]
				\arrow[no head, from=1-2, to=1-1]
				\arrow[no head, from=1-2, to=1-3]
				\arrow[no head, from=1-3, to=1-4]
				\arrow[no head, from=1-3, to=2-3]
				\arrow[no head, from=2-3, to=3-3]
			\end{tikzcd}\]
			\caption{The McKay graph of $\mathrm{BT}_{24}$ with $\rho_{\text{Nat}}=\rho_4$.}
			\label{fig:McKayGraphBT24}
		\end{figure}

	\end{eg}

	\begin{figure}[p]
		%	\centering
		\def\svgwidth{\columnwidth}
		$C_n$
		\vspace*{-1cm}
		\[\begin{tikzcd}[column sep=tiny, row sep=small]
			&&&& \times \\
			\bullet & \circ & \circ & \circ & \circ & \circ & \circ & \circ & \bullet
			\arrow[no head, from=2-5, to=2-4]
			\arrow[no head, from=2-2, to=2-1]
			\arrow[no head, from=2-5, to=2-6]
			\arrow[no head, from=2-8, to=2-9]
			\arrow[shift left=1, no head, from=2-1, to=1-5]
			\arrow[shift left=1, no head, from=1-5, to=2-9]
			\arrow[no head, from=2-8, to=2-7]
			\arrow[no head, from=2-3, to=2-2]
			\arrow["\cdots"{description}, no head, from=2-7, to=2-6]
			\arrow["\cdots"{description}, no head, from=2-3, to=2-4]
		\end{tikzcd}\]
		
		%	% https://q.uiver.app/?q=WzAsMTEsWzQsMSwiXFxjaXJjIl0sWzIsMSwiXFxjaXJjIl0sWzEsMSwiXFxjaXJjIl0sWzAsMSwiXFxidWxsZXQiXSxbNiwxLCJcXGNpcmMiXSxbNywxLCJcXGNkb3RzIl0sWzksMSwiXFxjaXJjIl0sWzEwLDEsIlxcYnVsbGV0Il0sWzUsMCwiXFx0aW1lcyJdLFs4LDEsIlxcY2lyYyJdLFszLDEsIlxcY2RvdHMiXSxbMiwzLCIiLDAseyJzdHlsZSI6eyJoZWFkIjp7Im5hbWUiOiJub25lIn19fV0sWzQsNSwiIiwwLHsic3R5bGUiOnsiaGVhZCI6eyJuYW1lIjoibm9uZSJ9fX1dLFs2LDcsIiIsMCx7InN0eWxlIjp7ImhlYWQiOnsibmFtZSI6Im5vbmUifX19XSxbMyw4LCIiLDAseyJvZmZzZXQiOi0yLCJzdHlsZSI6eyJoZWFkIjp7Im5hbWUiOiJub25lIn19fV0sWzgsNywiIiwxLHsib2Zmc2V0IjotMiwic3R5bGUiOnsiaGVhZCI6eyJuYW1lIjoibm9uZSJ9fX1dLFs2LDksIiIsMCx7InN0eWxlIjp7ImhlYWQiOnsibmFtZSI6Im5vbmUifX19XSxbOSw1LCIiLDEseyJzdHlsZSI6eyJoZWFkIjp7Im5hbWUiOiJub25lIn19fV0sWzEsMTAsIiIsMix7InN0eWxlIjp7ImhlYWQiOnsibmFtZSI6Im5vbmUifX19XSxbMTAsMCwiIiwyLHsic3R5bGUiOnsiaGVhZCI6eyJuYW1lIjoibm9uZSJ9fX1dLFsxLDIsIiIsMCx7InN0eWxlIjp7ImhlYWQiOnsibmFtZSI6Im5vbmUifX19XSxbMCw0LCIiLDEseyJzdHlsZSI6eyJoZWFkIjp7Im5hbWUiOiJub25lIn19fV1d
		%	\[\begin{tikzcd}[column sep=tiny, row sep=small]
			%		&&&&& \times \\
			%		\bullet & \circ & \circ & \cdots & \circ && \circ & \cdots & \circ & \circ & \bullet
			%		\arrow[no head, from=2-2, to=2-1]
			%		\arrow[no head, from=2-7, to=2-8]
			%		\arrow[no head, from=2-10, to=2-11]
			%		\arrow[shift left=2, no head, from=2-1, to=1-6]
			%		\arrow[shift left=2, no head, from=1-6, to=2-11]
			%		\arrow[no head, from=2-10, to=2-9]
			%		\arrow[no head, from=2-9, to=2-8]
			%		\arrow[no head, from=2-3, to=2-4]
			%		\arrow[no head, from=2-4, to=2-5]
			%		\arrow[no head, from=2-3, to=2-2]
			%		\arrow[no head, from=2-5, to=2-7]
			%	\end{tikzcd}\]
		
		\vspace*{1cm}
		$\mathrm{BD}_{4n}$
		\vspace*{-1cm}
		% https://q.uiver.app/?q=WzAsMTEsWzYsMSwiXFxjaXJjIl0sWzUsMSwiXFxjaXJjIl0sWzMsMSwiXFxjaXJjIl0sWzIsMSwiXFxjaXJjIl0sWzEsMSwiXFxidWxsZXQiXSxbNywxLCJcXGNpcmMiXSxbNCwxLCJcXGNkb3RzIl0sWzgsMCwiXFxjaXJjIl0sWzgsMiwiXFxjaXJjIl0sWzAsMCwiXFx0aW1lcyJdLFswLDIsIlxcY2lyYyJdLFswLDEsIiIsMCx7InN0eWxlIjp7ImhlYWQiOnsibmFtZSI6Im5vbmUifX19XSxbMiwzLCIiLDAseyJzdHlsZSI6eyJoZWFkIjp7Im5hbWUiOiJub25lIn19fV0sWzMsNCwiIiwwLHsic3R5bGUiOnsiaGVhZCI6eyJuYW1lIjoibm9uZSJ9fX1dLFswLDUsIiIsMCx7InN0eWxlIjp7ImhlYWQiOnsibmFtZSI6Im5vbmUifX19XSxbMiw2LCIiLDIseyJzdHlsZSI6eyJoZWFkIjp7Im5hbWUiOiJub25lIn19fV0sWzYsMSwiIiwyLHsic3R5bGUiOnsiaGVhZCI6eyJuYW1lIjoibm9uZSJ9fX1dLFs1LDcsIiIsMCx7InN0eWxlIjp7ImhlYWQiOnsibmFtZSI6Im5vbmUifX19XSxbNSw4LCIiLDAseyJzdHlsZSI6eyJoZWFkIjp7Im5hbWUiOiJub25lIn19fV0sWzQsOSwiIiwwLHsic3R5bGUiOnsiaGVhZCI6eyJuYW1lIjoibm9uZSJ9fX1dLFs0LDEwLCIiLDAseyJzdHlsZSI6eyJoZWFkIjp7Im5hbWUiOiJub25lIn19fV1d
		\[\begin{tikzcd}[column sep=tiny, row sep=small]
			\times &&&&&&&& \circ \\
			& \bullet & \circ & \circ & \cdots & \circ & \circ & \circ \\
			\circ &&&&&&&& \circ
			\arrow[no head, from=2-7, to=2-6]
			\arrow[no head, from=2-4, to=2-3]
			\arrow[no head, from=2-3, to=2-2]
			\arrow[no head, from=2-7, to=2-8]
			\arrow[no head, from=2-4, to=2-5]
			\arrow[no head, from=2-5, to=2-6]
			\arrow[no head, from=2-8, to=1-9]
			\arrow[no head, from=2-8, to=3-9]
			\arrow[no head, from=2-2, to=1-1]
			\arrow[no head, from=2-2, to=3-1]
		\end{tikzcd}\]
		
		\vspace*{1cm}
		$\mathrm{BT}_{24}$
		\vspace*{-0.5cm}
		% https://q.uiver.app/?q=WzAsNyxbNCwwLCJcXGNpcmMiXSxbMywwLCJcXGNpcmMiXSxbMSwwLCJcXGNpcmMiXSxbMCwwLCJcXGNpcmMiXSxbMiwwLCJcXGNpcmMiXSxbMiwxLCJcXGJ1bGxldCJdLFsyLDIsIlxcdGltZXMiXSxbMCwxLCIiLDAseyJzdHlsZSI6eyJoZWFkIjp7Im5hbWUiOiJub25lIn19fV0sWzIsMywiIiwwLHsic3R5bGUiOnsiaGVhZCI6eyJuYW1lIjoibm9uZSJ9fX1dLFsyLDQsIiIsMix7InN0eWxlIjp7ImhlYWQiOnsibmFtZSI6Im5vbmUifX19XSxbNCwxLCIiLDIseyJzdHlsZSI6eyJoZWFkIjp7Im5hbWUiOiJub25lIn19fV0sWzQsNSwiIiwwLHsic3R5bGUiOnsiaGVhZCI6eyJuYW1lIjoibm9uZSJ9fX1dLFs1LDYsIiIsMCx7InN0eWxlIjp7ImhlYWQiOnsibmFtZSI6Im5vbmUifX19XV0=
		\[\begin{tikzcd}[column sep=small, row sep=small]
			\circ & \circ & \circ & \circ & \circ \\
			&& \bullet \\
			&& \times
			\arrow[no head, from=1-5, to=1-4]
			\arrow[no head, from=1-2, to=1-1]
			\arrow[no head, from=1-2, to=1-3]
			\arrow[no head, from=1-3, to=1-4]
			\arrow[no head, from=1-3, to=2-3]
			\arrow[no head, from=2-3, to=3-3]
		\end{tikzcd}\]
		
		\vspace*{1cm}
		$\mathrm{BO}_{48}$
		\vspace*{-0.5cm}	
		% https://q.uiver.app/?q=WzAsOCxbNiwwLCJcXGNpcmMiXSxbNCwwLCJcXGNpcmMiXSxbMywwLCJcXGNpcmMiXSxbNSwwLCJcXGNpcmMiXSxbMywxLCJcXGNpcmMiXSxbMiwwLCJcXGNpcmMiXSxbMSwwLCJcXGJ1bGxldCJdLFswLDAsIlxcdGltZXMiXSxbMSwyLCIiLDAseyJzdHlsZSI6eyJoZWFkIjp7Im5hbWUiOiJub25lIn19fV0sWzEsMywiIiwyLHsic3R5bGUiOnsiaGVhZCI6eyJuYW1lIjoibm9uZSJ9fX1dLFszLDAsIiIsMix7InN0eWxlIjp7ImhlYWQiOnsibmFtZSI6Im5vbmUifX19XSxbMiw1LCIiLDAseyJzdHlsZSI6eyJoZWFkIjp7Im5hbWUiOiJub25lIn19fV0sWzUsNiwiIiwwLHsic3R5bGUiOnsiaGVhZCI6eyJuYW1lIjoibm9uZSJ9fX1dLFsyLDQsIiIsMCx7InN0eWxlIjp7ImhlYWQiOnsibmFtZSI6Im5vbmUifX19XSxbNiw3LCIiLDIseyJzdHlsZSI6eyJoZWFkIjp7Im5hbWUiOiJub25lIn19fV1d
		\[\begin{tikzcd}[column sep=small, row sep=small]
			\times & \bullet & \circ & \circ & \circ & \circ & \circ \\
			&&& \circ
			\arrow[no head, from=1-5, to=1-4]
			\arrow[no head, from=1-5, to=1-6]
			\arrow[no head, from=1-6, to=1-7]
			\arrow[no head, from=1-4, to=1-3]
			\arrow[no head, from=1-3, to=1-2]
			\arrow[no head, from=1-4, to=2-4]
			\arrow[no head, from=1-2, to=1-1]
		\end{tikzcd}\]
		
		\vspace*{1cm}
		$\mathrm{BI}_{120}$
		\vspace*{-0.5cm}
		% https://q.uiver.app/?q=WzAsOSxbNiwwLCJcXGNpcmMiXSxbNSwwLCJcXGNpcmMiXSxbMywwLCJcXGNpcmMiXSxbMiwwLCJcXGNpcmMiXSxbNCwwLCJcXGNpcmMiXSxbNSwxLCJcXGNpcmMiXSxbMSwwLCJcXGJ1bGxldCJdLFswLDAsIlxcdGltZXMiXSxbNywwLCJcXGNpcmMiXSxbMCwxLCIiLDAseyJzdHlsZSI6eyJoZWFkIjp7Im5hbWUiOiJub25lIn19fV0sWzIsMywiIiwwLHsic3R5bGUiOnsiaGVhZCI6eyJuYW1lIjoibm9uZSJ9fX1dLFsyLDQsIiIsMix7InN0eWxlIjp7ImhlYWQiOnsibmFtZSI6Im5vbmUifX19XSxbMyw2LCIiLDAseyJzdHlsZSI6eyJoZWFkIjp7Im5hbWUiOiJub25lIn19fV0sWzYsNywiIiwwLHsic3R5bGUiOnsiaGVhZCI6eyJuYW1lIjoibm9uZSJ9fX1dLFswLDgsIiIsMix7InN0eWxlIjp7ImhlYWQiOnsibmFtZSI6Im5vbmUifX19XSxbMSw1LCIiLDAseyJzdHlsZSI6eyJoZWFkIjp7Im5hbWUiOiJub25lIn19fV0sWzQsMSwiIiwyLHsic3R5bGUiOnsiaGVhZCI6eyJuYW1lIjoibm9uZSJ9fX1dXQ==
		\[\begin{tikzcd}[column sep=small, row sep=small]
			\times & \bullet & \circ & \circ & \circ & \circ & \circ & \circ \\
			&&&&& \circ
			\arrow[no head, from=1-7, to=1-6]
			\arrow[no head, from=1-4, to=1-3]
			\arrow[no head, from=1-4, to=1-5]
			\arrow[no head, from=1-3, to=1-2]
			\arrow[no head, from=1-2, to=1-1]
			\arrow[no head, from=1-7, to=1-8]
			\arrow[no head, from=1-6, to=2-6]
			\arrow[no head, from=1-5, to=1-6]
		\end{tikzcd}\]
		
		\caption{McKay graphs of the finite subgroups of $\SU(2)$. The cross identifies the trivial representation and the solid node is (or in the cyclic case is a component of) the natural representation.}
		\label{fig:McKayGraphs}
	\end{figure}
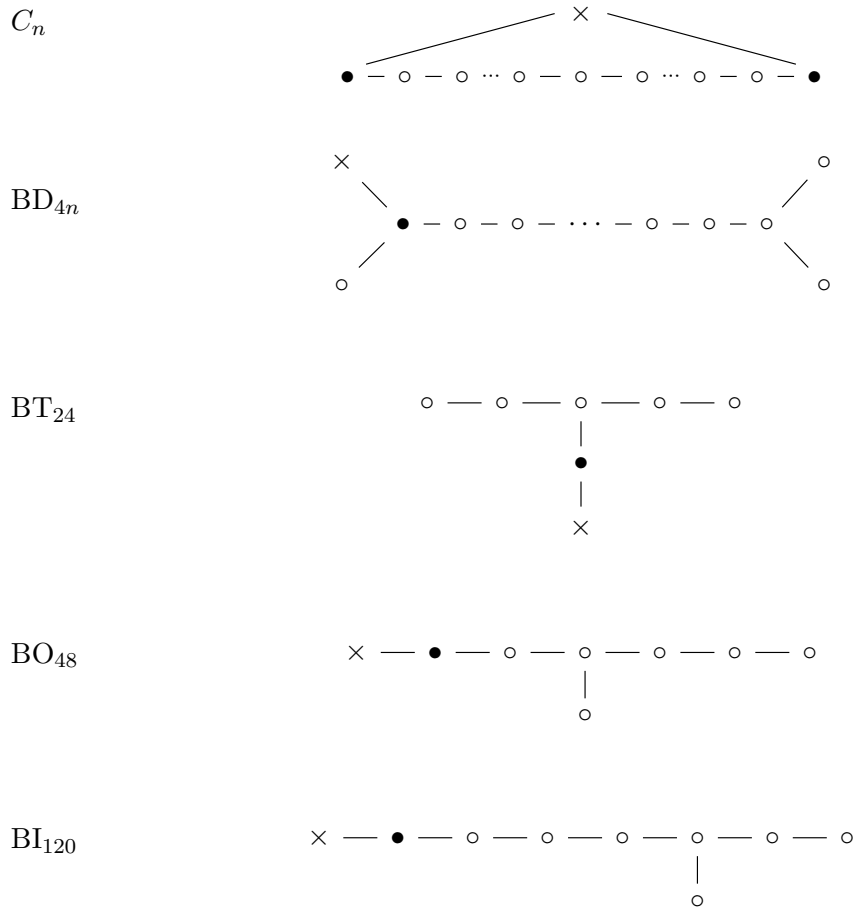
	
	Figure \ref{fig:McKayGraphs} shows the McKay graphs of the finite subgroups of SU(2). We have omitted the case of $2C_n$, who’s McKay graph is two disjoin copies of that of $C_n$. These graphs are exactly the \emph{extended Dynkin diagrams} of types $\widetilde{A}, \widetilde{D}, \widetilde{E_6}, \widetilde{E_7}, \widetilde{E_8}$ respectively. When one omits the vertex corresponding to the trivial representation, we are left with the \emph{Dynkyn diagrams} of types $A_n, D_n, E_6, E_7, E_8$. These occur in many areas of mathematics, including the classifications root systems, semisimple Lie algebras, and finite Coxeter groups.

	\pagebreak
	
	\subsection{Spinors and Pinors}
	
	The spinor group $\Spin (3)$ is the central extension of $\SO(3)$ given by the exact sequence,
	
	% https://q.uiver.app/?q=WzAsNSxbMywwLCJcXG1hdGhybXtTT30oMykiXSxbMiwwLCJcXG1hdGhybXtTcGlufSgzKSJdLFsxLDAsIlxcbGVmdFxce1xccG0gXFxib2xkezF9IFxccmlnaHRcXH0iXSxbNCwwLCJcXGJvbGR7MX0iXSxbMCwwLCJcXGJvbGR7MX0iXSxbMSwwLCJcXHBpIiwwLHsic3R5bGUiOnsiaGVhZCI6eyJuYW1lIjoiZXBpIn19fV0sWzIsMV0sWzAsM10sWzQsMl1d
	\[\begin{tikzcd}
		{1} & {\left\{\pm \mathbbm{1} \right\}} & {\mathrm{Spin}(3)} & {\mathrm{SO}(3)} & {1}
		\arrow["\pi", two heads, from=1-3, to=1-4]
		\arrow["\iota",hook, from=1-2, to=1-3]
		\arrow[from=1-4, to=1-5]
		\arrow[from=1-1, to=1-2]
	\end{tikzcd}\]
	
	which gives a double covering onto the special orthogonal group. Considering the exact sequence resulting from the inclusion into the orthogonal group,
	% https://q.uiver.app/?q=WzAsNSxbMiwwLCJcXG1hdGhybXtPfSgzKSJdLFszLDAsIlxcbGVmdFxce1xccG0gXFxib2xkezF9IFxccmlnaHRcXH0iXSxbMSwwLCJcXG1hdGhybXtTT30oMykiXSxbNCwwLCJcXGJvbGR7MX0iXSxbMCwwLCJcXGJvbGR7MX0iXSxbMCwxLCJcXG1hdGhybXtkZXR9IiwwLHsic3R5bGUiOnsiaGVhZCI6eyJuYW1lIjoiZXBpIn19fV0sWzIsMCwiaSIsMCx7InN0eWxlIjp7InRhaWwiOnsibmFtZSI6Imhvb2siLCJzaWRlIjoidG9wIn19fV0sWzEsM10sWzQsMl1d
	\[\begin{tikzcd}
		{1} & {\mathrm{SO}(3)} & {\mathrm{O}(3)} & {\left\{\pm \mathbbm{1} \right\}} & {1}
		\arrow["{\mathrm{det}}", two heads, from=1-3, to=1-4]
		\arrow["i", hook, from=1-2, to=1-3]
		\arrow[from=1-4, to=1-5]
		\arrow[from=1-1, to=1-2]
	\end{tikzcd}\]
	we would like to form some group to be a double covering of $\Orth(3)$ which preserves the map of the spinor group $\pi \colon \SU(2) = \Spin(3) \to \SO(3)$. This means the square in Figure \ref{fig:Pin3} will be commutative. As our induced map $\widetilde{\pi}$ will cover elements in $\Orth(3)$ with determinant not equal to $1$, we will lose the ``S'' from the label and call this a $\Pin(3)$ group. Atiyah, Bott, and Shapiro \citep{AtiyahBottShapiro} attribute this back formation to Serre.

	\begin{figure}[!hbtp]
		\centering
		\def\svgwidth{\columnwidth}
		% https://q.uiver.app/?q=WzAsMTYsWzIsMywiXFxtYXRocm17T30oMykiXSxbMywzLCJcXGxlZnRcXHtcXHBtIFxcYm9sZHsxfSBcXHJpZ2h0XFx9Il0sWzEsMywiXFxtYXRocm17U099KDMpIl0sWzQsMywiXFxib2xkezF9Il0sWzAsMywiXFxib2xkezF9Il0sWzEsMiwiXFxtYXRocm17U3Bpbn0oMykiXSxbMiwyLCJcXG1hdGhybXtQaW59X3tcXHBtfSgzKSJdLFszLDIsIlxcbGVmdFxce1xccG0gXFxib2xkezF9IFxccmlnaHRcXH0iXSxbMCwyLCJcXGJvbGR7MX0iXSxbMSwxLCJcXGxlZnRcXHtcXHBtIFxcYm9sZHsxfSBcXHJpZ2h0XFx9Il0sWzIsMSwiXFxsZWZ0XFx7XFxwbSBcXGJvbGR7MX0gXFxyaWdodFxcfSJdLFsxLDQsIlxcYm9sZHsxfSJdLFsyLDQsIlxcYm9sZHsxfSJdLFsxLDAsIlxcYm9sZHsxfSJdLFsyLDAsIlxcYm9sZHsxfSJdLFs0LDIsIlxcYm9sZHsxfSJdLFswLDEsIlxcbWF0aHJte2RldH0iLDAseyJzdHlsZSI6eyJoZWFkIjp7Im5hbWUiOiJlcGkifX19XSxbMiwwLCJpIiwwLHsic3R5bGUiOnsidGFpbCI6eyJuYW1lIjoiaG9vayIsInNpZGUiOiJ0b3AifX19XSxbMSwzXSxbNCwyXSxbNSwyLCJcXHBpIiwwLHsic3R5bGUiOnsiaGVhZCI6eyJuYW1lIjoiZXBpIn19fV0sWzYsMCwiXFx3aWRldGlsZGV7XFxwaX0iLDAseyJzdHlsZSI6eyJoZWFkIjp7Im5hbWUiOiJlcGkifX19XSxbNiw3LCIiLDIseyJzdHlsZSI6eyJoZWFkIjp7Im5hbWUiOiJlcGkifX19XSxbNSw2LCJpIiwwLHsic3R5bGUiOnsidGFpbCI6eyJuYW1lIjoiaG9vayIsInNpZGUiOiJ0b3AifX19XSxbOCw1XSxbOSw1LCJcXGlvdGEiLDAseyJzdHlsZSI6eyJ0YWlsIjp7Im5hbWUiOiJob29rIiwic2lkZSI6InRvcCJ9fX1dLFsxMCw2LCIiLDIseyJzdHlsZSI6eyJ0YWlsIjp7Im5hbWUiOiJob29rIiwic2lkZSI6InRvcCJ9fX1dLFsyLDExXSxbMCwxMl0sWzEzLDldLFsxNCwxMF0sWzcsMTVdXQ==
		\[\begin{tikzcd}
			& {1} & {1} \\
			& {\left\{\pm \mathbbm{1} \right\}} & {\left\{\pm \mathbbm{1} \right\}} \\
			{1} & {\mathrm{Spin}(3)} & {\mathrm{Pin}_{\pm}(3)} & {\left\{\pm 1 \right\}} & {1} \\
			{1} & {\mathrm{SO}(3)} & {\mathrm{O}(3)} & {\left\{\pm 1 \right\}} & {1} \\
			& {1} & {1}
			\arrow["{\mathrm{det}}", two heads, from=4-3, to=4-4]
			\arrow["i", hook, from=4-2, to=4-3]
			\arrow[from=4-4, to=4-5]
			\arrow[from=4-1, to=4-2]
			\arrow["\pi", two heads, from=3-2, to=4-2]
			\arrow["{\widetilde{\pi}}", two heads, from=3-3, to=4-3]
			\arrow[two heads, from=3-3, to=3-4]
			\arrow["i", hook, from=3-2, to=3-3]
			\arrow[from=3-1, to=3-2]
			\arrow["\iota", hook, from=2-2, to=3-2]
			\arrow[hook, from=2-3, to=3-3]
			\arrow[from=4-2, to=5-2]
			\arrow[from=4-3, to=5-3]
			\arrow[from=1-2, to=2-2]
			\arrow[from=1-3, to=2-3]
			\arrow[from=3-4, to=3-5]
		\end{tikzcd}\]
		\caption{The groups $\Pin_{+}(3)$ and $\Pin_{-}(3)$ are the two central extensions of $\Orth(3)$ by $\left\{ \pm \mathbbm{1}\right\}$ or $\left\{ \pm \mathbbm{1}\right\}$.}
		\label{fig:Pin3}
	\end{figure}

	\pagebreak
	
	% https://q.uiver.app/?q=WzAsNSxbMywwLCJcXG1hdGhybXtTT30oMykiXSxbMiwwLCJcXG1hdGhybXtTcGlufSgzKSJdLFsxLDAsIlxcbGVmdFxce1xccG0gXFxib2xkezF9IFxccmlnaHRcXH0iXSxbNCwwLCJcXGJvbGR7MX0iXSxbMCwwLCJcXGJvbGR7MX0iXSxbMSwwLCJcXHBpIiwwLHsic3R5bGUiOnsiaGVhZCI6eyJuYW1lIjoiZXBpIn19fV0sWzIsMV0sWzAsM10sWzQsMl1d
	\[\begin{tikzcd}
		{\mathbbm{1}} & {\left\{\pm \mathbbm{1} \right\}} & {\mathrm{Pin}(3)} & {\mathrm{O}(3)} & {\mathbbm{1}}
		\arrow["\widetilde{\pi}", two heads, from=1-3, to=1-4]
		\arrow[from=1-2, to=1-3]
		\arrow[from=1-4, to=1-5]
		\arrow[from=1-1, to=1-2]
	\end{tikzcd}\]
	
	It turns out that there are two possible choices of the $Pin(3)$ group. To see this, let $\rho \in \Orth(3)$ be an element of order 2, so $\rho^2=\mathbbm{1} \in \SO(3)$. Since $\widetilde{\pi}$ is surjective, there exists $\widehat{\rho} \in \Pin_{\pm}(3)$ such that $\widetilde{\pi}(\widehat{\rho})=\rho$, then
	
	\begin{equation*}
		\mathbbm{1}=\rho^2=\widetilde{\pi}(\widehat{\rho})^2=\widetilde{\pi}(\widehat{\rho}^2)=\widetilde{\pi} \circ i (\widehat{\rho}^2) = i \circ \pi (\widehat{\rho}^2)
	\end{equation*}
	
	So $\widehat{\rho}^2 \in \mathrm{ker}(\pi) = \mathrm{im}(\iota) = \left\{\pm\mathbbm{1}\right\}$, and we have a choice as to whether our reflection $\rho$ lifts to an element of order $2$ or order $4$. In the case that $\widehat{\rho}^2=-\mathbbm{1}$, we call this group $\Pin_-(3)$, and when  $\widehat{\rho}^2=\mathbbm{1}$ the covering group is called $\Pin_+(3)$.
	
	By \citep{Harvey} we have that these two groups satisfy the following isomorphisms,
	\begin{align*}
		&\Pin_{-}(3) \cong \left\{ A \in \mathrm{U}(2) \mid \det{A}=\pm 1 \right\}, \text{and}\\
		&\Pin_{+}(3) \cong \SU(2)\times C_2.
	\end{align*}
	More explicitly, the preimage of $\left\{\pm \mathbbm{1} \right\}$ in $\Pin_{-}(3)$ is the cyclic group $\left\{\mathbbm{1}, {W} , -\mathbbm{1}, -{W} \right\}$ generated by ${W}: = \left(\begin{smallmatrix}
		i & 0 \\ 0 & i
	\end{smallmatrix}\right)$ a central element with determinant $-1$.\\
	In the other case of $\Pin_{+}(3)$, the preimage is $\left\{ \left(\mathbbm{1},{1}\right),\left(\mathbbm{1},\bold{x}\right),\left(-\mathbbm{1},{1}\right),\left(-\mathbbm{1},\bold{x}\right) \right\}$ which is isomorphic to the Klein four-group $C_2 \times C_2$.
	
	The subgroups of $\SO(3)$ can be lifted to $\SL(2,\C)$ via the double cover. This preserves the network of containments. By taking the cases of $\Pin_{+}(3)$ and $\Pin_{-}(3)$ separately, we can get all the groups and their index $2$ containments. We list the groups and their matrix generators in Table \ref{Table:Pin3Subgroups} and produce graphs of index $2$ containments in Figures \ref{fig:O3Polyhedral} and \ref{fig:O3Polyhedral} analogous to those in \citep{CS}.
	
	\renewcommand{\arraystretch}{1.5}
	\begin{table}[!h]
		\centering
		\setlength\tabcolsep{3.5pt}
		\begin{tabular}{|c|c|c|}
			\hline
			Group                    & $\Pin_-(3)$                             & $\Pin_+(3)$                                             \\ \hline
			Binary cyclic            & $\mathrm{BC}_{2n}= \Span{A}$            & $\mathrm{BC}_{2n}= \Span{A}$                            \\ \hline
			Binary diplo-cyclic      & $\mathrm{B_-2C}_{4n}=\Span{A,W}$        & $\mathrm{B_+2C}_{4n}=\Span{A,\bold{x}}$                 \\ \hline
			Binary cyclo-cyclic      & $\mathrm{B_-CCD}_{4n}=\Span{A'W
%				,-A'W
			}$  & $\mathrm{B_+CCD}_{4n}=\Span{\bold{x}A
%			,\bold{x}(-A)
			}$    \\ \hline
			Binary dihedral          & $\mathrm{BD}_{4n}=\Span{A,B}$           & $\mathrm{BD}_{4n}=\Span{A,B}$                           \\ \hline
			Binary cyclo-dihedral    & $\mathrm{B_-CD}_{4n}=\Span{A,BW}$       & $\mathrm{B_+CD}_{4n}=\Span{A,\bold{x}B}$                \\ \hline
			Binary dihedro-dihedral  & $\mathrm{B_-DD}_{8n}=\Span{A'W
%				,-A'W
			,B}$ & $\mathrm{B_+DD}_{8n}=\Span{\bold{x}A'
%			,\bold{x}(-A')
			,B}$ \\ \hline
			Binary diplo-dihedral    & $\mathrm{B_-2D}_{8n}=\Span{A,B,W}$      & $\mathrm{B_+2D}_{8n}=\Span{A,B,\bold{x}}$               \\ \hline \hline
			Binary tetrahedral       & $\mathrm{BT}_{24}=\Span{I^2,Y}$         & $\mathrm{BT}_{24}=\Span{I^2,Y}$                         \\ \hline
			Binary diplo-tetrahedral & $\mathrm{B_-2T}_{48}=\Span{I^2,Y,W}$    & $\mathrm{B_+2T}_{48}=\Span{I^2,Y,\bold{x}}$             \\ \hline
			Binary octahedral        & $\mathrm{BO}_{48}=\Span{I,Y}$           & $\mathrm{BO}_{48}=\Span{I,Y}$                           \\ \hline
			Binary tetra-octahedral  & $\mathrm{B_-TO}_{48}=\Span{IW,Y}$       & $\mathrm{B_+TO}_{48}=\Span{\bold{x}I,Y}$                \\ \hline
			Binary diplo-octahedral  & $\mathrm{B_-2O}_{96}=\Span{I,Y,W}$      & $\mathrm{B_+2O}_{96}=\Span{I,Y,\bold{x}}$               \\ \hline
			Binary icosahedral       & $\mathrm{BI}_{120}=\Span{X,Z}$          & $\mathrm{BI}_{120}=\Span{X,Z}$                          \\ \hline
			Binary diplo-icosahedral & $\mathrm{B_-2I}_{240}=\Span{X,Z,W}$     & $\mathrm{B_+2I}_{240}=\Span{X,Z,\bold{x}}$              \\ \hline
		\end{tabular}
		
		\caption{The finite subgroups of $\Pin_\pm(3)$ and their matrix generators. To match the previous constructions, the axial groups are generated by matrices $A'=\begin{psmallmatrix}
				\ep^1 &  0 \\
				0 & \ep^{-1}
			\end{psmallmatrix}$, where $\ep$ is a primitive $4n$-th root of unity, $A={A'}^2$ and $B=\begin{psmallmatrix}
				0 & 1 \\
				-1 & 0
			\end{psmallmatrix}$. The polyhedral groups are generated by matrices $Y=\frac{1}{2}\begin{psmallmatrix}
				-1+i & 1+i\\
				-1+i &-1-i
			\end{psmallmatrix}$, $I=\begin{psmallmatrix}
				i & 0 \\
				0 & -i
			\end{psmallmatrix}$, $X=\frac{1}{\sqrt{5}}\begin{psmallmatrix}
				-\mu+\mu^4 & \mu^2-\mu^3 \\
				\mu^2-\mu^3 & \mu-\mu^4
			\end{psmallmatrix}$, and $Z=\begin{psmallmatrix}
				-\mu^3 & 0 \\
				0 & -\mu^2
			\end{psmallmatrix}$, where $\mu$ is a primitive $5$th root of unity. In the $\Pin_-(3)$ cases, we have lifted $-\mathbbm{1}\in\Orth(3)$ to $W=\frac{1}{2}\begin{psmallmatrix}
			-i & 0\\
			0 & i
		\end{psmallmatrix}$ which is central and of order $4$. In the $\Pin_+(3)$ cases, $-\mathbbm{1}\in\Orth(3)$ is lifted to a central element $\bold{x}$ of order 2. In our MAGMA calculations this is achieved by taking direct sums with appropriate elements of $C_2 = \left\{1,\bold{x}\right\}$.}
		
		\label{Table:Pin3Subgroups}
	\end{table}
	\renewcommand{\arraystretch}{1}
	
	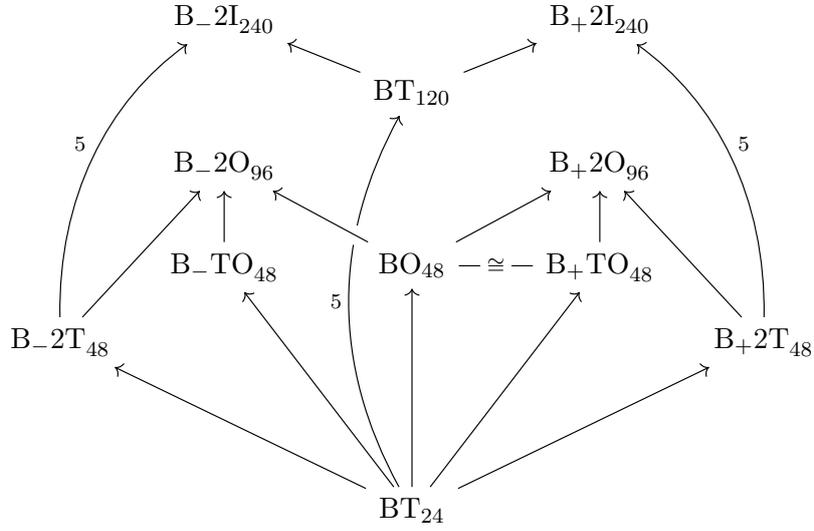
\begin{figure}[p]
		\centering
		\def\svgwidth{\columnwidth}
		% https://q.uiver.app/?q=WzAsMTMsWzMsOCwiXFxtYXRocm17QlR9X3syNH0iXSxbMyw0LCJcXG1hdGhybXtCT31fezQ4fSJdLFszLDIsIlxcbWF0aHJte0JUfV97MTIwfSJdLFswLDUsIlxcbWF0aHJte0Jfey19MlR9X3s0OH0iXSxbMSw0LCJcXG1hdGhybXtCX3stfVRPfV97NDh9Il0sWzEsMiwiXFxtYXRocm17Ql97LX0yT31fezk2fSJdLFs1LDQsIlxcbWF0aHJte0Jfeyt9VE99X3s0OH0iXSxbNiw1LCJcXG1hdGhybXtCX3srfTJUfV97NDh9Il0sWzUsMCwiXFxtYXRocm17Ql97K30ySX1fezI0MH0iXSxbMSwwLCJcXG1hdGhybXtCX3stfTJJfV97MjQwfSJdLFs1LDIsIlxcbWF0aHJte0Jfeyt9Mk99X3s5Nn0iXSxbMyw2LCIgIl0sWzMsNywiICJdLFswLDFdLFswLDIsIjUiLDAseyJjdXJ2ZSI6LTR9XSxbMCw0XSxbMSw1XSxbNCw1XSxbMyw1XSxbMCwzLCIiLDEseyJjdXJ2ZSI6LTN9XSxbMCw2XSxbMCw3LCIiLDEseyJjdXJ2ZSI6Mn1dLFsyLDksIiIsMix7ImN1cnZlIjoxfV0sWzIsOCwiIiwyLHsiY3VydmUiOi0xfV0sWzcsOCwiNSIsMix7ImN1cnZlIjozfV0sWzMsOSwiNSIsMCx7ImN1cnZlIjotM31dLFs2LDEwXSxbMSwxMF0sWzcsMTBdLFsxLDYsIlxcY29uZyIsMSx7InN0eWxlIjp7ImhlYWQiOnsibmFtZSI6Im5vbmUifX19XV0=
		\[\begin{tikzcd}[column sep=small, sep=small]
			& {\mathrm{B_{-}2I}_{240}} &&&& {\mathrm{B_{+}2I}_{240}} \\
			&&& {\mathrm{BT}_{120}}\\
			& {\mathrm{B_{-}2O}_{96}} &&&& {\mathrm{B_{+}2O}_{96}} \\
			\\
			& {\mathrm{B_{-}TO}_{48}} && {\mathrm{BO}_{48}} && {\mathrm{B_{+}TO}_{48}} \\
			{\mathrm{B_{-}2T}_{48}} &&&&&& {\mathrm{B_{+}2T}_{48}} \\
			&&& { } \\
			&&& { } \\
			&&& {\mathrm{BT}_{24}}
			\arrow[from=9-4, to=5-4]
			\arrow["5", curve={height=-30pt}, from=9-4, to=2-4]
			\arrow[from=9-4, to=5-2]
			\arrow[from=5-4, to=3-2, crossing over]
			\arrow[from=5-2, to=3-2]
			\arrow[from=6-1, to=3-2]
			\arrow[from=9-4, to=6-1]
			\arrow[from=9-4, to=5-6]
			\arrow[from=9-4, to=6-7]
			\arrow[from=2-4, to=1-2]
			\arrow[from=2-4, to=1-6]
			\arrow["5"', curve={height=25pt}, from=6-7, to=1-6]
			\arrow["5", curve={height=-25pt}, from=6-1, to=1-2]
			\arrow[from=5-6, to=3-6]
			\arrow[from=5-4, to=3-6]
			\arrow[from=6-7, to=3-6]
			\arrow["\cong"{description}, no head, from=5-4, to=5-6]
		\end{tikzcd}\]
		\caption{Containments between the finite subgroups of $\Pin_{\pm}(3)$. The index is 2 when not labelled. Groups on the left labelled with the $\mathrm{B_-}$ prefix are in $\Pin_{-}(3)$ and groups on the right with the $\mathrm{B_+}$ are in $\Pin_{+}(3)$.}
		\label{fig:Pin3SubgroupsPolyhedral}
	\end{figure}

	\begin{figure}[p]
		\centering
		\def\svgwidth{\columnwidth}
		% https://q.uiver.app/?q=WzAsMTMsWzMsNSwiIFxcbWF0aHJte0JDfV97Mm59Il0sWzQsMywiXFxtYXRocm17Ql97K31DRH1fezRufSJdLFszLDIsIlxcbWF0aHJte0JEfV97NG59Il0sWzYsMywiXFxtYXRocm17Ql97K31DQ31fezRufSJdLFs2LDIsIlxcbWF0aHJte0Jfeyt9MkN9X3s0bn0iXSxbMyw2LCJDX24iXSxbNCwwLCJcXG1hdGhybXtCX3srfUREfV97OG59Il0sWzYsMCwiXFxtYXRocm17Ql97K30yRH1fezhufSJdLFsyLDMsIlxcbWF0aHJte0Jfey19Q0R9X3s0bn0iXSxbMCwzLCJcXG1hdGhybXtCX3stfUNDfV97NG59Il0sWzAsMiwiXFxtYXRocm17Ql97LX0yQ31fezRufSJdLFsyLDAsIlxcbWF0aHJte0Jfey19RER9X3s4bn0iXSxbMCwwLCJcXG1hdGhybXtCX3stfTJEfV97OG59Il0sWzAsMV0sWzAsMiwiIiwyLHsib2Zmc2V0IjoxfV0sWzAsM10sWzUsMF0sWzIsNl0sWzQsN10sWzEsN10sWzIsN10sWzAsNCwiIiwyLHsib2Zmc2V0IjotMX1dLFszLDZdLFsxLDZdLFszLDRdLFs2LDddLFswLDhdLFswLDldLFswLDEwXSxbOSwxMF0sWzExLDEyXSxbOCwxMV0sWzIsMTFdLFs5LDExXSxbMTAsMTJdLFs4LDEyXSxbMiwxMl0sWzIsMSwiXFxjb25nIiwzLHsic3R5bGUiOnsiYm9keSI6eyJuYW1lIjoibm9uZSJ9LCJoZWFkIjp7Im5hbWUiOiJub25lIn19fV1d
		\[\begin{tikzcd}[column sep=small]
			{\mathrm{B_{-}2D}_{8n}} && {\mathrm{B_{-}DD}_{8n}} && {\mathrm{B_{+}DD}_{8n}} && {\mathrm{B_{+}2D}_{8n}} \\
			\\
			{\mathrm{B_{-}2C}_{4n}} &&& {\mathrm{BD}_{4n}} &&& {\mathrm{B_{+}2C}_{4n}} \\
			{\mathrm{B_{-}CC}_{4n}} && {\mathrm{B_{-}CD}_{4n}} && {\mathrm{B_{+}CD}_{4n}} && {\mathrm{B_{+}CC}_{4n}} \\
			\\
			&&& { \mathrm{BC}_{2n}} \\
			&&& {C_n}
			\arrow[from=6-4, to=4-5]
			\arrow[shift right=1, from=6-4, to=3-4]
			\arrow[from=6-4, to=4-7]
			\arrow[from=7-4, to=6-4]
			\arrow[from=3-4, to=1-5]
			\arrow[from=3-7, to=1-7]
			\arrow[from=4-5, to=1-7]
			\arrow[from=3-4, to=1-7]
			\arrow[from=6-4, to=3-7]
			\arrow[from=4-7, to=1-5]
			\arrow[from=4-5, to=1-5]
			\arrow[from=4-7, to=3-7, dashed]
			\arrow[from=1-5, to=1-7, Rightarrow, dashed]
			\arrow[from=6-4, to=4-3]
			\arrow[from=6-4, to=4-1]
			\arrow[from=6-4, to=3-1]
			\arrow[from=4-1, to=3-1, dashed]
			\arrow[from=1-3, to=1-1, Rightarrow, dashed]
			\arrow[from=4-3, to=1-3]
			\arrow[from=3-4, to=1-3]
			\arrow[from=4-1, to=1-3]
			\arrow[from=3-1, to=1-1]
			\arrow[from=4-3, to=1-1]
			\arrow[from=3-4, to=1-1]
			\arrow["\cong"{description}, no head, from=3-4, to=4-5]
			\arrow[from=3-4, to=3-4, loop left, Rightarrow, dashed, looseness=7]
			\arrow[from=4-5, to=4-5, loop left, Rightarrow, dashed, looseness=5]
			\arrow[from=4-3, to=4-3, loop right, Leftarrow, dashed, looseness=5]
			\arrow[from=1-1, to=1-1, loop left, Rightarrow, dashed, looseness=5]
			\arrow[from=1-7, to=1-7, loop right, Leftarrow, dashed, looseness=5]
			\arrow[from=3-1, to=3-1, loop left, rightarrow, dashed, looseness=5]
			\arrow[from=3-7, to=3-7, loop right, leftarrow, dashed, looseness=5]
		\end{tikzcd}\]
		\caption{Index 2 containments between the axial subgroups of $\Pin_{\pm}(3)$. Groups on the left labelled with the $\mathrm{B_-}$ prefix are in $\Pin_{-}(3)$ and groups on the right with the $\mathrm{B_+}$ are in $\Pin_{+}(3)$.}
		\label{fig:Pin3SubgroupsAxial}
	\end{figure}
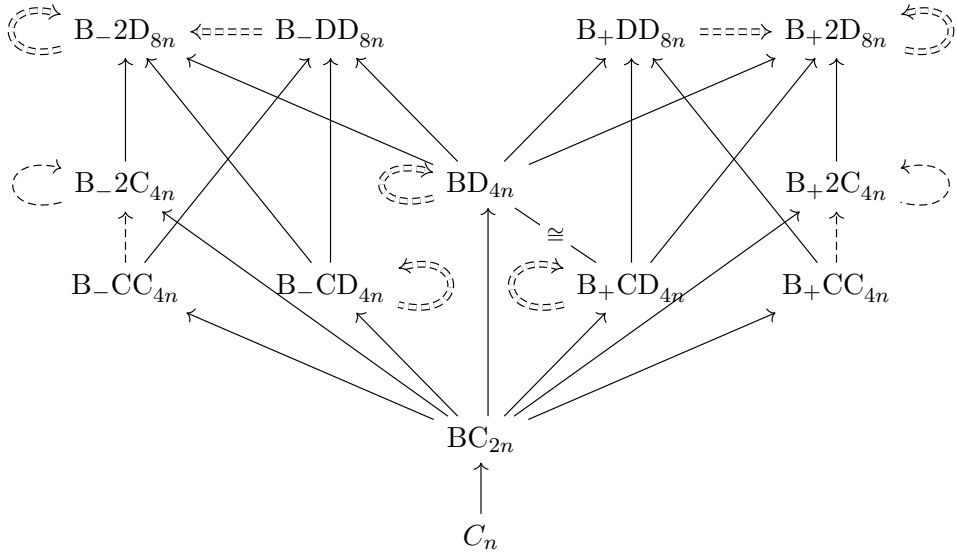

	\pagebreak
	
	\pagebreak

\section{K-Theory}

\subsection{Vector Bundles}
We take a brief detour to establish some basic topological $K$-theory. Using mainly Hatcher \citep{HatcherVBKT} and Atiyah \citep{Atiyah67}.

\begin{defn}\label{defn:VectorBundle}
Let $X$ be a topological space. An \textit{$n$-dimensional vector bundle} over $X$ is a topological space $E$ together with a map $p : E \to X$, where $p^{-1}(x)$ is an $n$-dimensional vector space over a field $k$, such that there is a cover of $X$ by open sets $U_\alpha$ for each of which there exists a homeomorphism $h_\alpha : p^{-1}(U_\alpha) \to U_\alpha \times k^n$, which takes $p^{-1}(x)$ to a $\{x\}\times k^n$ by a vector space isomorphism for each $x\in X$.
\end{defn}
Here, $h_\alpha$ is called a \textit{local trivialisation} of our \textit{total space} $E$. The space $X$ is called the base space and $E_x := p^{-1}(x)$ is the \textit{fibre over $x$}.
If we take our field to be $\C$ or $\R$ which we call our $E$ a \textit{complex vector bundle} or \textit{real vector bundle} respectively.

Let the $\ep^n \to X$ denote the the trivial $n$-dimensional vector bundle. Two vector bundles $E_1, E_2$ over $X$ are defined to be \emph{stably isomorphic}, denoted $E_1 \approx_s E_2$ if and only if $E_1 \oplus \ep^n \approx E_2 \oplus \ep^n$ for some $n$ \citep{HatcherVBKT}.

Consider the set $\mathrm{Vect}(X)$ of stable isomorphism classes of complex vector bundles over $X$. Together with the direct sum, $(\mathrm{Vect}(X),\oplus)$ has the structure of a commutative monoid. That is, it has an additive structure with identity element given by $\ep^0$.

\begin{defn}[{\citep{HahnOmeara}}]\label{defn:GrothendieckGroup}
Given a commutative monoid $M$, the \emph{Grothendieck group} associated to $M$ is a pair $\left(K,i\right)$ consisting of an Abelian group $K$ and a monoid homomorphism $i:M\to K$ with the following universal property:\\
For any Abelian $A$ with a monoid homomorphism $f:M \to A$, there is a unique group homomorphism $g : K \to A$ such that $f = g \circ i$. \\
Equivalently, the diagram
% https://q.uiver.app/?q=WzAsMyxbMCwwLCJNIl0sWzEsMCwiSyJdLFsxLDEsIkEiXSxbMCwxLCJpIl0sWzAsMiwiZiIsMl0sWzEsMiwiZyJdXQ==
\[\begin{tikzcd}
M & K \\
& A
\arrow["i", from=1-1, to=1-2]
\arrow["f"', from=1-1, to=2-2]
\arrow["g", from=1-2, to=2-2]
\end{tikzcd}\]
commutes.
\end{defn}
An immediate consequence of this universal property is that, if it exists,  the Grothendieck group is unique up to a canonical isomorphism.

Just as the rational numbers are constructed from the integers by forming quotients $\frac{a}{b}$ with an equivalence relation $\frac{a}{b}=\frac{c}{d}$ if and only if $ad=bc$ \citep{HatcherVBKT}, we form the the Grothendieck group $K(\mathrm{Vect}(X))$ which consists of formal differences $E-E'$ of vector bundles $E, E'$ on our space $X$, up to an equivalence relation $E_1-E_1' = E_2-E_2'$ if and only if $E_1 \oplus E_2' \approx_s E_2 \oplus E_1$.

\subsection{Equivariant K-Theory}
Equivariant $K$-theory was invented by Atiyah in 1966, and is used to study vector bundles which respect a group action. This section mainly follows Atiyah and Segal's treatment of equivariant $K$-theory in \citep{AtiyahSegal} and \citep{SegalEqK}, we will use this to construct an isomorphism between the $G$-equivariant $K$-theory and the representation ring $R_\C(G)$, which is constructed as the Grothendieck ring of the finite-dimensional complex representations of $G$.

\begin{defn}
Let $G$ be a group, we define a \emph{ $G$-space} $X$ to be a topological space on which $G$ has a continuous action.
\end{defn}

\pagebreak

\begin{defn}
A $G$-\emph{vector bundle} on a $G$-space $X$ is a $G$-space $E$ together with a projection map $p : E \to X$, such that
\begin{itemize}
\item[i] $E$ is a complex vector bundle $E$ over $X$,
\item[ii] for any $\xi \in E$, $p(g \cdot \xi) = g \cdot p(\xi)$, and
\item[iii] for any $g \in G, x\in X$ the group action $g : E_x \to E_ {gx}$ is a vector space homomorphism.
\end{itemize}
\end{defn}

As in the previous section, stable isomorphism classes of $G$-vector bundles form a monoid with direct sum. Forming the Grothendieck group out of formal differences and noting that this group is closed into tensor product allows us to form the $G$-equivariant $K$-theory $ K_G(X)$ which is a subring of $K(X)$.

Segal introduced the following in \citep{SegalRR}.

\begin{defn}
The \emph{representation ring} of a group $G$ is the Grothendieck group of the monoid of complex representations of $G$. Concretely, it is comprised of formal differences of representations with addition given by direct sum, and is made into a ring with multiplication given by tensor product.
\end{defn}

The following fact is used in many places, and our formulation of it is closest to that of Segal's in \citep{SegalEqK}.

\begin{prop} \label{prop:KGRG}
\begin{equation*}
K_G(\mathrm{pt})\simeq K_G(\C^2)\simeq R_\C(G)
\end{equation*}
\end{prop}
\begin{proof}
When the topological space is just a point, a $G$-vector bundle is just a $\C[G]$-module, and so the two rings coincide. The equivalence to that of $K_G(\C^2)$ uses homotopy invariance of the $K(\cdot)$ functor as a cohomology theory. For this detail, we defer to Thomason \citep[Theorem~4.1]{Thomason}.
\end{proof}

\subsection{Algebraic K-theory}\label{Sec:AlgK}

As we wish to work with algebraic varieties and schemes we need to recreate the topological work work in the context of algebraic $K$-theory. Much of our construction here is from Weibel's book \citep{Weibel}. We omit basic scheme and sheaf theoretic definitions and refer to \citep{Hartshorne}.

To do this, we work with sheaves of modules over schemes, this has many features in common with the topological vector bundles that we considered previously.

\begin{defn}
Given $X$ a scheme, a sheaf of $\mathcal{O}_X$-modules $\mathcal{F}$ is \emph{quasicoherent} if and only if $X$ can be covered by affine opens $U_i = \Spec(R_i)$ such that $\left.\mathcal{F}\right|_{U_i}$ is $\tilde{M_i}$  for an $R$-module $M_i$. Moreover, $\mathcal{F}$ is \emph{coherent} if and only if each $M_i$ is a finitely generated $R$-module.
\end{defn}
For our purposes, we only need to work with affine schemes $X$, so  $\mathcal{F}=\tilde{M}$ is coherent if and only if $M$ is a finitely generated $R$-module \citep{Weibel}. Working explicitly over $\C$, a coherent sheaf is exactly a $\C[X]$-module where $\C[X]$ is the coordinate ring of $X$.

We can view such sheaves of $\mathcal{O}_X$-modules as sections of the \emph{\'etal\'e space}.

\begin{defn}
Given a sheaf $\mathcal{F}$ over a scheme $X$, the \emph{\'etal\'e space} of $\mathcal{F}$ is a topological space $E$ together with a local homeomorphism $\pi :E\to X$ such that the sheaf of sections $\Gamma (\pi ,-)$ of $\pi$ is $\mathcal{F}$. We can concretely construct $E$ to be the disjoint union of the stalks $\mathcal{F}_x$ with a topology based by $\left\{s_x\; \mid \;U \subset X \text{ open}, s \in \mathcal{F}(U)\right\}$.
\end{defn}

The following algebraic definition of a vector bundle is from \citep[p.~39]{Weibel}.

\begin{defn}
An \emph{algebraic vector bundle} over a ringed space $(X, \mathcal{O}_X)$ is a locally free $\mathcal{O}_X$ module whose rank is finite at every point. We write $\mathrm{Vect}(X,\mathcal{O}_X)$ for the set of vector bundles on $(X,\mathcal{O}_X)$.
\end{defn}

\begin{eg}[{\citep[Example~5.1.1.]{Weibel}}]
Note that for a topological space $X$, we can form a locally ringed space $X_{top}:=(X,\mathcal{O}_{top})$, where $\mathcal{O}_{top}$ is the sheaf of $\C$-valued continuous functions on $X$. Then we have an equivalence between $\mathrm{Vect}(X_{top})$ and our topological vector bundles $\mathrm{Vect}(X)$ from section 3.
\end{eg}

\begin{lem}[{\citep[Lemma~5.1.3]{Weibel}}]
For a scheme $X$, an $\mathcal{O}_X$-module $\mathcal{F}$ is a vector bundle if and only if $\mathcal{F}$ is coherent and the stalks $\mathcal{F}_x$ are free $\mathcal{O}_{X,x}$-modules.
\end{lem}

Weibel constructs the $K$-theory $K(X)$ of a scheme $X$ in \citep[p.~61]{Weibel} to be the Grothendieck group of algebraic vector bundles. One similarly generalises the previous section to obtain $G$-equivariant vector bundles in this new sense. These generalisations will be used in the following section.

\section{McKay Correspondence as an Equivalence of K-Theories}
The correspondence, observed by McKay in 1980 \citep{McKay80}, is a one-to-one correspondence between the representation theory finite subgroups of $\SU(2)$ and the simply-laced extended Dynkin diagrams arising from the corresponding quotient singularities. See \citep{IVD} or \citep{RS}.

Where $X=\C^2$, the minimal resolution $\widetilde{X \git G}$ has an exceptional divisor (the preimage of the singular point) comprised of a union of complex projective lines. Dependent on the group $G$, the dual graph of the configuration of lines is exactly a Dynkin diagram of type $A_n$, $D_n$, $E_6$, $E_7$, $E_8$. On the other hand, by constructing the \textit{McKay graph} whose vertices are irreducible representations of $G$, we get the extended Dynkin diagrams of the corresponding types. This relates the geometry of Kleinian singularity to the representation theory of the group itself.

We give an overview of the work by Gonzalez-Springberg and Verdier \citep{GSV}, where they demonstrated the connection between the geometric and representation theoretic views of the McKay correspondence. Slodowy also gives a useful overview of this argument in \citep{Slodowy}.

\pagebreak

\begin{defn}\label{def:fibreproduct}
	Given topological spaces $X,Y,Z$ and homeomorphisms \\ $f:X\to Z,g:Y\to Z$ we define the \emph{fibre product} or \emph{pullback} to be
	\begin{equation*}
		X \times_Z Y := \left\{(x,y)\subset X\times Y \mid f(x)=g(x) \right\}
	\end{equation*}
\end{defn}
This can also be defined for more general categories via a universal property, but this concrete construction will suffice for our purposes.

Our finite groups $G\subset \SU(2) \subset \GL(2,\C)$ acts naturally on $X=\C^2$. We would like to construct a topological space out of these $G$-orbits. However naively taking what is known as the \emph{geometric quotient} to be the set $ X / G := \left\{G\cdot x \mid x \in X\right\}$ might not be "good" when you consider the closure of the orbits (with respect to the Zariski topology when considering algebraic varieties). To respect the coordinate rings of algebraic varieties or schemes, we introduce the geometric invariant theory quotient below.

\begin{defn}\label{def:GITquotient}
	For a reductive group $G$ with an action on an affine variety $X$, there is inherited action on the coordinate ring $k[X]$. The \emph{GIT quotient} of a reductive group action $G$ on an affine variety $X$ is defined by $X \git G := \Spec(k[X]^G)$, where $k[X]^G$ is ring of $G$-invariants.
\end{defn}

It can be shown (see \citep{NewsteadGIT} or \citep{MumfordGIT}) that the GIT quotient is a categorical quotient. This implies that when the geometric quotient respects the category of varieties, the two quotients coincide.

\begin{eg}
	Let $n \in \N$ and $G = \mathrm{BD}_{4n} \subset \SU(2)$ be the binary dihedral group. This is the cyclic extension of the dihedral group of order $2n$, and so has group  presentation
	\begin{equation*}
		\mathrm{BD}_{4n}=\langle a, b \mid a^{n} = b^4=1, bab^{-1}=a^{-1} \rangle.
	\end{equation*}
	Its natural presentation in $\SU(2)$ can be given by a matrix group with generators
	\begin{equation*}
		G=\langle A, B\rangle, \quad \text{where }
		A=
		\begin{pmatrix}
			\ep & 0 \\
			0 & \ep^{-1}
		\end{pmatrix}, \,
		B=
		\begin{pmatrix}
			0 & 1 \\
			-1 & 0
		\end{pmatrix}
	\end{equation*}
	where $\ep$ is a primitive $2n$-th root of unity.
	
	The coordinate ring of affine space $\C^2$ is $\C[u,v]$. The process of computing the generators and relations for the ring of invariant polynomials can be found in \citep{FK}, \citep{ReidDuVal}, \citep{JCS}, \citep{FineganMcKay} or \citep{joncheahSU2}. This allows us to write
	\begin{align*}
		\C[u,v]^G&=\C[uv\left(u^{2n}-v^{2n}\right),u^{2n}+v^{2n}, (uv)^2],\\
		&=\C[x,y,z]/\left(x^2+y^2z +z^n\right).
	\end{align*}
	up to some rescaling of the generators $x,y,z$, see \citep[p.~13]{FineganMcKay} for the explicit scalars.
	By the above construction, our orbit quotient is then
	\begin{equation*}
		\C^2 \git G = \Spec\left(\C[u,v]^G\right) = \Spec\left(\C[x,y,z]/\left(x^2+y^2z +z^n\right)\right).
	\end{equation*}

\end{eg}
The quotients formed by such $G\subset\SU(2)$ are known as \emph{Kleinian singularities}, \emph{Du Val singularities}, or \emph{a rational double point} \citep{RS}. We embed this  variety as a surface in $\mathbb{A}^3_\C$ with an isolated singularity at the origin. We need a way to resolve this singularity, i.e. find another algebraic variety (or scheme) which is birational to our Kleinian singularity. This can be done via blow-ups (for a detailed exploration of the Kleinian singularities and their resolutions via blow-ups see \citep{JJGreen}) or via other means. For our current purposes we use the $G$-Hilbert scheme as introduced by Ito and Nakamura in \citep{ItoNakamura99}.

\begin{defn}
	For a group $G$ with an action on $\C[u,v]$, define the \emph{$G$-Hilbert Scheme} to be
	
	\begin{equation*}
		\GHilb(\C^2)=\left\{I\ \unlhd \, \C[u,v] \mid G \cdot I = I, \; \C[u,v]/I \cong \C[G] \right\}.
	\end{equation*}
\end{defn}

Ito and Nakamura proved the $G$-Hilbert scheme to be a minimal resolution of $\C^2 \git G$ for finite subgroups $G\subset \SL(2,\C)$. That is to say, for any other resolution of the Kleinian singularity, the birational maps factor through $\GHilb(\C^2)$. Their proof \citep[p.~187]{ItoNakamura99} was a general argument, independent of the classification of the finite subgroups, and their results were expanded upon by other authors such as \citep{Ishii02} who showed that the $G$-Hilbert scheme is a minimal resolution of $\C^2 \git G$ for a small finite subgroup $G\subset \GL(2,\C)$.

\begin{eg}
	The concrete case of the binary dihedral group in the case of $n=2$ is explored in \citep{ReidDuVal} and \citep{RS}. In fact, this group is isomorphic to the quaternion group $\mathrm{Q}_8 = \langle i,j,k \mid i^2=j^2=k^2=ijk=-1  \rangle $, where $-1$ is a central element of order $2$. Here, the real resolution of the real variety $\C^3/(x^2+y^2z +z^3) \cap \R^3$ shown in Figure \ref{fig:BD8} gives a fairly faithful picture of what is happening. The exceptional divisor, or the preimage of the origin under the desingualarisation, is comprised of a union of projective lines whose configuration gives the Dynkin diagram $D_4$.

	\begin{figure}[h]
		\centering
		\def\svgwidth{0.8\columnwidth}
%% Accompanies image file 'drawing.pdf' (pdf, eps, ps)
\begingroup%
  \makeatletter%
  \providecommand\color[2][]{%
    \errmessage{(Inkscape) Color is used for the text in Inkscape, but the package 'color.sty' is not loaded}%
    \renewcommand\color[2][]{}%
  }%
  \providecommand\transparent[1]{%
    \errmessage{(Inkscape) Transparency is used (non-zero) for the text in Inkscape, but the package 'transparent.sty' is not loaded}%
    \renewcommand\transparent[1]{}%
  }%
  \providecommand\rotatebox[2]{#2}%
  \newcommand*\fsize{\dimexpr\f@size pt\relax}%
  \newcommand*\lineheight[1]{\fontsize{\fsize}{#1\fsize}\selectfont}%
  \ifx\svgwidth\undefined%
    \setlength{\unitlength}{476.87178019bp}%
    \ifx\svgscale\undefined%
      \relax%
    \else%
      \setlength{\unitlength}{\unitlength * \real{\svgscale}}%
    \fi%
  \else%
    \setlength{\unitlength}{\svgwidth}%
  \fi%
  \global\let\svgwidth\undefined%
  \global\let\svgscale\undefined%
  \makeatother%
  \begin{picture}(1,0.68661499)%
    \lineheight{1}%
    \setlength\tabcolsep{0pt}%
    \put(0,0){\includegraphics[width=\unitlength,page=1]{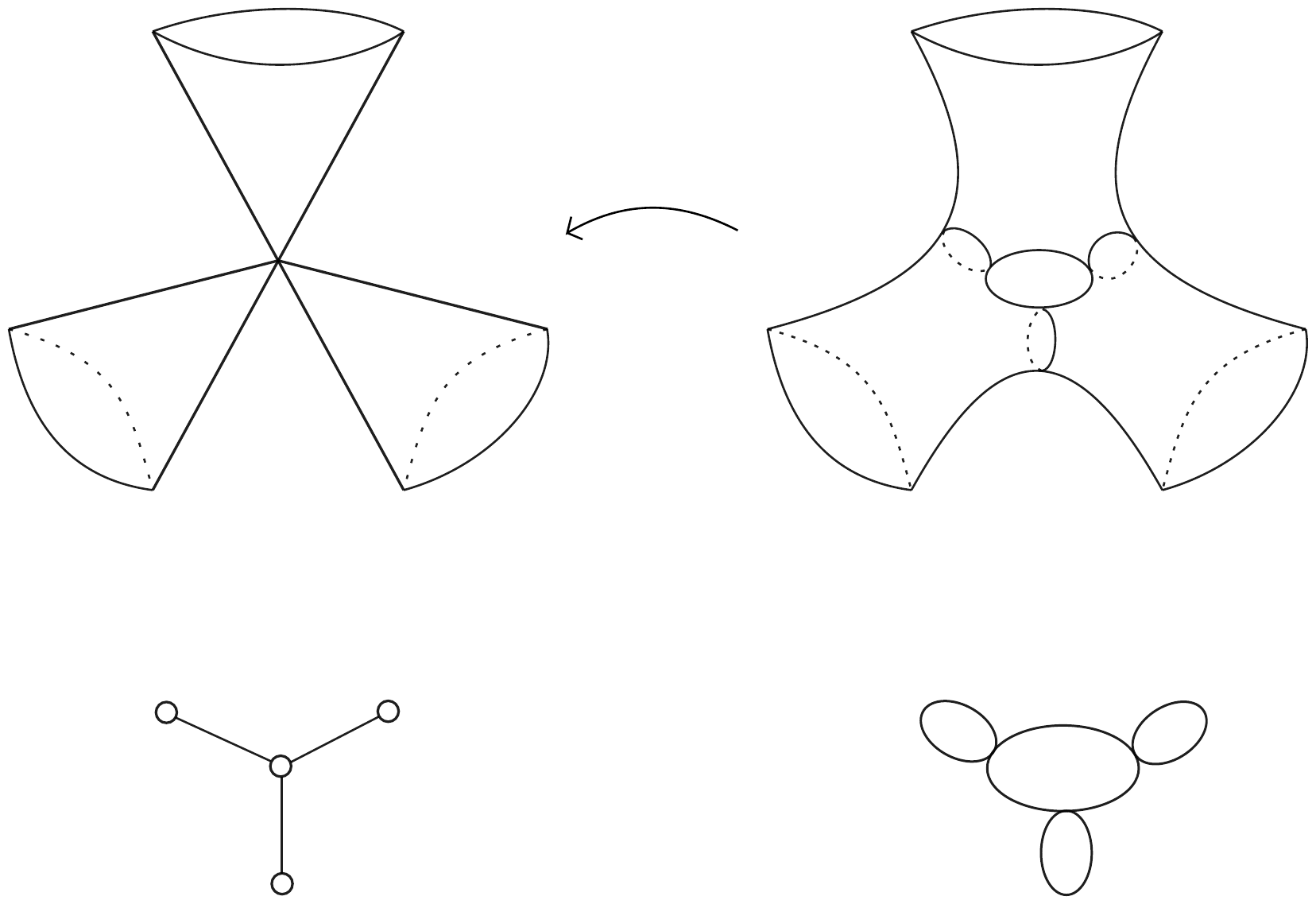}}%
    \put(0.4644527,0.53571793){\color[rgb]{0,0,0}\makebox(0,0)[lt]{\lineheight{1.25}\smash{\begin{tabular}[t]{l}$\pi$\end{tabular}}}}%
    \put(0.14805964,0.4853408){\color[rgb]{0,0,0}\makebox(0,0)[lt]{\lineheight{1.25}\smash{\begin{tabular}[t]{l}$O$\end{tabular}}}}%
    \put(0.88,0.51052935){\color[rgb]{0,0,0}\makebox(0,0)[lt]{\lineheight{1.25}\smash{\begin{tabular}[t]{l}$\pi^{-1}(O)$\end{tabular}}}}%
    \put(0,0){\includegraphics[width=\unitlength,page=2]{drawing.pdf}}%
    \put(0.25717356,0.01390129){\color[rgb]{0,0,0}\makebox(0,0)[lt]{\lineheight{1.25}\smash{\begin{tabular}[t]{l}$D_4$\end{tabular}}}}%
  \end{picture}%
\endgroup%

		\caption{The real picture when $G=\mathrm{BD}_{8}$, recreated from \citep[Fig.~A.1]{RS}.}
		\label{fig:BD8}
	\end{figure}
\end{eg}

\pagebreak

To view this correspondence as Gonzalez-Springberg and Verdier did, we construct a diagram of spaces as below. Here, $X=\C^2$ and $G\subset \SU(2)$ is a finite subgroup, the maps $p_1,p_2$ are the natural projection from the fibre product, and $\widetilde{X \git G}$ is the minimal resolution of the singularity in $X\git G$.

% https://q.uiver.app/?q=WzAsNCxbMCwxLCJYIl0sWzEsMCwiWFxcdGltZXNfe1ggXFxnaXQgR30gXFx3aWRldGlsZGV7WFxcZ2l0IEd9Il0sWzIsMSwiXFx3aWRldGlsZGV7WFxcZ2l0IEd9Il0sWzEsMV0sWzEsMiwicF8yIl0sWzEsMCwicF8xIiwyXV0=
\[\begin{tikzcd}[column sep=tiny]
	& {X\times_{X \git G} \widetilde{X\git G}} \\
	X && {\widetilde{X\git G}}
	\arrow["{p_2}", from=1-2, to=2-3]
	\arrow["{p_1}"', from=1-2, to=2-1]
\end{tikzcd}\]

\begin{thm}[The McKay Correspondence, \citep{GSV}]
	
	For any finite subgroup $G\subset \SU(2)$, there is an isomorphism between the representation ring $R_\C(G)$ and the $G$-equivariant $K$-theory of $X=\C^2$. Additionally the projection maps from the fibre product space induce an isomorphism between ${K}_G(X)$ and $K\left(\widetilde{X \git G} \right)$.
	% https://q.uiver.app/?q=WzAsNCxbMCwxLCJSX1xcQyhHKSJdLFsxLDEsIktfRyhYKSJdLFsyLDAsIktcXGxlZnQoWFxcdGltZXNfe1ggXFxnaXQgR30gXFx3aWRldGlsZGV7WFxcZ2l0IEd9XFxyaWdodCkiXSxbMywxLCJLXFxsZWZ0KFxcd2lkZXRpbGRle1ggXFxnaXQgR31cXHJpZ2h0KSJdLFswLDEsIlxcc2ltZXEiLDMseyJzdHlsZSI6eyJib2R5Ijp7Im5hbWUiOiJub25lIn0sImhlYWQiOnsibmFtZSI6Im5vbmUifX19XSxbMSwyLCJwXzFeKiJdLFsyLDMsIlxcbWF0aHJte0ludn1cXGNpcmN7cF8yfV8qIl0sWzEsMywiXFxzaW1lcSJdXQ==
	\[\begin{tikzcd}[column sep=tiny]
		&& {K\left(X\times_{X \git G} \widetilde{X\git G}\right)} \\
		{R_\C(G)} & {K_G(X)} && {K\left(\widetilde{X \git G}\right)}
		\arrow["\simeq"{marking}, draw=none, from=2-1, to=2-2]
		\arrow["{p_1^*}", from=2-2, to=1-3]
		\arrow["{\mathrm{Inv}\circ{p_2}_*}", from=1-3, to=2-4]
		\arrow["\simeq", from=2-2, to=2-4]
	\end{tikzcd}\]
\end{thm}

The explicit isomorphisms $R_\C(G) \simeq K_G(X)$ and $K_G(X) \simeq K\left(\widetilde{X \git G}\right)$ in \citep{GSV} use the fundamental cycle, a definition of which can be found in \citep{MA}.

\pagebreak
\section{Calculating Real Frobenius-Schur Indicators}

A consequence of Schur's lemma is that the endomorphism ring commuting with the group action in a representation is a real associative division algebra. Frobenius's theorem then gives that this can only be isomorphic to the real numbers $\R$, complex numbers $\C$, or quaternions $\mathbb{H}$. The Frobenius-Schur indicator is used to distinguish between these cases.

\begin{defn}\label{defn:FSindicators}
	For a fixed group $G$ the \emph{Frobenius-Schur indicator} of a $\C G$-module $V$ is
	\begin{equation*}
		\mathcal{F}_\C (V)=\mathcal{F}_\C (\chi):= \frac{1}{\left|G\right|}\sum_{g\in G}\chi(g^2)
	\end{equation*}
	where $\chi$ is the character of $V$.
\end{defn}
For characters of irreducible representations $\mathcal{F}_\C =1,0,$ or $-1$, corresponding to the cases $\R$, $\C$, and $\mathbb{H}$ respectively.

We will be looking at $C_2$ graded subgroups, and introduce more general representations of  $G\lhd\widehat{G}$ used in \citep{AtiyahSegal} and \citep{JTDR}.

\begin{defn}
	A \emph{Real representation} or \emph{antilinear representation} of $C_2$ graded groups $G\lhd \widehat{G}$ is a homomorphism $\rho : \widehat{G} \to \GL(V)$ of $C_2$ graded groups. Equivalently, it is a representation of $G$ with an antilinear involution which respects the conjugation action of $\widehat{G}$ on $G$.
\end{defn}

Given $G\lhd \widehat{G}$, Taylor \citep{JTDR} studies the corresponding \emph{antilinear block} or \emph{$A$-block}, which is the block of $\R$ algebras inclusions
% https://q.uiver.app/?q=WzAsNCxbMCwwLCJcXFIgRyJdLFsxLDAsIlxcQyBHIl0sWzAsMSwiXFxSIFxcd2lkZWhhdHtHfSJdLFsxLDEsIlxcQyAqIFxcd2lkZWhhdHtHfSJdLFswLDIsIiIsMSx7InN0eWxlIjp7InRhaWwiOnsibmFtZSI6Imhvb2siLCJzaWRlIjoidG9wIn19fV0sWzIsMywiIiwxLHsic3R5bGUiOnsidGFpbCI6eyJuYW1lIjoiaG9vayIsInNpZGUiOiJ0b3AifX19XSxbMCwxLCIiLDEseyJzdHlsZSI6eyJ0YWlsIjp7Im5hbWUiOiJob29rIiwic2lkZSI6InRvcCJ9fX1dLFsxLDMsIiIsMSx7InN0eWxlIjp7InRhaWwiOnsibmFtZSI6Imhvb2siLCJzaWRlIjoidG9wIn19fV1d
\[\begin{tikzcd}
	{\mathcal{A} := f\R G} & {\mathcal{C} := f\C G} \\
	{\mathcal{B} := f\R \widehat{G}} & {\mathcal{D} := f\C\mkern-3mu *\mkern-3mu \widehat{G}}
	\arrow[hook, from=1-1, to=2-1]
	\arrow[hook, from=2-1, to=2-2]
	\arrow[hook, from=1-1, to=1-2]
	\arrow[hook, from=1-2, to=2-2]
\end{tikzcd}\]
where $f$ is one of two central idempotents based on the conjugation action of $\widehat{G}$ on $G$. Here $\C\mkern-3mu *\mkern-3mu \widehat{G}$ is the skew group algebra with basis $\widehat{G}$ and multiplication $ag \cdot bh = a\pi(g)(b)gh$, where $\pi : \widehat{G} \to \widehat{G} / G \cong C_2$ is the quotient map. Each of these algebras is isomorphic to one of $\R$, $\C$, $\mathbb{H}$ and he uses Dyson's theorem \citep{Dyson} to classify the 10 possible structures the $A$-block can take.

\begin{defn}
	For $C_2$ graded groups $G\lhd \widehat{G}$, we further define the \emph{Real Frobenius-Schur indicator} of a $\C G$-module $V$ to be
	\begin{equation*}
		\mathcal{F}_\bold{R}(V)=\mathcal{F}_\bold{R}(\chi):= \frac{1}{\left|G\right|}\sum_{g\in \widehat{G} \setminus G}\chi(g^2)
	\end{equation*}
	
\end{defn}
This variant of the Frobenius Schur indicator is described in \citep{Dyson} and \citep{GM}. We restate the Bargmann-Frobenius-Schur Condition which Taylor proves in \citep{JTDR} using his classification of the possible $A$-blocks and calculating the Frobenius Schur indicator values for each of the 10 possible $A$-blocks.  

\begin{thm}[Bargmann-Frobenius-Schur Condition]
	Let $\chi$ be the character be a simple $\C G$-module $V$. If $W$ is a simple  $\C \! \ast \! \widehat{G}$-module in the same $A$-block as $V$, then
	\begin{equation*}
		\mathcal{F}_\bold{R}(\chi) = \begin{cases*}
			\; \, 1 &\text{ if $W$ is of type $\R$,}\\
			\; \, 0 &\text{ if $W$ is of type $\C$,}\\
			-1 &\text{ if $W$ is of type $\mathbb{H}$.}
		\end{cases*}
	\end{equation*}
\end{thm}

As a result, we have the following corollary.

\begin{cor}
	Together, the pair of Frobenius-Schur indicators $\left(\mathcal{F}_{\C}(\chi), \mathcal{F}_\bold{R}(\chi)\right)$ nearly distinguish the $A$-type of an irreducible $\C G$-module $V$. More explicitly, it distinguishes eight of the types, and leaves ambiguous the case of IV/VII. This may be resolved by considering the conjugation action of $\widehat{G}$ on $G$, and the number of corresponding $\C G$-modules of a given dimension.
	
	\renewcommand{\arraystretch}{1.15}
	\begin{table}[h]
		\centering
		\begin{tabular}{|c|c|c|c|c|c|c|c|c|c|c|}
			\hline
			Type & I                            & II                           & III                          & IV                           & V                            & VI                           & VII                          & VIII                         & IX                           & X                            \\ \hline
			DL & RR & QR & CR & CC2 & RC & QC & CC1 & QQ & RQ & CQ \\  
			$\mathcal{F}_{\C}(\chi)$ & $1$                          & $1$                          & $1$                          & $0$                          & $0$                          & $0$                          & $0$                          & $-1$                         & $-1$                         & $-1$                         \\ 
			$\mathcal{F}_\bold{R}(\chi)$      & $1$                          & $-1$                         & $0$                          & $0$                          & $1$                          & $-1$                         & $0$                          & $-1$                         & $1$                          & $0$                          \\ 
			Colour 			             & \DLI & \DLII & \DLIII & \DLIV & \DLV & \DLVI & \DLVII & \DLVIII & \DLIX & \DLX \\ \hline
		\end{tabular}
		\caption{The Frobenius-Schur indicator values for and irreducible representation of $G\lhd \widehat{G}$. We attach a colouring to the labels for decorating the McKay graphs in a future section.}
	\end{table}
	
	\renewcommand{\arraystretch}{1}
\end{cor}

Each of our index 2 containments gives a $C_2$ grading of groups $G\lhd \widehat{G}$. We calculate the Frobenius-Schur indicators for irreducible representations $G$ to determine the corresponding $A$-block structures. 

\subsection{The Polyhedral Groups}

\begin{eg}
	Let $G=\mathrm{BT}_{24}$ and $\widehat{G}=\mathrm{BO}_{48}$ respectively be the binary tetrahedral and binary octahedral groups. These are the lifts of $\mathrm{T}_{12} \lhd \mathrm{O}_{24} \subset \SO(3)$ the tetrahedral and octahedral groups to $\Spin(3) \simeq \SU(2)$. The character table of $G$ was shown in Table \ref{Table:BT24chartable}.

	we get the following Frobenius-Schur indicators.
	
	\renewcommand{\arraystretch}{1.15}	
	\begin{table}[!htbp]
		\centering
		\begin{tabular}{| c | c c c c c c c |} 
			
			\hline
			$\chi$& $\chi_1$ & $\chi_2$ & $\chi_3$ & $\chi_4$ & $\chi_5$ & $\chi_6$ & $\chi_7$ \\
			\hline
			$\mathcal{F}_\C (\chi)$ & 1 & 0 & 0 & -1 & 0 & 0 & 1   \\
			$\mathcal{F}_\bold{R} (\chi)$    & 1 & 1 & 1 & -1 & -1 & -1 & 1 \\
			\hline
		\end{tabular}
		\caption{The Real and complex Frobenius Schur indicators.}
		\label{Table:BT24FSindicators}
	\end{table}
	
	\renewcommand{\arraystretch}{1}
	
	This allows us to decorate our previously drawn McKay graph with colours corresponding $A$-block type for each representation, see Figure \ref{fig:DecoratedBT24}. We omit the labels but keep respective nodes in the same positions.
	\begin{figure}[!h]
		\[\begin{tikzcd}[column sep=tiny, row sep=tiny]
			\DLV & \DLVI & \DLI & \DLVI & \DLV \\
			&& \DLVIII \\
			&& \DLI
			\arrow[no head, from=1-5, to=1-4]
			\arrow[no head, from=1-2, to=1-1]
			\arrow[no head, from=1-2, to=1-3]
			\arrow[no head, from=1-3, to=1-4]
			\arrow[no head, from=1-3, to=2-3]
			\arrow[no head, from=2-3, to=3-3]
		\end{tikzcd}\]
		\caption{The McKay graph of $\mathrm{BT}_{24}$ decorated with the Dyson A-block structures of $\mathrm{T}_{12} \lhd \mathrm{O}_{24}$.} \label{fig:DecoratedBT24}
	\end{figure}
	
\end{eg}

We do the same for the other containments and tabulate our results below. The code used can be found in \citep{joncheahMAGMA}.

\vspace{1cm}
\renewcommand{\arraystretch}{1.15}
\begin{table}[!htbp]
	\centering
	\begin{tabular}{| c || c | c | c | c | c |}
		\hline
		$\mathrm{BT}_{24} \, \lhd \, \widehat{G} $ & $\mathrm{BO}_{48}$ & $\mathrm{B}_{+}\mathrm{TO}_{48}$ & $\mathrm{B}_{+}\mathrm{2T}_{48}$ & $\mathrm{B}_{-}\mathrm{TO}_{48}$ & $\mathrm{B}_{-}\mathrm{2T}_{48}$ \\
		\hline
		$\rho_1$ & I & I & I & I & I  \\
		$\rho_2$ & V & V & IV & V & IV  \\
		$\rho_3$ & V & V & IV & V & IV  \\
		$\rho_{\text{Nat}} = \rho_4$ & VIII  & VIII & VIII & IX & IX  \\
		$\rho_5$ & VI & VI & IV & V & IV  \\
		$\rho_6$ & VI & VI & IV & V & IV  \\
		$\rho_7$ & I & I & I & I & I  \\
		\hline
	\end{tabular}
	\caption{Dyson A-block structures of the index 2 containments $\mathrm{BT}_{24} \lhd \widehat{G}$.}
	\label{Table:BT24_AblockTypes}
\end{table}
\renewcommand{\arraystretch}{1}

\vspace{2cm}

\renewcommand{\arraystretch}{1.15}
\begin{table}[!htbp]
	\centering
	\begin{tabular}{| c || c | c |}
		\hline
		$\mathrm{BI}_{120} \; \lhd \widehat{G} $ & $\mathrm{B}_{+}\mathrm{2I}_{240}$ & $\mathrm{B}_{-}\mathrm{2I}_{240}$ \\
		\hline
		$\rho_1$ & I & I  \\
		$\rho_{\text{Nat}} = \rho_2$ & VIII & IX  \\
		$\rho_3$ & VIII & IX  \\
		$\rho_4$ & VIII  & IX  \\
		$\rho_5$ & I & I  \\
		$\rho_6$ & I & I  \\
		$\rho_7$ & I & I   \\
		$\rho_8$ & I & I   \\
		$\rho_8$ & VIII & IX   \\
		\hline
	\end{tabular}
	\caption{Dyson A-block structures of the index 2 containments $\mathrm{BI}_{120} \lhd \widehat{G}$.}
	\label{Table:BI120_AblockTypes}
\end{table}
\renewcommand{\arraystretch}{1}

\begin{table}
	\centering
	\begin{tabular}{|c | c|}
		\hline
		$\mathrm{BT}_{24} \, \lhd \, \mathrm{BO}_{48} $ & \begin{tikzcd}[column sep=tiny, row sep=tiny]
			\DLV & \DLVI & \DLI & \DLVI & \DLV \\
			&& \DLVIII \\
			&& \DLI
			\arrow[no head, from=1-5, to=1-4]
			\arrow[no head, from=1-2, to=1-1]
			\arrow[no head, from=1-2, to=1-3]
			\arrow[no head, from=1-3, to=1-4]
			\arrow[no head, from=1-3, to=2-3]
			\arrow[no head, from=2-3, to=3-3]
		\end{tikzcd} \\ \hline
		$\mathrm{BT}_{24} \, \lhd \, \mathrm{B_+TO}_{48} $ & \begin{tikzcd}[column sep=tiny, row sep=tiny]
			\DLV & \DLVI & \DLI & \DLVI & \DLV \\
			&& \DLVIII \\
			&& \DLI
			\arrow[no head, from=1-5, to=1-4]
			\arrow[no head, from=1-2, to=1-1]
			\arrow[no head, from=1-2, to=1-3]
			\arrow[no head, from=1-3, to=1-4]
			\arrow[no head, from=1-3, to=2-3]
			\arrow[no head, from=2-3, to=3-3]
		\end{tikzcd}  \\ \hline
		
		$\mathrm{BT}_{24} \, \lhd \, \mathrm{B_+2T}_{48} $ & \begin{tikzcd}[column sep=tiny, row sep=tiny]
			\DLIV & \DLIV & \DLI & \DLIV & \DLIV \\
			&& \DLVIII \\
			&& \DLI
			\arrow[no head, from=1-5, to=1-4]
			\arrow[no head, from=1-2, to=1-1]
			\arrow[no head, from=1-2, to=1-3]
			\arrow[no head, from=1-3, to=1-4]
			\arrow[no head, from=1-3, to=2-3]
			\arrow[no head, from=2-3, to=3-3]
		\end{tikzcd}  \\ \hline
		
		$\mathrm{BT}_{24} \, \lhd \, \mathrm{B_-TO}_{48} $ & \begin{tikzcd}[column sep=tiny, row sep=tiny]
			\DLV & \DLV & \DLI & \DLV & \DLV \\
			&& \DLIX \\
			&& \DLI
			\arrow[no head, from=1-5, to=1-4]
			\arrow[no head, from=1-2, to=1-1]
			\arrow[no head, from=1-2, to=1-3]
			\arrow[no head, from=1-3, to=1-4]
			\arrow[no head, from=1-3, to=2-3]
			\arrow[no head, from=2-3, to=3-3]
		\end{tikzcd}  \\ \hline
		
		$\mathrm{BT}_{24} \, \lhd \, \mathrm{B_-2T}_{48} $ & \begin{tikzcd}[column sep=tiny, row sep=tiny]
			\DLIV & \DLIV & \DLI & \DLIV & \DLIV \\
			&& \DLIX \\
			&& \DLI
			\arrow[no head, from=1-5, to=1-4]
			\arrow[no head, from=1-2, to=1-1]
			\arrow[no head, from=1-2, to=1-3]
			\arrow[no head, from=1-3, to=1-4]
			\arrow[no head, from=1-3, to=2-3]
			\arrow[no head, from=2-3, to=3-3]
		\end{tikzcd}  \\ \hline
	\end{tabular}
	\caption{McKay graphs of $\mathrm{BT}_{24}$ decorated with the Dyson A-block structures of the index 2 containments $\mathrm{BT}_{24} \lhd \widehat{G}$.} \label{Table:BT24_AblockMcKay}
\end{table}

\begin{table}
	\centering
	\begin{tabular}{|c | c|}
		\hline
		$\mathrm{BI}_{120} \, \lhd \, \mathrm{B_+2I}_{240} $ & \begin{tikzcd}[column sep=small, row sep=small]
			\DLI & \DLVIII & \DLI & \DLVIII & \DLI & \DLVIII & \DLI & \DLVIII \\
			&&&&& \DLI
			\arrow[no head, from=1-7, to=1-6]
			\arrow[no head, from=1-4, to=1-3]
			\arrow[no head, from=1-4, to=1-5]
			\arrow[no head, from=1-3, to=1-2]
			\arrow[no head, from=1-2, to=1-1]
			\arrow[no head, from=1-7, to=1-8]
			\arrow[no head, from=1-6, to=2-6]
			\arrow[no head, from=1-5, to=1-6]
		\end{tikzcd} \\ \hline
		$\mathrm{BI}_{120} \, \lhd \, \mathrm{B_-2I}_{240}  $ & \begin{tikzcd}[column sep=small, row sep=small]
			\DLI & \DLIX & \DLI & \DLIX & \DLI & \DLIX & \DLI & \DLIX \\
			&&&&& \DLI
			\arrow[no head, from=1-7, to=1-6]
			\arrow[no head, from=1-4, to=1-3]
			\arrow[no head, from=1-4, to=1-5]
			\arrow[no head, from=1-3, to=1-2]
			\arrow[no head, from=1-2, to=1-1]
			\arrow[no head, from=1-7, to=1-8]
			\arrow[no head, from=1-6, to=2-6]
			\arrow[no head, from=1-5, to=1-6]
		\end{tikzcd}  \\ \hline
	\end{tabular}
	\caption{Decorated McKay graphs of $\mathrm{BI}_{120}$ with the Dyson A-block structures of the index 2 containments $\mathrm{BI}_{120} \lhd \widehat{G}$.} \label{Table:BI120_AblockMcKay}
\end{table}

\pagebreak

\renewcommand{\arraystretch}{1.15}
\begin{table}[!htbp]
	\centering
	\begin{tabular}{| c || c | c |}
		\hline
		$\mathrm{BO}_{48} \; \lhd \widehat{G} $ & $\mathrm{B}_{+}\mathrm{2O}_{96}$ & $\mathrm{B}_{-}\mathrm{2O}_{96}$ \\
		\hline
		$\rho_1$ & I & I  \\
		$\rho_2$ & I & I  \\
		$\rho_3$ & I & I  \\
		$\rho_{\text{Nat}} = \rho_4$ & VIII  & IX  \\
		$\rho_5$ & VIII & IX  \\
		$\rho_6$ & I & VI  \\
		$\rho_7$ & I & I   \\
		$\rho_8$ & VIII & IX   \\
		\hline
	\end{tabular}
	\caption{Dyson A-block structures of the index 2 containments $\mathrm{BO}_{48} \lhd \widehat{G}$.}
	\label{Table:BO48_AblockTypes}
\end{table}

\renewcommand{\arraystretch}{1}

\vspace{2cm}

%\begin{table}[!htbp]
%	\centering
%	\parbox{.4\linewidth}{
	%		\centering
	%		\begin{tabular}{| c || c | c |}
		%			\hline
		%			$\mathrm{BO}_{48} \; \lhd \widehat{G} $ & $\mathrm{B}_{+}\mathrm{O}_{96}$ & $\mathrm{B}_{-}\mathrm{2O}_{96}$ \\
		%			\hline
		%			$\rho_1$ & I & I  \\
		%			$\rho_2$ & I & I  \\
		%			$\rho_3$ & I & I  \\
		%			$\rho_{\text{Nat}} = \rho_4$ & VIII  & IX  \\
		%			$\rho_5$ & VIII & IX  \\
		%			$\rho_6$ & I & VI  \\
		%			$\rho_7$ & I & I   \\
		%			$\rho_8$ & VIII & IX   \\
		%			\hline
		%		\end{tabular}
	%		\caption{Dyson A-block structures of the index 2 containments $\mathrm{BO}_{48} \lhd \widehat{G}$.		\label{Table:BO48_AblockTypes}}
	%	}
%	\quad
%	\parbox{.4\linewidth}{
	%		\centering
	%		\begin{tabular}{| c || c | c |}
		%			\hline
		%			$\mathrm{BI}_{120} \; \lhd \widehat{G} $ & $\mathrm{B}_{+}\mathrm{2I}_{240}$ & $\mathrm{B}_{-}\mathrm{2I}_{240}$ \\
		%			\hline
		%			$\rho_1$ & I & I  \\
		%			$\rho_{\text{Nat}} = \rho_2$ & VIII & IX  \\
		%			$\rho_3$ & VIII & IX  \\
		%			$\rho_4$ & VIII  & IX  \\
		%			$\rho_5$ & I & I  \\
		%			$\rho_6$ & I & I  \\
		%			$\rho_7$ & I & I   \\
		%			$\rho_8$ & I & I   \\
		%			$\rho_8$ & VIII & IX   \\
		%			\hline
		%		\end{tabular}
	%		\caption{Dyson A-block structures of the index 2 containments $\mathrm{BI}_{120} \lhd \widehat{G}$.		\label{Table:BI120_AblockTypes}}
	%	}
%\end{table}

\begin{table}[!htbp]
	\centering
	\parbox{.4\linewidth}{
		\centering
		\begin{tabular}{|cc|}
			\hline
			\multicolumn{2}{|c|}{$\mathrm{B_{+}TO}_{48} \lhd \mathrm{B}_{+}\mathrm{2O}_{96}$} \\ \hline
			\multicolumn{1}{|c|}{$\rho_1$}                  &        I          \\
			\multicolumn{1}{|c|}{$\rho_2$}                  &        I          \\
			\multicolumn{1}{|c|}{$\rho_3$}                  &        I          \\
			\multicolumn{1}{|c|}{$\rho_4$}                  &        VIII          \\
			\multicolumn{1}{|c|}{$\rho_5$}                  &        VIII         \\
			\multicolumn{1}{|c|}{$\rho_6$}                  &        I          \\
			\multicolumn{1}{|c|}{$\rho_7$}                  &        I          \\
			\multicolumn{1}{|c|}{$\rho_8$}                  &        VIII          \\ \hline
		\end{tabular}
		\caption{The Dyson A-block structures of the containment $\mathrm{B_{+}TO}_{48} \lhd \mathrm{B}_{+}\mathrm{2O}_{96}$.	\label{Table:BO48_BmO96_AblockTypes}}
	}
	\quad
	\parbox{.4\linewidth}{
		\centering
		\begin{tabular}{|cc|}
			\hline
			\multicolumn{2}{|c|}{$\mathrm{B_{-}TO}_{48} \lhd \mathrm{B}_{-}\mathrm{2O}_{96}$} \\ \hline
			\multicolumn{1}{|c|}{$\rho_1$}                  &        I          \\
			\multicolumn{1}{|c|}{$\rho_2$}                  &        I          \\
			\multicolumn{1}{|c|}{$\rho_3$}                  &        I          \\
			\multicolumn{1}{|c|}{$\rho_4$}                  &        IV        \\
			\multicolumn{1}{|c|}{$\rho_5$}                  &        IV         \\
			\multicolumn{1}{|c|}{$\rho_6$}                  &        I          \\
			\multicolumn{1}{|c|}{$\rho_7$}                  &        I          \\
			\multicolumn{1}{|c|}{$\rho_8$}                  &        II          \\ \hline
		\end{tabular}
		\caption{The Dyson A-block structures of the containment $\mathrm{B_{-}TO}_{48} \lhd \mathrm{B}_{-}\mathrm{2O}_{96}$.	\label{Table:BmTO48_BmO96_AblockTypes}}
	}
\end{table}

\vspace{2cm}

\begin{table}[!htbp]
	\centering
	\parbox{.4\linewidth}{
		\centering
		\begin{tabular}{|cc|}
			\hline
			\multicolumn{2}{|c|}{$\mathrm{B_{+}2T}_{48} \lhd \mathrm{B}_{+}\mathrm{2O}_{96}$} \\ \hline
			\multicolumn{1}{|c|}{$\rho_1, \rho_2$}          &        I          \\
			\multicolumn{1}{|c|}{$\rho_3, \rho_4$}          &        V          \\
			\multicolumn{1}{|c|}{$\rho_5, \rho_6$}          &        V          \\
			\multicolumn{1}{|c|}{$\rho_7, \rho_8$}          &        VIII          \\
			\multicolumn{1}{|c|}{$\rho_9, \rho_{10}$}          &        VI         \\
			\multicolumn{1}{|c|}{$\rho_{11}, \rho_{12}$}          &        VI          \\
			\multicolumn{1}{|c|}{$\rho_{13}, \rho_{14}$}          &        I          \\ \hline
		\end{tabular}
		\caption{The Dyson A-block structures of the containment $\mathrm{B_{+}2T}_{24} \lhd \mathrm{B}_{-}\mathrm{2O}_{96}$.	\label{Table:Bp2T48_AblockTypes}}
	}
	\quad
	\parbox{.4\linewidth}{
		\centering
		\begin{tabular}{|cc|}
			\hline
			\multicolumn{2}{|c|}{$\mathrm{B_{-}2T}_{48} \lhd \mathrm{B}_{-}\mathrm{2O}_{96}$} \\ \hline
			\multicolumn{1}{|c|}{$\rho_1, \rho_2$}          &        I          \\
			\multicolumn{1}{|c|}{$\rho_3, \rho_4$}          &        V          \\
			\multicolumn{1}{|c|}{$\rho_5, \rho_6$}          &        V          \\
			\multicolumn{1}{|c|}{$\rho_7, \rho_8$}          &        IV          \\
			\multicolumn{1}{|c|}{$\rho_9, \rho_{10}$}          &       VII         \\
			\multicolumn{1}{|c|}{$\rho_{11}, \rho_{12}$}          &        VII          \\
			\multicolumn{1}{|c|}{$\rho_{13}, \rho_{14}$}          &        I          \\ \hline
		\end{tabular}
		\caption{The Dyson A-block structures of the containment $\mathrm{B_{-}2T}_{48} \lhd \mathrm{B}_{-}\mathrm{2O}_{96}$.	\label{Table:Bm2T48_AblockTypes}}
	}
\end{table}

\pagebreak

As subgroups of $\Orth(3)$, the tetra-octahedral group is isomorphic to the octahedral group. However, one can show via explicit calculations that $\mathrm{B_+ TO}_{48} \cong \mathrm{BO}_{48} \not \cong \mathrm{B_- TO}_{48}$. The conjugacy classes of the latter are of different orders and sizes to the two former groups. Despite this, working explicitly with the character tables gives us isomorphic McKay graphs; all three graphs are the extended Dynkin diagram $\widetilde{E_8}$.
By contrast, the case of the containments $\mathrm{B_\pm 2T}_{48}\lhd \mathrm{B_\pm 2O}_{96}$ is interesting, because the McKay graphs don't fall into the same families as we had seen before. The $\Pin_+(3)$ case has two disjoint copies of the McKay graph, whereas the $\Pin_-(3)$ case has two copies of the McKay graph, with edges in an ``alternating'' configuration. We give the character table of $\mathrm{B_- 2T}_{48}$ in Table  \ref{Table:Bm2T48chartable} below , and refer to the files \verb!Pin3--BmT48 McKayGraph.txt! , \verb!Pin3--BpT48 McKayGraph.txt! \citep{joncheahMAGMA} for the verification of the McKay graphs shown in Figures \ref{fig:McKayGraphBmT48} and \ref{fig:McKayGraphBpT48}.

\vspace{2cm}

\begin{table}[!htbp]
	\centering
	\setlength\tabcolsep{3.5pt}
	\begin{tabular}{|c | c c c c c c c c c c c c c c|}
		\hline
		Class	 & 1 & 2 & 3 & 4 & 5 & 6 & 7 & 8 & 9 & 10 & 11 & 12 & 13 & 14 \\
		Size	 & 1 & 1 & 6 & 4 & 4 & 1 & 1 & 6 & 4 & 4 & 4 & 4 & 4 & 4 \\
		Order	 & 1 & 2 & 2 & 3 & 3 & 4 & 4 & 4 & 6 & 6 & 12 & 12 & 12 & 12\\
		\hline
		$\chi_1$  & 1 & 1 & 1  & 1   & 1   & 1  & 1  & 1 & 1 & 1 & 1 & 1 & 1 & 1 \\
		$\chi_2$  & 1 & 1 & $-1$ &  1  &   1 & $-1$ & $-1$ & 1 &   1 &   1  & $-1$ &  $-1$  & $-1$  & $-1$ \\
		$\chi_3$  & 1 & 1 & $-1$ &$\omega^2$ &   $\omega$ & $-1$ & $-1$ &  1& $\omega^2$&    $\omega$&   $-\omega$&  $-\omega^2$&   $-\omega$&  $-\omega^2$ \\
		$\chi_4$  & 1 & 1 & 1  &$\omega^2$ &   $\omega$ &  1 &  1 & 1 &$\omega^2$ &   $\omega$ &    $\omega$& $\omega^2$&    $\omega$& $\omega^2$ \\
		$\chi_5$  & 1 & 1 &$-1$  &  $\omega$  &$\omega^2$ & $-1$ & $-1$ & 1 &   $\omega$ &$\omega^2$ & $-\omega^2$ &  $-\omega$ & $-\omega^2$ &  $-\omega$ \\
		$\chi_6$  & 1 & 1 & 1  &  $\omega$  &$\omega^2$ &  1 &  1 & 1 &   $\omega$ &$\omega^2$ &$\omega^2$ &   $\omega$ &$\omega^2$ &   $\omega$ \\
		$\chi_7$  & 2 &$-2$& 0  & $-1$  & $-1$  &$2i$ &$-2i$&  0&    1&    1&    $i$&    $i$&   $-i$&   $-i$ \\
		$\chi_8$  & 2 &$-2$& 0  & $-1$  & $-1$  &$-2i$& $2i$&  0&    1&    1&   $-i$&   $-i$&    $i$&    $i$ \\
		$\chi_9$  & 2 &$-2$& 0  & $-\omega$  & $-\omega^2$ &$-2i$& $2i$&  0&    $\omega$&$\omega^2$ &  $\zeta$ & $\zeta^5$&  $-\zeta$& $-\zeta^5$ \\
		$\chi_{10}$ & 2 &$-2$& 0  & $-\omega$  & $-\omega^2$ &$2i$& $-2i$&  0&    $\omega$&$\omega^2$ & $-\zeta$ &$-\zeta^5$&   $\zeta$&  $\zeta^5$ \\
		$\chi_{11}$ & 2 &$-2$& 0  & $-\omega^2$  & $-\omega$ &$-2i$& $2i$&  0& $\omega^2$&   $\omega$ & $\zeta^5$&   $\zeta$&$-\zeta^5$&  $-\zeta$ \\
		$\chi_{12}$ & 2 &$-2$& 0  & $-\omega^2$  & $-\omega$ &$2i$& $-2i$&  0& $\omega^2$&   $\omega$ &$-\zeta^5$&  $-\zeta$& $\zeta^5$&   $\zeta$ \\
		$\chi_{13}$ & 3 & 3 & 1  &  0  &   0 &  $-3$& $-3$ & $-1$&    0&    0&    0&    0&    0&   0  \\
		$\chi_{14}$ & 3 & 3 & $-1$ &  0  &   0 &   3&  3 & $-1$&    0&    0&    0&    0&    0&   0  \\
		\hline
	\end{tabular}
	\caption{Character table of the binary diplo-tetrahedral group $\mathrm{B}_{-}\mathrm{T}_{48}$. Here $\omega,i$ are a primitive third and fourth roots of unity respectively, and $\zeta$ is the twelfth root of unity satisfying $\zeta = i \omega + i$.}
	\label{Table:Bm2T48chartable}
\end{table}

%\begin{figure}[!hbtp]
%	\centering
%	\def\svgwidth{\columnwidth}
%	% https://q.uiver.app/?q=WzAsMTQsWzAsMiwiNiJdLFsxLDMsIjEwIl0sWzIsMiwiMTMiXSxbMywxLCIxMSJdLFs0LDAsIjQiXSxbMSwyLCI5Il0sWzAsMywiNSJdLFsyLDMsIjE0Il0sWzMsMywiOCJdLFszLDQsIjciXSxbNCw0LCIxIl0sWzQsNSwiMiJdLFszLDIsIjEyIl0sWzQsMSwiMyJdLFswLDFdLFsxLDJdLFsyLDNdLFszLDRdLFs1LDBdLFsxLDZdLFs2LDVdLFs1LDddLFs3LDFdLFsyLDVdLFsyLDhdLFs3LDldLFs4LDEwXSxbOSwxMV0sWzcsMTJdLFsxMiwxM10sWzEzLDNdLFs0LDEyXSxbMTIsMl0sWzMsN10sWzgsN10sWzksMl0sWzExLDhdLFsxMCw5XV0=
%	\[\begin{tikzcd}[column sep=large]
	%		&&&& 4 \\
	%		&&& 11 & 3 \\
	%		&& 13 & 12 \\
	%		& 9 & 14 & 8 \\
	%		6 & 10 && 7 & 1 \\
	%		5 &&&& 2
	%		\arrow[from=5-1, to=5-2]
	%		\arrow[from=5-2, to=3-3]
	%		\arrow[from=3-3, to=2-4]
	%		\arrow[from=2-4, to=1-5]
	%		\arrow[from=4-2, to=5-1]
	%		\arrow[from=5-2, to=6-1]
	%		\arrow[from=6-1, to=4-2]
	%		\arrow[from=4-2, to=4-3]
	%		\arrow[from=4-3, to=5-2]
	%		\arrow[from=3-3, to=4-2]
	%		\arrow[from=3-3, to=4-4]
	%		\arrow[from=4-3, to=5-4]
	%		\arrow[from=4-4, to=5-5]
	%		\arrow[from=5-4, to=6-5]
	%		\arrow[from=4-3, to=3-4]
	%		\arrow[from=3-4, to=2-5]
	%		\arrow[from=2-5, to=2-4]
	%		\arrow[from=1-5, to=3-4]
	%		\arrow[from=3-4, to=3-3]
	%		\arrow[from=2-4, to=4-3]
	%		\arrow[from=4-4, to=4-3]
	%		\arrow[from=5-4, to=3-3]
	%		\arrow[from=6-5, to=4-4]
	%		\arrow[from=5-5, to=5-4]
	%		\arrow[from=3-3, to=4-4, crossing over]
	%		\arrow[from=5-4, to=3-3, crossing over]
	%		\arrow[from=4-4, to=4-3]
	%	\end{tikzcd}\]
%	\label{fig:McKayGraphBmT48}
%	\caption{The McKay graph of $\mathrm{B}_{-}\mathrm{T}_{48}$.}
%\end{figure}

\begin{figure}[!hbtp]
	\centering
	\def\svgwidth{\columnwidth}
	% https://q.uiver.app/?q=WzAsMTQsWzEsMCwiNiJdLFsyLDEsIjEwIl0sWzUsMCwiMTMiXSxbNywwLCIxMSJdLFs5LDAsIjQiXSxbMywwLCI5Il0sWzAsMSwiNSJdLFs0LDEsIjE0Il0sWzUsMiwiOCJdLFs0LDMsIjciXSxbNSw0LCIxIl0sWzQsNSwiMiJdLFs2LDEsIjEyIl0sWzgsMSwiMyJdLFswLDFdLFsxLDJdLFsyLDNdLFszLDRdLFs1LDBdLFsxLDZdLFs2LDVdLFs1LDddLFs3LDFdLFsyLDVdLFsyLDhdLFs3LDldLFs4LDEwXSxbOSwxMV0sWzcsMTJdLFsxMiwxM10sWzEzLDNdLFs0LDEyXSxbMTIsMl0sWzMsN10sWzgsN10sWzksMl0sWzExLDhdLFsxMCw5XV0=
	\[\begin{tikzcd}[column sep=small, row sep=small]
		& 6 && 9 && 13 && 11 && 4 \\
		5 && 10 && 14 && 12 && 3 \\
		&&&&& 8 \\
		&&&& 7 \\
		&&&&& 1 \\
		&&&& 2
		\arrow[from=1-2, to=2-3]
		\arrow[from=2-3, to=1-6]
		\arrow[from=1-6, to=1-8]
		\arrow[from=1-8, to=1-10]
		\arrow[from=1-4, to=1-2]
		\arrow[from=2-3, to=2-1]
		\arrow[from=2-1, to=1-4]
		\arrow[from=1-4, to=2-5]
		\arrow[from=2-5, to=2-3]
		\arrow[from=1-6, to=1-4]
		\arrow[from=1-6, to=3-6]
		\arrow[from=2-5, to=4-5]
		\arrow[from=3-6, to=5-6]
		\arrow[from=4-5, to=6-5]
		\arrow[from=2-7, to=2-9]
		\arrow[from=2-9, to=1-8]
		\arrow[from=1-10, to=2-7]
		\arrow[from=3-6, to=2-5]
		\arrow[from=4-5, to=1-6]
		\arrow[from=6-5, to=3-6]
		\arrow[from=5-6, to=4-5]
		\arrow[from=2-5, to=2-7, crossing over]
		\arrow[from=1-8, to=2-5, crossing over]
		\arrow[from=2-7, to=1-6]
	\end{tikzcd}\]
	\caption{The McKay graph of $\mathrm{B}_{-}\mathrm{2T}_{48}$ with $\rho_{\text{Nat}}=\rho_7$.}
	\label{fig:McKayGraphBmT48}
\end{figure}

\begin{figure}[!hbtp]
	\centering
	\def\svgwidth{\columnwidth}
	% https://q.uiver.app/?q=WzAsMTQsWzQsNSwiMSJdLFs0LDMsIjciXSxbNCwxLCIxNCJdLFsyLDEsIjkiXSxbMCwxLCI0Il0sWzYsMSwiMTAiXSxbOCwxLCIzIl0sWzUsNCwiMiJdLFs1LDIsIjgiXSxbNSwwLCIxMyJdLFszLDAsIjExIl0sWzcsMCwiMTIiXSxbOSwwLCI1Il0sWzEsMCwiNiJdLFswLDFdLFsxLDJdLFsyLDNdLFszLDRdLFsyLDVdLFs1LDZdLFs3LDhdLFs4LDldLFs5LDEwXSxbOSwxMV0sWzExLDEyXSxbMTAsMTNdXQ==
	\[\begin{tikzcd}[column sep=small, row sep=small]
		& 6 && 11 && {13} && 12 && 5 \\
		4 && 9 && 14 && 10 && 3 \\
		&&&&& 8 \\
		&&&& 7 \\
		&&&&& 2 \\
		&&&& 1
		\arrow[no head, from=6-5, to=4-5]
		\arrow[no head, from=4-5, to=2-5]
		\arrow[no head, from=2-5, to=2-3]
		\arrow[no head, from=2-3, to=2-1]
		\arrow[no head, from=2-7, to=2-9]
		\arrow[no head, from=5-6, to=3-6]
		\arrow[no head, from=3-6, to=1-6]
		\arrow[no head, from=1-6, to=1-4]
		\arrow[no head, from=1-6, to=1-8]
		\arrow[no head, from=1-8, to=1-10]
		\arrow[no head, from=1-4, to=1-2]
		\arrow[no head, from=2-5, to=2-7, crossing over]
	\end{tikzcd}\]
	\caption{The McKay graph of $\mathrm{B}_{+}\mathrm{2T}_{48}$ with $\rho_{\text{Nat}}=\rho_7$.}
	\label{fig:McKayGraphBpT48}
\end{figure}

\begin{table}[!htbp]
	\centering
	\begin{tabular}{|c | c |}
		\hline
		$\mathrm{BO}_{48} \, \lhd \, \mathrm{B_+{2}O}_{96} $ & \begin{tikzcd}[column sep=small, row sep=small]
			\DLI & \DLVIII & \DLI & \DLVIII & \DLI & \DLVIII & \DLI \\
			&&& \DLI
			\arrow[no head, from=1-5, to=1-4]
			\arrow[no head, from=1-5, to=1-6]
			\arrow[no head, from=1-6, to=1-7]
			\arrow[no head, from=1-4, to=1-3]
			\arrow[no head, from=1-3, to=1-2]
			\arrow[no head, from=1-4, to=2-4]
			\arrow[no head, from=1-2, to=1-1]
		\end{tikzcd} \\ \hline
		$\mathrm{B_+TO}_{48} \, \lhd \, \mathrm{B_+2O}_{96} $ & \begin{tikzcd}[column sep=small, row sep=small]
			\DLI & \DLVIII & \DLI & \DLVIII & \DLI & \DLVIII & \DLI \\
			&&& \DLI
			\arrow[no head, from=1-5, to=1-4]
			\arrow[no head, from=1-5, to=1-6]
			\arrow[no head, from=1-6, to=1-7]
			\arrow[no head, from=1-4, to=1-3]
			\arrow[no head, from=1-3, to=1-2]
			\arrow[no head, from=1-4, to=2-4]
			\arrow[no head, from=1-2, to=1-1]
		\end{tikzcd} \\ \hline

		$\mathrm{B_+2T}_{48} \, \lhd \, \mathrm{B_+2O}_{96}  $ & \begin{tikzcd}[column sep=tiny, row sep=tiny]
			& \DLV && \DLVI && \DLI && \DLVI && \DLV \\
			\DLV && \DLVI && \DLI && \DLVI && \DLV \\
			&&&&& \DLVIII \\
			&&&& \DLVIII \\
			&&&&& \DLI \\
			&&&& \DLI
			\arrow[no head, from=6-5, to=4-5]
			\arrow[no head, from=4-5, to=2-5]
			\arrow[no head, from=2-5, to=2-3]
			\arrow[no head, from=2-3, to=2-1]
			\arrow[no head, from=2-7, to=2-9]
			\arrow[no head, from=5-6, to=3-6]
			\arrow[no head, from=3-6, to=1-6]
			\arrow[no head, from=1-6, to=1-4]
			\arrow[no head, from=1-6, to=1-8]
			\arrow[no head, from=1-8, to=1-10]
			\arrow[no head, from=1-4, to=1-2]
			\arrow[no head, from=2-5, to=2-7, crossing over]
		\end{tikzcd}  \\ \hline

		$\mathrm{BO}_{48} \, \lhd \, \mathrm{B_-2O}_{96}  $ & \begin{tikzcd}[column sep=small, row sep=small]
			\DLI & \DLIX & \DLI & \DLIX & \DLI & \DLIX & \DLI \\
			&&& \DLI
			\arrow[no head, from=1-5, to=1-4]
			\arrow[no head, from=1-5, to=1-6]
			\arrow[no head, from=1-6, to=1-7]
			\arrow[no head, from=1-4, to=1-3]
			\arrow[no head, from=1-3, to=1-2]
			\arrow[no head, from=1-4, to=2-4]
			\arrow[no head, from=1-2, to=1-1]
		\end{tikzcd}  \\ \hline
		
		$\mathrm{B_-TO}_{48} \, \lhd \, \mathrm{B_-2O}_{96}  $ & \begin{tikzcd}[column sep=small, row sep=small]
			\DLI & \DLIV & \DLI & \DLII & \DLI & \DLIV & \DLI \\
			&&& \DLI
			\arrow[no head, from=1-5, to=1-4]
			\arrow[no head, from=1-5, to=1-6]
			\arrow[no head, from=1-6, to=1-7]
			\arrow[no head, from=1-4, to=1-3]
			\arrow[no head, from=1-3, to=1-2]
			\arrow[no head, from=1-4, to=2-4]
			\arrow[no head, from=1-2, to=1-1]
		\end{tikzcd}  \\ \hline
		
		$\mathrm{B_-2T}_{48} \, \lhd \, \mathrm{B_-2O}_{96}  $ & \begin{tikzcd}[column sep=tiny, row sep=tiny]
			& \DLV && \DLVII && \DLI && \DLVII && \DLV \\
			\DLV && \DLVII && \DLI && \DLVII && \DLV \\
			&&&&& \DLIV \\
			&&&& \DLIV \\
			&&&&& \DLI \\
			&&&& \DLI
			\arrow[from=1-2, to=2-3]
			\arrow[from=2-3, to=1-6]
			\arrow[from=1-6, to=1-8]
			\arrow[from=1-8, to=1-10]
			\arrow[from=1-4, to=1-2]
			\arrow[from=2-3, to=2-1]
			\arrow[from=2-1, to=1-4]
			\arrow[from=1-4, to=2-5]
			\arrow[from=2-5, to=2-3]
			\arrow[from=1-6, to=1-4]
			\arrow[from=1-6, to=3-6]
			\arrow[from=2-5, to=4-5]
			\arrow[from=3-6, to=5-6]
			\arrow[from=4-5, to=6-5]
			\arrow[from=2-7, to=2-9]
			\arrow[from=2-9, to=1-8]
			\arrow[from=1-10, to=2-7]
			\arrow[from=3-6, to=2-5]
			\arrow[from=4-5, to=1-6]
			\arrow[from=6-5, to=3-6]
			\arrow[from=5-6, to=4-5]
			\arrow[from=2-5, to=2-7, crossing over]
			\arrow[from=1-8, to=2-5, crossing over]
			\arrow[from=2-7, to=1-6]
		\end{tikzcd}  \\ \hline
		
	\end{tabular}
	\caption{McKay graphs of the index 2 containments into $\mathrm{B_\pm 2O}_{96}$ decorated with the Dyson A-block structures.} \label{Table:BO48_AblockMcKay}
\end{table}

\pagebreak
\subsection{The Axial Groups}

Much of the process here is identical to the polyhedral groups. As we are dealing with infinite families, rather than list tables Dyson labels of for each index $n$, we just present the decorated McKay graphs. As before, our code can be found in \citep{joncheahMAGMA}.

\begin{table}[!hp]
	\centering
	\begin{tabular}{|c | c|}
		\hline
		$\mathrm{C}_{n} \, \lhd \, \mathrm{BC}_{2n} $, $n$ odd &  \begin{tikzcd}[column sep=tiny, row sep=tiny]
			&&&& \DLI \\
			\DLIV & \DLIV & \DLIV &\DLIV && \DLIV& \DLIV & \DLIV & \DLIV
			\arrow[no head, from=2-2, to=2-1]
			\arrow[no head, from=2-4, to=2-6]
			\arrow[no head, from=2-8, to=2-9]
			\arrow[shift left=1, no head, from=2-1, to=1-5]
			\arrow[shift left=1, no head, from=1-5, to=2-9]
			\arrow[no head, from=2-8, to=2-7]
			\arrow[no head, from=2-3, to=2-2]
			\arrow["\cdots"{description}, no head, from=2-7, to=2-6]
			\arrow["\cdots"{description}, no head, from=2-3, to=2-4]
		\end{tikzcd}\\ 
		
		$\mathrm{C}_{n} \, \lhd \, \mathrm{BC}_{2n} $, $n$ even & \begin{tikzcd}[column sep=tiny, row sep=tiny]
			&&&& \DLI \\
			\DLIV & \DLIV & \DLIV &\DLIV & \DLII & \DLIV& \DLIV & \DLIV & \DLIV
			\arrow[no head, from=2-5, to=2-4]
			\arrow[no head, from=2-2, to=2-1]
			\arrow[no head, from=2-5, to=2-6]
			\arrow[no head, from=2-8, to=2-9]
			\arrow[shift left=1, no head, from=2-1, to=1-5]
			\arrow[shift left=1, no head, from=1-5, to=2-9]
			\arrow[no head, from=2-8, to=2-7]
			\arrow[no head, from=2-3, to=2-2]
			\arrow["\cdots"{description}, no head, from=2-7, to=2-6]
			\arrow["\cdots"{description}, no head, from=2-3, to=2-4]
		\end{tikzcd}  \\ \hline 
		
		$\mathrm{BC}_{2n} \, \lhd \, \mathrm{BD}_{4n} $, $n$ odd & \begin{tikzcd}[column sep=tiny, row sep=tiny]
			&&&& \DLI \\
			\DLVI & \DLV & \DLVI &\DLV & \DLII & \DLV& \DLVI & \DLV & \DLVI
			\arrow[no head, from=2-5, to=2-4]
			\arrow[no head, from=2-2, to=2-1]
			\arrow[no head, from=2-5, to=2-6]
			\arrow[no head, from=2-8, to=2-9]
			\arrow[shift left=1, no head, from=2-1, to=1-5]
			\arrow[shift left=1, no head, from=1-5, to=2-9]
			\arrow[no head, from=2-8, to=2-7]
			\arrow[no head, from=2-3, to=2-2]
			\arrow["\cdots"{description}, no head, from=2-7, to=2-6]
			\arrow["\cdots"{description}, no head, from=2-3, to=2-4]
		\end{tikzcd}  \\ 
		
		$\mathrm{BC}_{2n} \, \lhd \, \mathrm{BD}_{4n} $, $n$ even & \begin{tikzcd}[column sep=tiny, row sep=tiny]
			&&&& \DLI \\
			\DLVI & \DLV & \DLVI &\DLVI & \DLI & \DLVI& \DLVI & \DLV & \DLVI
			\arrow[no head, from=2-5, to=2-4]
			\arrow[no head, from=2-2, to=2-1]
			\arrow[no head, from=2-5, to=2-6]
			\arrow[no head, from=2-8, to=2-9]
			\arrow[shift left=1, no head, from=2-1, to=1-5]
			\arrow[shift left=1, no head, from=1-5, to=2-9]
			\arrow[no head, from=2-8, to=2-7]
			\arrow[no head, from=2-3, to=2-2]
			\arrow["\cdots"{description}, no head, from=2-7, to=2-6]
			\arrow["\cdots"{description}, no head, from=2-3, to=2-4]
		\end{tikzcd}  \\ \hline
		
		$\mathrm{BC}_{2n} \, \lhd \, \mathrm{B_+CD}_{4n} $ & \begin{tikzcd}[column sep=tiny, row sep=tiny]
			&&&& \DLI \\
			\DLVI & \DLV & \DLVI &\DLV & \DLII & \DLV& \DLVI & \DLV & \DLVI
			\arrow[no head, from=2-5, to=2-4]
			\arrow[no head, from=2-2, to=2-1]
			\arrow[no head, from=2-5, to=2-6]
			\arrow[no head, from=2-8, to=2-9]
			\arrow[shift left=1, no head, from=2-1, to=1-5]
			\arrow[shift left=1, no head, from=1-5, to=2-9]
			\arrow[no head, from=2-8, to=2-7]
			\arrow[no head, from=2-3, to=2-2]
			\arrow["\cdots"{description}, no head, from=2-7, to=2-6]
			\arrow["\cdots"{description}, no head, from=2-3, to=2-4]
		\end{tikzcd}  \\ \hline
		
		$\mathrm{BC}_{2n} \, \lhd \, \mathrm{B_-CD}_{4n} $ & \begin{tikzcd}[column sep=tiny, row sep=tiny]
			&&&& \DLI \\
			\DLV & \DLV & \DLV &\DLV & \DLII & \DLV& \DLV & \DLV & \DLV
			\arrow[no head, from=2-5, to=2-4]
			\arrow[no head, from=2-2, to=2-1]
			\arrow[no head, from=2-5, to=2-6]
			\arrow[no head, from=2-8, to=2-9]
			\arrow[shift left=1, no head, from=2-1, to=1-5]
			\arrow[shift left=1, no head, from=1-5, to=2-9]
			\arrow[no head, from=2-8, to=2-7]
			\arrow[no head, from=2-3, to=2-2]
			\arrow["\cdots"{description}, no head, from=2-7, to=2-6]
			\arrow["\cdots"{description}, no head, from=2-3, to=2-4]
		\end{tikzcd}  \\ \hline

		$\mathrm{BC}_{2n} \, \lhd \, \mathrm{B_+CC}_{4n} $ & \begin{tikzcd}[column sep=tiny, row sep=tiny]
			&&&& \DLI \\
			\DLIV & \DLIV & \DLIV & \DLIV & \DLII & \DLIV & \DLIV & \DLIV & \DLIV
			\arrow[no head, from=2-5, to=2-4]
			\arrow[no head, from=2-2, to=2-1]
			\arrow[no head, from=2-5, to=2-6]
			\arrow[no head, from=2-8, to=2-9]
			\arrow[shift left=1, no head, from=2-1, to=1-5]
			\arrow[shift left=1, no head, from=1-5, to=2-9]
			\arrow[no head, from=2-8, to=2-7]
			\arrow[no head, from=2-3, to=2-2]
			\arrow["\cdots"{description}, no head, from=2-7, to=2-6]
			\arrow["\cdots"{description}, no head, from=2-3, to=2-4]
		\end{tikzcd}  \\ \hline

		$\mathrm{BC}_{2n} \, \lhd \, \mathrm{B_+2C}_{4n} $ & \begin{tikzcd}[column sep=tiny, row sep=tiny]
			&&&& \DLI \\
			\DLIV & \DLIV & \DLIV & \DLIV & \DLI & \DLIV & \DLIV & \DLIV & \DLIV
			\arrow[no head, from=2-5, to=2-4]
			\arrow[no head, from=2-2, to=2-1]
			\arrow[no head, from=2-5, to=2-6]
			\arrow[no head, from=2-8, to=2-9]
			\arrow[shift left=1, no head, from=2-1, to=1-5]
			\arrow[shift left=1, no head, from=1-5, to=2-9]
			\arrow[no head, from=2-8, to=2-7]
			\arrow[no head, from=2-3, to=2-2]
			\arrow["\cdots"{description}, no head, from=2-7, to=2-6]
			\arrow["\cdots"{description}, no head, from=2-3, to=2-4]
		\end{tikzcd}  \\ \hline 
		
		$\mathrm{BC}_{2n} \, \lhd \, \mathrm{B_-CC}_{4n} $, $n$ odd & \begin{tikzcd}[column sep=tiny, row sep=tiny]
			&&&& \DLI \\
			\DLIV & \DLIV & \DLIV & \DLIV & \DLI & \DLIV & \DLIV & \DLIV & \DLIV
			\arrow[no head, from=2-5, to=2-4]
			\arrow[no head, from=2-2, to=2-1]
			\arrow[no head, from=2-5, to=2-6]
			\arrow[no head, from=2-8, to=2-9]
			\arrow[shift left=1, no head, from=2-1, to=1-5]
			\arrow[shift left=1, no head, from=1-5, to=2-9]
			\arrow[no head, from=2-8, to=2-7]
			\arrow[no head, from=2-3, to=2-2]
			\arrow["\cdots"{description}, no head, from=2-7, to=2-6]
			\arrow["\cdots"{description}, no head, from=2-3, to=2-4]
		\end{tikzcd}  \\

		$\mathrm{BC}_{2n} \, \lhd \, \mathrm{B_-CC}_{4n} $, $n$ even & \begin{tikzcd}[column sep=tiny, row sep=tiny]
			&&&& \DLI \\
			\DLIV & \DLIV & \DLIV & \DLIV & \DLII & \DLIV & \DLIV & \DLIV & \DLIV
			\arrow[no head, from=2-5, to=2-4]
			\arrow[no head, from=2-2, to=2-1]
			\arrow[no head, from=2-5, to=2-6]
			\arrow[no head, from=2-8, to=2-9]
			\arrow[shift left=1, no head, from=2-1, to=1-5]
			\arrow[shift left=1, no head, from=1-5, to=2-9]
			\arrow[no head, from=2-8, to=2-7]
			\arrow[no head, from=2-3, to=2-2]
			\arrow["\cdots"{description}, no head, from=2-7, to=2-6]
			\arrow["\cdots"{description}, no head, from=2-3, to=2-4]
		\end{tikzcd}  \\ \hline
		
		$\mathrm{BC}_{2n} \, \lhd \, \mathrm{B_-2C}_{4n} $, $n$ odd & \begin{tikzcd}[column sep=tiny, row sep=tiny]
			&&&& \DLI \\
			\DLIV & \DLIV & \DLIV & \DLIV & \DLII & \DLIV & \DLIV & \DLIV & \DLIV
			\arrow[no head, from=2-5, to=2-4]
			\arrow[no head, from=2-2, to=2-1]
			\arrow[no head, from=2-5, to=2-6]
			\arrow[no head, from=2-8, to=2-9]
			\arrow[shift left=1, no head, from=2-1, to=1-5]
			\arrow[shift left=1, no head, from=1-5, to=2-9]
			\arrow[no head, from=2-8, to=2-7]
			\arrow[no head, from=2-3, to=2-2]
			\arrow["\cdots"{description}, no head, from=2-7, to=2-6]
			\arrow["\cdots"{description}, no head, from=2-3, to=2-4]
		\end{tikzcd}  \\

		$\mathrm{BC}_{2n} \, \lhd \, \mathrm{B_-2C}_{4n} $, $n$ even & \begin{tikzcd}[column sep=tiny, row sep=tiny]
			&&&& \DLI \\
			\DLIV & \DLIV & \DLIV & \DLIV & \DLI & \DLIV & \DLIV & \DLIV & \DLIV
			\arrow[no head, from=2-5, to=2-4]
			\arrow[no head, from=2-2, to=2-1]
			\arrow[no head, from=2-5, to=2-6]
			\arrow[no head, from=2-8, to=2-9]
			\arrow[shift left=1, no head, from=2-1, to=1-5]
			\arrow[shift left=1, no head, from=1-5, to=2-9]
			\arrow[no head, from=2-8, to=2-7]
			\arrow[no head, from=2-3, to=2-2]
			\arrow["\cdots"{description}, no head, from=2-7, to=2-6]
			\arrow["\cdots"{description}, no head, from=2-3, to=2-4]
		\end{tikzcd}  \\ \hline
		
	\end{tabular}
	\caption{Decorated McKay graphs of $\mathrm{BC}_{2n}$ with the Dyson A-block structures of the index 2 containments $\mathrm{BC}_{2n} \lhd \widehat{G}$.} \label{Table:BC2nAxialAblockMcKay}
\end{table}

\begin{table}[!p]
	\centering
	\begin{tabular}{|c | c|}
		\hline
		
		$\mathrm{BD}_{4n} \, \lhd \, \mathrm{B_+DD}_{8n} $, $n$ odd & \begin{tikzcd}[column sep=tiny, row sep=tiny]
			\DLI &&&&&&&& \DLVI \\
			& \DLVIII & \DLI & \DLVIII & \cdots & \DLI & \DLVIII & \DLI \\
			\DLI &&&&&&&& \DLVI
			\arrow[no head, from=2-7, to=2-6]
			\arrow[no head, from=2-4, to=2-3]
			\arrow[no head, from=2-3, to=2-2]
			\arrow[no head, from=2-7, to=2-8]
			\arrow[no head, from=2-4, to=2-5]
			\arrow[no head, from=2-5, to=2-6]
			\arrow[no head, from=2-8, to=1-9]
			\arrow[no head, from=2-8, to=3-9]
			\arrow[no head, from=2-2, to=1-1]
			\arrow[no head, from=2-2, to=3-1]
		\end{tikzcd}  \\
		
		$\mathrm{BD}_{4n} \, \lhd \, \mathrm{B_+DD}_{8n} $, $n$ even & \begin{tikzcd}[column sep=tiny, row sep=tiny]
			\DLI &&&&&&&& \DLIII \\
			& \DLVIII & \DLI & \DLVIII & \cdots & \DLVIII & \DLI & \DLVIII \\
			\DLI &&&&&&&& \DLIII
			\arrow[no head, from=2-7, to=2-6]
			\arrow[no head, from=2-4, to=2-3]
			\arrow[no head, from=2-3, to=2-2]
			\arrow[no head, from=2-7, to=2-8]
			\arrow[no head, from=2-4, to=2-5]
			\arrow[no head, from=2-5, to=2-6]
			\arrow[no head, from=2-8, to=1-9]
			\arrow[no head, from=2-8, to=3-9]
			\arrow[no head, from=2-2, to=1-1]
			\arrow[no head, from=2-2, to=3-1]
		\end{tikzcd}  \\ \hline

		$\mathrm{B_+CD}_{4n} \, \lhd \, \mathrm{B_+DD}_{8n} $, $n$ odd & \begin{tikzcd}[column sep=tiny, row sep=tiny]
			\DLI &&&&&&&& \DLVI \\
			& \DLVIII & \DLI & \DLVIII & \cdots & \DLI & \DLVIII & \DLI \\
			\DLI &&&&&&&& \DLVI
			\arrow[no head, from=2-7, to=2-6]
			\arrow[no head, from=2-4, to=2-3]
			\arrow[no head, from=2-3, to=2-2]
			\arrow[no head, from=2-7, to=2-8]
			\arrow[no head, from=2-4, to=2-5]
			\arrow[no head, from=2-5, to=2-6]
			\arrow[no head, from=2-8, to=1-9]
			\arrow[no head, from=2-8, to=3-9]
			\arrow[no head, from=2-2, to=1-1]
			\arrow[no head, from=2-2, to=3-1]
		\end{tikzcd}  \\
		
		$\mathrm{B_+CD}_{4n} \, \lhd \, \mathrm{B_+DD}_{8n} $, $n$ even & \begin{tikzcd}[column sep=tiny, row sep=tiny]
			\DLI &&&&&&&& \DLIII \\
			& \DLVIII & \DLI & \DLVIII & \cdots & \DLVIII & \DLI & \DLVIII \\
			\DLI &&&&&&&& \DLIII
			\arrow[no head, from=2-7, to=2-6]
			\arrow[no head, from=2-4, to=2-3]
			\arrow[no head, from=2-3, to=2-2]
			\arrow[no head, from=2-7, to=2-8]
			\arrow[no head, from=2-4, to=2-5]
			\arrow[no head, from=2-5, to=2-6]
			\arrow[no head, from=2-8, to=1-9]
			\arrow[no head, from=2-8, to=3-9]
			\arrow[no head, from=2-2, to=1-1]
			\arrow[no head, from=2-2, to=3-1]
		\end{tikzcd}  \\ \hline
		$\mathrm{B_{+}CC}_{4n} \, \lhd \, \mathrm{B_{+}DD}_{8n} $ & \begin{tikzcd}[column sep=tiny, row sep=tiny]
			&&&& \DLI \\
			\DLVI & \DLV & \DLVI &\DLVI & \DLI & \DLVI& \DLVI & \DLV & \DLVI
			\arrow[no head, from=2-5, to=2-4]
			\arrow[no head, from=2-2, to=2-1]
			\arrow[no head, from=2-5, to=2-6]
			\arrow[no head, from=2-8, to=2-9]
			\arrow[shift left=1, no head, from=2-1, to=1-5]
			\arrow[shift left=1, no head, from=1-5, to=2-9]
			\arrow[no head, from=2-8, to=2-7]
			\arrow[no head, from=2-3, to=2-2]
			\arrow["\cdots"{description}, no head, from=2-7, to=2-6]
			\arrow["\cdots"{description}, no head, from=2-3, to=2-4]
		\end{tikzcd}  \\ \hline
		
		$\mathrm{BD}_{4n} \, \lhd \, \mathrm{B_-DD}_{8n} $, $n$ odd & \begin{tikzcd}[column sep=tiny, row sep=tiny]
			\DLI &&&&&&&& \DLV \\
			& \DLIX & \DLI & \DLIX & \cdots & \DLI & \DLIX & \DLI \\
			\DLI &&&&&&&& \DLV
			\arrow[no head, from=2-7, to=2-6]
			\arrow[no head, from=2-4, to=2-3]
			\arrow[no head, from=2-3, to=2-2]
			\arrow[no head, from=2-7, to=2-8]
			\arrow[no head, from=2-4, to=2-5]
			\arrow[no head, from=2-5, to=2-6]
			\arrow[no head, from=2-8, to=1-9]
			\arrow[no head, from=2-8, to=3-9]
			\arrow[no head, from=2-2, to=1-1]
			\arrow[no head, from=2-2, to=3-1]
		\end{tikzcd}  \\
		
		$\mathrm{BD}_{4n} \, \lhd \, \mathrm{B_-DD}_{8n} $, $n$ even & \begin{tikzcd}[column sep=tiny, row sep=tiny]
			\DLI &&&&&&&& \DLIII \\
			& \DLIX & \DLI & \DLIX & \cdots & \DLIX & \DLI & \DLIX \\
			\DLI &&&&&&&& \DLIII
			\arrow[no head, from=2-7, to=2-6]
			\arrow[no head, from=2-4, to=2-3]
			\arrow[no head, from=2-3, to=2-2]
			\arrow[no head, from=2-7, to=2-8]
			\arrow[no head, from=2-4, to=2-5]
			\arrow[no head, from=2-5, to=2-6]
			\arrow[no head, from=2-8, to=1-9]
			\arrow[no head, from=2-8, to=3-9]
			\arrow[no head, from=2-2, to=1-1]
			\arrow[no head, from=2-2, to=3-1]
		\end{tikzcd}  \\ \hline

		$\mathrm{B_-CD}_{4n} \, \lhd \, \mathrm{B_-DD}_{8n} $, $n$ odd & \begin{tikzcd}[column sep=tiny, row sep=tiny]
			\DLI &&&&&&&& \DLIII \\
			& \DLII & \DLI & \DLII & \cdots & \DLI & \DLII & \DLI \\
			\DLI &&&&&&&& \DLIII
			\arrow[no head, from=2-7, to=2-6]
			\arrow[no head, from=2-4, to=2-3]
			\arrow[no head, from=2-3, to=2-2]
			\arrow[no head, from=2-7, to=2-8]
			\arrow[no head, from=2-4, to=2-5]
			\arrow[no head, from=2-5, to=2-6]
			\arrow[no head, from=2-8, to=1-9]
			\arrow[no head, from=2-8, to=3-9]
			\arrow[no head, from=2-2, to=1-1]
			\arrow[no head, from=2-2, to=3-1]
		\end{tikzcd}  \\
		
		$\mathrm{B_-CD}_{4n} \, \lhd \, \mathrm{B_-DD}_{8n} $, $n$ even & \begin{tikzcd}[column sep=tiny, row sep=tiny]
			\DLI &&&&&&&& \DLIII \\
			& \DLII & \DLI & \DLII & \cdots & \DLII & \DLI & \DLII \\
			\DLI &&&&&&&& \DLIII
			\arrow[no head, from=2-7, to=2-6]
			\arrow[no head, from=2-4, to=2-3]
			\arrow[no head, from=2-3, to=2-2]
			\arrow[no head, from=2-7, to=2-8]
			\arrow[no head, from=2-4, to=2-5]
			\arrow[no head, from=2-5, to=2-6]
			\arrow[no head, from=2-8, to=1-9]
			\arrow[no head, from=2-8, to=3-9]
			\arrow[no head, from=2-2, to=1-1]
			\arrow[no head, from=2-2, to=3-1]
		\end{tikzcd}  \\ \hline

		$\mathrm{B_{-}CC}_{4n} \, \lhd \, \mathrm{B_{-}DD}_{8n} $, $n$ odd & \begin{tikzcd}[column sep=tiny, row sep=tiny]
			&&&& \DLI \\
			\DLVII & \DLV & \DLVII &\DLV & \DLIII & \DLV& \DLVII & \DLV & \DLVII
			\arrow[dashed, no head, from=2-5, to=2-4]
			\arrow[dashed, no head, from=2-2, to=2-1]
			\arrow[dashed, no head, from=2-5, to=2-6]
			\arrow[dashed, no head, from=2-8, to=2-9]
			\arrow[dashed, shift left=1, no head, from=2-1, to=1-5]
			\arrow[dashed, shift left=1, no head, from=1-5, to=2-9]
			\arrow[dashed, no head, from=2-8, to=2-7]
			\arrow[dashed, no head, from=2-3, to=2-2]
			\arrow[dashed, "\cdots"{description}, no head, from=2-7, to=2-6]
			\arrow[dashed, "\cdots"{description}, no head, from=2-3, to=2-4]
		\end{tikzcd}  \\ 
		
		$\mathrm{B_{-}CC}_{4n} \, \lhd \, \mathrm{B_{-}DD}_{8n} $, $n$ even & \begin{tikzcd}[column sep=tiny, row sep=tiny]
			&&&& \DLI \\
			\DLVII & \DLV & \DLVII &\DLVII & \DLV & \DLVII& \DLVII & \DLV & \DLVII
			\arrow[dashed, no head, from=2-5, to=2-4]
			\arrow[dashed, no head, from=2-2, to=2-1]
			\arrow[dashed, no head, from=2-5, to=2-6]
			\arrow[dashed, no head, from=2-8, to=2-9]
			\arrow[dashed, shift left=1, no head, from=2-1, to=1-5]
			\arrow[dashed, shift left=1, no head, from=1-5, to=2-9]
			\arrow[dashed, no head, from=2-8, to=2-7]
			\arrow[dashed, no head, from=2-3, to=2-2]
			\arrow[dashed, "\cdots"{description}, no head, from=2-7, to=2-6]
			\arrow[dashed, "\cdots"{description}, no head, from=2-3, to=2-4]
		\end{tikzcd}  \\ \hline
	\end{tabular}
	\caption{McKay graphs of the index 2 containments into $\mathrm{B_\pm DD}_{4n}$ decorated with the Dyson A-block structures.The dashed lines in the $\mathrm{B_-CC}_{4n} \lhd \mathrm{B_-DD}_{8n}$ containment are used to indicate that there are two components of the McKay graph with the edges between corresponding nodes in a similar configuration to what occurs in Figure \ref{fig:McKayGraphBmT48}.} \label{Table:BDD8nAxialAblockMcKay}
\end{table}

\begin{table}[!p]
	\centering
	\begin{tabular}{|c | c|}
		\hline
		
		$\mathrm{BD}_{4n} \, \lhd \, \mathrm{B_+2D}_{8n} $, $n$ odd & \begin{tikzcd}[column sep=tiny, row sep=tiny]
			\DLI &&&&&&&& \DLIV \\
			& \DLVIII & \DLI & \DLVIII & \cdots & \DLI & \DLVIII & \DLI \\
			\DLI &&&&&&&& \DLIV
			\arrow[no head, from=2-7, to=2-6]
			\arrow[no head, from=2-4, to=2-3]
			\arrow[no head, from=2-3, to=2-2]
			\arrow[no head, from=2-7, to=2-8]
			\arrow[no head, from=2-4, to=2-5]
			\arrow[no head, from=2-5, to=2-6]
			\arrow[no head, from=2-8, to=1-9]
			\arrow[no head, from=2-8, to=3-9]
			\arrow[no head, from=2-2, to=1-1]
			\arrow[no head, from=2-2, to=3-1]
		\end{tikzcd}  \\
		
		$\mathrm{BD}_{4n} \, \lhd \, \mathrm{B_+2D}_{8n} $, $n$ even & \begin{tikzcd}[column sep=tiny, row sep=tiny]
			\DLI &&&&&&&& \DLI \\
			& \DLVIII & \DLI & \DLVIII & \cdots & \DLVIII & \DLI & \DLVIII \\
			\DLI &&&&&&&& \DLI
			\arrow[no head, from=2-7, to=2-6]
			\arrow[no head, from=2-4, to=2-3]
			\arrow[no head, from=2-3, to=2-2]
			\arrow[no head, from=2-7, to=2-8]
			\arrow[no head, from=2-4, to=2-5]
			\arrow[no head, from=2-5, to=2-6]
			\arrow[no head, from=2-8, to=1-9]
			\arrow[no head, from=2-8, to=3-9]
			\arrow[no head, from=2-2, to=1-1]
			\arrow[no head, from=2-2, to=3-1]
		\end{tikzcd}  \\ \hline

		$\mathrm{B_+CD}_{4n} \, \lhd \, \mathrm{B_+2D}_{8n} $, $n$ odd & \begin{tikzcd}[column sep=tiny, row sep=tiny]
			\DLI &&&&&&&& \DLIV \\
			& \DLVIII & \DLI & \DLVIII & \cdots & \DLVIII & \DLI & \DLVIII \\
			\DLI &&&&&&&& \DLIV
			\arrow[no head, from=2-7, to=2-6]
			\arrow[no head, from=2-4, to=2-3]
			\arrow[no head, from=2-3, to=2-2]
			\arrow[no head, from=2-7, to=2-8]
			\arrow[no head, from=2-4, to=2-5]
			\arrow[no head, from=2-5, to=2-6]
			\arrow[no head, from=2-8, to=1-9]
			\arrow[no head, from=2-8, to=3-9]
			\arrow[no head, from=2-2, to=1-1]
			\arrow[no head, from=2-2, to=3-1]
		\end{tikzcd}  \\
		
		$\mathrm{B_+CD}_{4n} \, \lhd \, \mathrm{B_+2D}_{8n} $, $n$ even & \begin{tikzcd}[column sep=tiny, row sep=tiny]
			\DLI &&&&&&&& \DLI \\
			& \DLVIII & \DLI & \DLVIII & \cdots & \DLVIII & \DLI & \DLVIII \\
			\DLI &&&&&&&& \DLI
			\arrow[no head, from=2-7, to=2-6]
			\arrow[no head, from=2-4, to=2-3]
			\arrow[no head, from=2-3, to=2-2]
			\arrow[no head, from=2-7, to=2-8]
			\arrow[no head, from=2-4, to=2-5]
			\arrow[no head, from=2-5, to=2-6]
			\arrow[no head, from=2-8, to=1-9]
			\arrow[no head, from=2-8, to=3-9]
			\arrow[no head, from=2-2, to=1-1]
			\arrow[no head, from=2-2, to=3-1]
		\end{tikzcd}  \\ \hline
			$\mathrm{B_{+}2C}_{4n} \, \lhd \, \mathrm{B_{+}2D}_{8n} $, $n$ odd & \begin{tikzcd}[column sep=tiny, row sep=tiny]
			&&&& \DLI \\
			\DLVI & \DLV & \DLVI &\DLV & \DLII & \DLV& \DLVI & \DLV & \DLVI
			\arrow[Rightarrow, no head, from=2-5, to=2-4]
			\arrow[Rightarrow, no head, from=2-2, to=2-1]
			\arrow[Rightarrow, no head, from=2-5, to=2-6]
			\arrow[Rightarrow, no head, from=2-8, to=2-9]
			\arrow[Rightarrow, shift left=1, no head, from=2-1, to=1-5]
			\arrow[Rightarrow, shift left=1, no head, from=1-5, to=2-9]
			\arrow[Rightarrow, no head, from=2-8, to=2-7]
			\arrow[Rightarrow, no head, from=2-3, to=2-2]
			\arrow[Rightarrow, "\cdots"{description}, no head, from=2-7, to=2-6]
			\arrow[Rightarrow, "\cdots"{description}, no head, from=2-3, to=2-4]
		\end{tikzcd}  \\ 
		
		$\mathrm{B_{+}2C}_{4n} \, \lhd \, \mathrm{B_{+}2D}_{8n} $, $n$ even & \begin{tikzcd}[column sep=tiny, row sep=tiny]
			&&&& \DLI \\
			\DLVI & \DLV & \DLVI &\DLVI & \DLI & \DLVI& \DLVI & \DLV & \DLVI
			\arrow[Rightarrow, no head, from=2-5, to=2-4]
			\arrow[Rightarrow, no head, from=2-2, to=2-1]
			\arrow[Rightarrow, no head, from=2-5, to=2-6]
			\arrow[Rightarrow, no head, from=2-8, to=2-9]
			\arrow[Rightarrow, shift left=1, no head, from=2-1, to=1-5]
			\arrow[Rightarrow, shift left=1, no head, from=1-5, to=2-9]
			\arrow[Rightarrow, no head, from=2-8, to=2-7]
			\arrow[Rightarrow, no head, from=2-3, to=2-2]
			\arrow[Rightarrow, "\cdots"{description}, no head, from=2-7, to=2-6]
			\arrow[Rightarrow, "\cdots"{description}, no head, from=2-3, to=2-4]
		\end{tikzcd}  \\ \hline
		
		$\mathrm{BD}_{4n} \, \lhd \, \mathrm{B_-2D}_{8n} $, $n$ odd & \begin{tikzcd}[column sep=tiny, row sep=tiny]
			\DLI &&&&&&&& \DLIV \\
			& \DLIX & \DLI & \DLIX & \cdots & \DLI & \DLIX & \DLI \\
			\DLI &&&&&&&& \DLIV
			\arrow[no head, from=2-7, to=2-6]
			\arrow[no head, from=2-4, to=2-3]
			\arrow[no head, from=2-3, to=2-2]
			\arrow[no head, from=2-7, to=2-8]
			\arrow[no head, from=2-4, to=2-5]
			\arrow[no head, from=2-5, to=2-6]
			\arrow[no head, from=2-8, to=1-9]
			\arrow[no head, from=2-8, to=3-9]
			\arrow[no head, from=2-2, to=1-1]
			\arrow[no head, from=2-2, to=3-1]
		\end{tikzcd}  \\
		
		$\mathrm{BD}_{4n} \, \lhd \, \mathrm{B_-2D}_{8n} $, $n$ even & \begin{tikzcd}[column sep=tiny, row sep=tiny]
			\DLI &&&&&&&& \DLI \\
			& \DLIX & \DLI & \DLIX & \cdots & \DLIX & \DLI & \DLIX \\
			\DLI &&&&&&&& \DLI
			\arrow[no head, from=2-7, to=2-6]
			\arrow[no head, from=2-4, to=2-3]
			\arrow[no head, from=2-3, to=2-2]
			\arrow[no head, from=2-7, to=2-8]
			\arrow[no head, from=2-4, to=2-5]
			\arrow[no head, from=2-5, to=2-6]
			\arrow[no head, from=2-8, to=1-9]
			\arrow[no head, from=2-8, to=3-9]
			\arrow[no head, from=2-2, to=1-1]
			\arrow[no head, from=2-2, to=3-1]
		\end{tikzcd}  \\ \hline

		$\mathrm{B_-CD}_{4n} \, \lhd \, \mathrm{B_-2D}_{8n} $, $n$ odd & \begin{tikzcd}[column sep=tiny, row sep=tiny]
			\DLI &&&&&&&& \DLII \\
			& \DLII & \DLI & \DLII & \cdots & \DLI & \DLII & \DLI \\
			\DLI &&&&&&&& \DLII
			\arrow[no head, from=2-7, to=2-6]
			\arrow[no head, from=2-4, to=2-3]
			\arrow[no head, from=2-3, to=2-2]
			\arrow[no head, from=2-7, to=2-8]
			\arrow[no head, from=2-4, to=2-5]
			\arrow[no head, from=2-5, to=2-6]
			\arrow[no head, from=2-8, to=1-9]
			\arrow[no head, from=2-8, to=3-9]
			\arrow[no head, from=2-2, to=1-1]
			\arrow[no head, from=2-2, to=3-1]
		\end{tikzcd}  \\
		
		$\mathrm{B_-CD}_{4n} \, \lhd \, \mathrm{B_-2D}_{8n} $, $n$ even & \begin{tikzcd}[column sep=tiny, row sep=tiny]
			\DLI &&&&&&&& \DLI \\
			& \DLII & \DLI & \DLII & \cdots & \DLII & \DLI & \DLII \\
			\DLI &&&&&&&& \DLI
			\arrow[no head, from=2-7, to=2-6]
			\arrow[no head, from=2-4, to=2-3]
			\arrow[no head, from=2-3, to=2-2]
			\arrow[no head, from=2-7, to=2-8]
			\arrow[no head, from=2-4, to=2-5]
			\arrow[no head, from=2-5, to=2-6]
			\arrow[no head, from=2-8, to=1-9]
			\arrow[no head, from=2-8, to=3-9]
			\arrow[no head, from=2-2, to=1-1]
			\arrow[no head, from=2-2, to=3-1]
		\end{tikzcd}  \\ \hline

		$\mathrm{B_{-}2C}_{4n} \, \lhd \, \mathrm{B_{-}2D}_{8n} $, $n$ odd & \begin{tikzcd}[column sep=tiny, row sep=tiny]
			&&&& \DLI \\
			\DLVII & \DLV & \DLVII &\DLV & \DLIV & \DLV& \DLVII & \DLV & \DLVII
			\arrow[dashed, no head, from=2-5, to=2-4]
			\arrow[dashed, no head, from=2-2, to=2-1]
			\arrow[dashed, no head, from=2-5, to=2-6]
			\arrow[dashed, no head, from=2-8, to=2-9]
			\arrow[dashed, shift left=1, no head, from=2-1, to=1-5]
			\arrow[dashed, shift left=1, no head, from=1-5, to=2-9]
			\arrow[dashed, no head, from=2-8, to=2-7]
			\arrow[dashed, no head, from=2-3, to=2-2]
			\arrow[dashed, "\cdots"{description}, no head, from=2-7, to=2-6]
			\arrow[dashed, "\cdots"{description}, no head, from=2-3, to=2-4]
		\end{tikzcd}  \\ 
		
		$\mathrm{B_{-}2C}_{4n} \, \lhd \, \mathrm{B_{-}2D}_{8n} $, $n$ even & \begin{tikzcd}[column sep=tiny, row sep=tiny]
			&&&& \DLI \\
			\DLVII & \DLV & \DLVII &\DLVII & \DLI & \DLVII& \DLVII& \DLV & \DLVII
			\arrow[dashed, no head, from=2-5, to=2-4]
			\arrow[dashed, no head, from=2-2, to=2-1]
			\arrow[dashed, no head, from=2-5, to=2-6]
			\arrow[dashed, no head, from=2-8, to=2-9]
			\arrow[dashed, shift left=1, no head, from=2-1, to=1-5]
			\arrow[dashed, shift left=1, no head, from=1-5, to=2-9]
			\arrow[dashed, no head, from=2-8, to=2-7]
			\arrow[dashed, no head, from=2-3, to=2-2]
			\arrow[dashed, "\cdots"{description}, no head, from=2-7, to=2-6]
			\arrow[dashed, "\cdots"{description}, no head, from=2-3, to=2-4]
		\end{tikzcd}  \\ \hline
	\end{tabular}
	\caption{McKay graphs of the index 2 containments into $\mathrm{B_\pm 2D}_{4n}$ decorated with the Dyson A-block structures.The double lines in the $\mathrm{B_+2C}_{4n} \lhd \mathrm{B_+2D}_{8n}$ containment are used to indicate that there are two disjoint but corresponding components of the McKay graph similar to what occurs in Figure \ref{fig:McKayGraphBpT48}.} \label{Table:B2D4nAxialAblockMcKay}
\end{table}

\pagebreak

\section{KR-Theory}
\subsection{Topological KR-Theory}
We follow Atiyah in \citep{Atiyah66} in developing topological $KR$-theory.
\begin{defn}
	A \emph{Real space} is a topological space $X$ equipped with a $2$-periodic homeomorphism $\tau : X \to X$ (i.e. $\tau^2=\mathrm{Id}_X$). This can be viewed as a space with an action of the cyclic group $C_2$ where the non-identity element acts continuously.
\end{defn}

We are interested in applying the action of $C_2 \cong \widehat{G}/G$ onto the geometric quotient space $X / G$.

\begin{lem}
	For any $C_2$-grading $G\lhd \widehat{G}$ and $\widehat{G}$-space $X$, the orbit space $X / G$ admits an action of $\widehat{G}$ with kernel $G$.
\end{lem}
\begin{proof}
	We have that $\widehat{G}$ acts continuously on $X / G$. By the construction of the orbit space,  $g\in G$ then $g$ acts as the identity. If $g \in \widehat{G}\setminus G$, $g^2 \in G$ so the action of $g$ is $2$-periodic. Furthermore, any other $h\in \widehat{G}\setminus G$ acts the same way as $hg^-1 \in G$ so
	\begin{equation*}
		g(x) = \mathbbm{1} \cdot g (x) = hg^{-1}\cdot g (x) = h(x),
	\end{equation*}
	for any $x\in X$.
\end{proof}

%Henceforth, we take $X$ to be affine space $\mathbb{A}^2_\C \simeq \C^2$ equipped with the Zariski topology.

\begin{defn} \label{Def:RealVB}
	Let $(X,\tau)$ be a Real space and its involution. A \emph{Real vector bundle} is a Real space $E$ which is a complex vector bundle over $X$ such that
	\begin{itemize}
		\item the projection map $p: E \to X$ commutes with the involutions on $E$ and $X$;
		\item the map $E_x \to E_{\tau(x)}$ is antilinear, i.e. the diagram
		% https://q.uiver.app/?q=WzAsNCxbMCwwLCJcXEMgXFx0aW1lcyBFX3giXSxbMSwwLCJFX3giXSxbMCwxLCJcXEMgXFx0aW1lcyBFX3tcXHRhdSh4KX0iXSxbMSwxLCJFX3tcXHRhdSh4KX0iXSxbMCwxXSxbMCwyXSxbMSwzXSxbMiwzXV0=
		\[\begin{tikzcd}
			{\C \times E_x} & {E_x} \\
			{\C \times E_{\tau(x)}} & {E_{\tau(x)}}
			\arrow[from=1-1, to=1-2]
			\arrow[from=1-1, to=2-1]
			\arrow[from=1-2, to=2-2]
			\arrow[from=2-1, to=2-2]
		\end{tikzcd}\]
		commutes. Here, the vertical arrows are given by the involution $\tau$ and $\C$ is endowed with its standard Real structure $\tau(z)=\bar{z}$.
	\end{itemize}
\end{defn}
Atiyah constructs the Grothendieck group of the category of Real vector bundles over a Real space $X$ and denotes this by $KR(X)$.

Atiyah and Segal similarly construct equivariant $KR$-theory in \citep{AtiyahSegal} by considering the Grothendieck group of Real $G$-equivariant vector bundles. As usual, it becomes a ring under the tensor-product. Similarly to the non-Real case, there is a relationship between the $G$-equivariant $KR$-theory and the representation theory of $G$.

\begin{defn}
	The \emph{Real representation ring} $R_\bold{R}(G)$ of a group $G$ is the Grothendieck group of the monoid of Real representations of $G$. Concretely, it is comprised of formal differences of Real representations with addition given by direct sum, and is made into a ring with multiplication given by tensor product.
\end{defn}

Atiyah and Segal note the following result in \citep{AtiyahSegal} which is analogous to Proposition~\ref{prop:KGRG}.
\begin{prop}
	\begin{equation*}
		KR_G(\mathrm{pt})\simeq KR_G(\C^2)\simeq R_\bold{R}(G)
	\end{equation*}
\end{prop}

%From Serre, analytic sheaves are the same as algebraic sheaves. Want to talk about algebraic varieties. Need KR theory of algebraic variety)

\subsection{Algebraic KR-Theory}
Again, to work with algebraic varieties we need to generalise Atiyah's work from topology to algebraic geometry to obtain an algebraic $KR$-theory. We then apply this approach to our context of $C_2$ graded Kleinian subgroups.

\begin{defn}
	Given a $C_2$-grading $G\lhd \widehat{G}$, and $X=\C^2$ with a continuous $\widehat{G}$ action, let $\C[X]*\widehat{G}$ denote the skew group algebra (see section 5 or \citep{JTDR}). We define a \emph{Real coherent sheaf} on a ringed space $X$ to be a finitely generated $\C[X]*\widehat{G}$-module $\mathcal{F}$ with an action of $\widehat{G}$ such that:
	\begin{itemize}
		\item the action of $\widehat{G}$ commutes the projection $\mathcal{F}\to X$;
		\item any $g\in G$ acts linearly on stalks $\mathcal{F}_x\to \mathcal{F}_{g\cdot x}$ ;
		\item any $g\in \widehat{G} \setminus G$ acts antilinearly on stalks $\mathcal{F}_x\to \mathcal{F}_{g\cdot x}$.
	\end{itemize}
\end{defn}

Similar to section \ref{Sec:AlgK}, this is our generalisation of a Real vector bundle. We can once again take the Grothendieck group of Real coherent sheaves on $X$ to obtain $KR(X)$. This is made into a ring under the tensor product as usual.

When working with the resolution of the quotient $ \widetilde{X \git G}$, \emph{Real coherent sheaves} are just coherent sheaves on $ \widetilde{X \git G} \simeq \GHilb(X)$ with \'etal\'e space $E\to \widetilde{X \git G}$ together with an involution $\tau_{\widehat{G}}$ which acts in the sense of definition \ref{Def:RealVB}. That is, it acts antilinearly on fibres and commutes with the projection map.

In our context of $C_2$-gradings $G\lhd \widehat{G}$, we will be taking the $KR$-theory of the resolution of the Kleinian singularity $X \git G$. Where Atiyah \citep{Atiyah66} often suppresses the choice of involution in the notation, we have several groups $G$ which admit multiple $C_2$ gradings and hence different possible involutions. We denote such a pairing of scheme and involution as $\left(\widetilde{X \git G},\tau_{\widehat{G}} \right)$.

\pagebreak
\section{McKay Correspondence on Real spaces}

We conclude this paper by conjecturing the result we’ve been building towards, an analogue of Gonzalez-Springberg and Verdier's statement of the McKay correspondence in the context of $C_2$-gradings KR-theory and Real spaces.

\begin{conj*}[Real McKay correspondence]
	For any finite subgroup $G\subset \SU(2)$ and a $\widehat{G}\in \Pin_{\pm}(3)$ such that $G\lhd \widehat{G}$ is a $C_2$ grading, there is an isomorphism between the Real representation ring $R_\bold{R}(G,\widehat{G})$ and the $G$-equivariant $KR$-theory of $X=\C^2$. Moreover, the projection maps from the fibre product induce an isomorphism between ${KR}_G(X)$ and $KR\left(\widetilde{X \git G} \right)$.
	
	\[\begin{tikzcd}[column sep=tiny]
		&& {KR\left(X\times_{X \git G} \widetilde{X\git G}\right)} \\
		{R_\bold{R}(G,\widehat{G})} & {KR_G(X)} && {KR\left(\widetilde{X \git G},\tau_{\widehat{G}}\right)}
		\arrow["\simeq"{marking}, draw=none, from=2-1, to=2-2]
		\arrow["{p_1^*}", from=2-2, to=1-3]
		\arrow["{\mathrm{Inv}\circ{p_2}_*}", from=1-3, to=2-4]
		\arrow["\simeq", from=2-2, to=2-4]
	\end{tikzcd}\]
	
\end{conj*}

There might be alternative ways to formulate of this conjecture. One might consider allowing $G \subset \Pin_{-}(3) \subset  \mathrm{U}(2) \subset \GL(2,\C)$ in the hypothesis. This certainly has a natural action onto the coordinate ring $\C[u,v]$ and we can apply Ishii’s result \citep{Ishii02} that a minimal resolution of the quotient singularity $\C \git G$ exists and that $\GHilb(\C)$ is one such.

However, for $G \subset \Pin_{+}(3) = \SU(2) \times C_2$, it is not immediately clear what the natural action on the coordinate ring would be. We’d require an order 2 action which is central in the group and distinct from the action of $-\mathbbm{1}$. It might be possible to utilise a classification of finite subgroups of $\GL(3,\C)$ where we only consider block diagonal matrices so as to permit a distinct central element of order $2$, but this might just raise extra issues trying to resolve the resulting singularity.

% https://q.uiver.app/?q=WzAsNyxbMCwzLCJSKEcpIl0sWzEsMywiS19HKFgpIl0sWzIsMiwiSyhYXFx0aW1lc197WC9HfSBcXHdpZGV0aWxkZXtYXFxnaXQgR30pIl0sWzMsMywiSyhcXHdpZGV0aWxkZXtYIFxcZ2l0IEd9KSJdLFsxLDEsIlgiXSxbMiwwLCJYXFx0aW1lc197WC9HfSBcXHdpZGV0aWxkZXtYXFxnaXQgR30iXSxbMywxLCJcXHdpZGV0aWxkZXtYXFxnaXQgR30iXSxbMCwxLCJcXHNpbWVxIiwzLHsic3R5bGUiOnsiYm9keSI6eyJuYW1lIjoibm9uZSJ9LCJoZWFkIjp7Im5hbWUiOiJub25lIn19fV0sWzEsMiwicF8xXioiXSxbMiwzLCJcXG1hdGhybXtJbnZ9XFxjaXJje3BfMn1fKiJdLFs1LDYsInBfMiJdLFs1LDQsInBfMSIsMl0sWzEsMywiXFxzaW1lcSJdXQ==

\pagebreak

\bibliography{citelist}

\begin{thebibliography}{34}
\providecommand{\natexlab}[1]{#1}
\providecommand{\url}[1]{\texttt{#1}}
\expandafter\ifx\csname urlstyle\endcsname\relax
  \providecommand{\doi}[1]{doi: #1}\else
  \providecommand{\doi}{doi: \begingroup \urlstyle{rm}\Url}\fi

\bibitem[Artin(1966)]{MA}
Michael Artin.
\newblock On isolated rational singularities of surfaces.
\newblock \emph{American Journal of Mathematics}, 88\penalty0 (1):\penalty0
  129--136, 1966.
\newblock ISSN 00029327, 10806377.
\newblock URL \url{http://www.jstor.org/stable/2373050%7D}.

\bibitem[Atiyah(1966)]{Atiyah66}
M.~F. Atiyah.
\newblock K-theory and {R}eality.
\newblock \emph{The Quarterly Journal of Mathematics}, 17\penalty0
  (1):\penalty0 367--386, 01 1966.
\newblock ISSN 0033-5606.
\newblock \doi{10.1093/qmath/17.1.367}.
\newblock URL \url{https://doi.org/10.1093/qmath/17.1.367}.

\bibitem[Atiyah(1967)]{Atiyah67}
M.~F. Atiyah.
\newblock K-theory, 1967.
\newblock URL \url{https://www.maths.ed.ac.uk/~v1ranick/papers/atiyahk.pdf}.

\bibitem[Atiyah and Segal(1969)]{AtiyahSegal}
M.~F. Atiyah and G.~B. Segal.
\newblock Equivariant {K}-theory and completion.
\newblock \emph{Journal of Differential Geometry}, 3\penalty0 (1-2):\penalty0 1
  -- 18, 1969.
\newblock \doi{10.4310/jdg/1214428815}.
\newblock URL \url{https://doi.org/10.4310/jdg/1214428815}.

\bibitem[Atiyah et~al.(1964)Atiyah, Bott, and Shapiro]{AtiyahBottShapiro}
M.F. Atiyah, R.~Bott, and A.~Shapiro.
\newblock Clifford modules.
\newblock \emph{Topology}, 3:\penalty0 3--38, 1964.
\newblock ISSN 0040-9383.
\newblock \doi{https://doi.org/10.1016/0040-9383(64)90003-5}.
\newblock URL
  \url{https://www.sciencedirect.com/science/article/pii/0040938364900035}.

\bibitem[Carrasco~Serrano(2014)]{JCS}
Javier Carrasco~Serrano.
\newblock Finite subgroups of {SL}(2, {C}) and {SL}(3, {C}), May 2014.
\newblock URL
  \url{https://homepages.warwick.ac.uk/~masda/McKay/Carrasco_Project.pdf}.

\bibitem[Cheah(2020)]{joncheahSU2}
Jon Cheah.
\newblock Finite subgroups of {SL}(2,{C}) and {SL}(3,{C}), September 2020.
\newblock URL
  \url{https://drive.google.com/file/d/1cOw6TLnzWz8IADBw_xe0L-OuhXq3t5VV/view}.

\bibitem[Cheah(2022)]{joncheahMAGMA}
Jon Cheah.
\newblock Magma code on {P}in(3) subgroups, March 2022.
\newblock URL \url{https://doi.org/10.6084/m9.figshare.19390763}.

\bibitem[Conway and Smith(2003)]{CS}
John~H. Conway and Derek~A. Smith.
\newblock \emph{On Quaternions and Octonions: Their Geometry, Arithmetic, and
  Symmetry}.
\newblock A K Peters Ltd., 2003.
\newblock ISBN 1-56881-134-9.

\bibitem[Dolgachev(2007)]{IVD}
Igor~V. Dolgachev.
\newblock Mckay correspondence. winter 2006/07, October 2007.
\newblock URL \url{http://www.math.lsa.umich.edu/~idolga/McKaybook.pdf}.

\bibitem[Dyson(1962)]{Dyson}
Freeman~J. Dyson.
\newblock The threefold way. algebraic structure of symmetry groups and
  ensembles in quantum mechanics.
\newblock \emph{Journal of Mathematical Physics}, 3\penalty0 (6):\penalty0
  1199--1215, 1962.
\newblock \doi{10.1063/1.1703863}.
\newblock URL \url{https://doi.org/10.1063/1.1703863}.

\bibitem[Farley and Ortiz(2014)]{FarleyOrtiz}
Daniel Farley and Ivonne Ortiz.
\newblock Algebraic {K}-theory of crystallographic groups: The
  three-dimensional splitting case, 01 2014.

\bibitem[Finegan(2020)]{FineganMcKay}
Oscar Finegan.
\newblock The {M}c{K}ay correspondence, 2020.
\newblock URL
  \url{https://people.bath.ac.uk/ac886/students/OscarFineganProject.pdf}.

\bibitem[Geiko and Moore(2021)]{GM}
Roman Geiko and Gregory~W. Moore.
\newblock Dyson’s classification and real division superalgebras.
\newblock \emph{Journal of High Energy Physics}, 2021\penalty0 (4), Apr 2021.
\newblock ISSN 1029-8479.
\newblock \doi{10.1007/jhep04(2021)299}.
\newblock URL \url{http://dx.doi.org/10.1007/JHEP04(2021)299}.

\bibitem[Gonzalez-Sprinberg and Verdier(1983)]{GSV}
Gerardo Gonzalez-Sprinberg and Jean~louis Verdier.
\newblock Construction g\'eom\'etrique de la correspondance de mckay.
\newblock \emph{Annales scientifiques de l'\'Ecole Normale Sup\'erieure}, 4e
  s{\'e}rie, 16\penalty0 (3):\penalty0 409--449, 1983.
\newblock \doi{10.24033/asens.1454}.
\newblock URL \url{http://www.numdam.org/articles/10.24033/asens.1454/}.

\bibitem[Green(2014)]{JJGreen}
James~J. Green.
\newblock Resolution of {K}leinian singularities, April 2014.
\newblock URL \url{https://people.bath.ac.uk/ac886/students/JamesADE.pdf}.

\bibitem[Hahn and O'Meara(1980)]{HahnOmeara}
Alexander~J. Hahn and O.~Timothy O'Meara.
\newblock \emph{The Classical Groups and {K}-Theory}.
\newblock Number 291 in A Series of Comprehensive Studies in Mathematics.
  Springer-Verlag, Berlin, Heidelberg, 1980.
\newblock ISBN 978-3-642-05737-3.

\bibitem[Hartshorne(1977)]{Hartshorne}
Robin Hartshorne.
\newblock \emph{Algebraic Geometry}.
\newblock New York : Springer, 8th edition, 1977.
\newblock ISBN 9780387902449.

\bibitem[Harvey(1990)]{Harvey}
F.~Reese Harvey.
\newblock \emph{Spinors and Calibrations}, volume~9 of \emph{Perspectives in
  Mathematics}.
\newblock Academic Press Inc., 22-28 Oval Road, London, 1990.

\bibitem[Hatcher(2017)]{HatcherVBKT}
Allen Hatcher.
\newblock Vector bundles and {K}-theory, 2017.
\newblock URL \url{https://pi.math.cornell.edu/~hatcher/VBKT/VB.pdf}.

\bibitem[Ishii(2002)]{Ishii02}
A.~Ishii.
\newblock On the {M}ckay correspondence for a finite small subgroup of
  {GL}(2,{C}).
\newblock \emph{Journal Fur Die Reine Und Angewandte Mathematik - J REINE ANGEW
  MATH}, 2002:\penalty0 221--233, 01 2002.
\newblock \doi{10.1515/crll.2002.064}.

\bibitem[Ito and Nakamura(1999)]{ItoNakamura99}
Y.~Ito and I.~Nakamura.
\newblock \emph{Hilbert schemes and simple singularities}, page 151–234.
\newblock London Mathematical Society Lecture Note Series. Cambridge University
  Press, 1999.
\newblock \doi{10.1017/CBO9780511721540.008}.

\bibitem[Klein(1888)]{FK}
Felix Klein.
\newblock \emph{Lectures On The Ikosahedron And The Solution Of The Fifth
  Degree}.
\newblock London : Tr\"ubner \& Co, 1888.

\bibitem[McKay(1980)]{McKay80}
John McKay.
\newblock Graphs, singularities, and finite groups.
\newblock In \emph{The {S}anta {C}ruz {C}onference on {F}inite {G}roups
  ({U}niv. {C}alifornia, {S}anta {C}ruz, {C}alif., 1979)}, volume~37 of
  \emph{Proc. Sympos. Pure Math.}, pages 183--186. Amer. Math. Soc.,
  Providence, R.I., 1980.

\bibitem[Mumford et~al.(1994)Mumford, Fogarty, and Kirwan]{MumfordGIT}
David Mumford, John Fogarty, and Frances Kirwan.
\newblock \emph{Geometric Invariant Theory}.
\newblock Springer-Verlag, Berlin, Heidelberg, 3 edition, 1994.
\newblock ISBN 978-3-540-56963-3.

\bibitem[Newstead(2006)]{NewsteadGIT}
P.E. Newstead.
\newblock Geometric invariant theory, November 2006.
\newblock URL
  \url{https://www.cimat.mx/Eventos/c_vectorbundles/newstead_notes.pdf}.

\bibitem[Reid()]{ReidDuVal}
Miles Reid.
\newblock The {D}u {V}al singularities {A}n, {D}n, {E}6, {E}7, {E}8.
\newblock URL \url{https://homepages.warwick.ac.uk/~masda/surf/more/DuVal.pdf}.

\bibitem[Rumynin and Taylor(2021)]{JTDR}
Dmitriy Rumynin and James Taylor.
\newblock Real representations of {C}2-graded groups: The antilinear theory.
\newblock \emph{Linear Algebra and its Applications}, 610:\penalty0 135–168,
  Feb 2021.
\newblock ISSN 0024-3795.
\newblock \doi{10.1016/j.laa.2020.09.040}.
\newblock URL \url{http://dx.doi.org/10.1016/j.laa.2020.09.040}.

\bibitem[Segal(1968{\natexlab{a}})]{SegalEqK}
Graeme Segal.
\newblock Equivariant {K}-theory.
\newblock \emph{Publications mathématiques de l'IHÉS}, 34\penalty0
  (3):\penalty0 129--151, 1968{\natexlab{a}}.
\newblock \doi{10.1007/BF02684593}.
\newblock URL \url{http://www.numdam.org/item/PMIHES_1968__34__129_0.pdf}.

\bibitem[Segal(1968{\natexlab{b}})]{SegalRR}
Graeme Segal.
\newblock The representation-ring of a compact {Lie} group.
\newblock \emph{Publications Math\'ematiques de l'IH\'ES}, 34:\penalty0
  113--128, 1968{\natexlab{b}}.
\newblock URL \url{http://www.numdam.org/item/PMIHES_1968__34__113_0/}.

\bibitem[Slodowy(1983)]{Slodowy}
P.~Slodowy.
\newblock {P}latonic solids, {K}leinian singularities, and {L}ie groups.
\newblock In I.~Dolgachev, editor, \emph{Algebraic Geometry}, pages 102--138,
  Berlin, Heidelberg, 1983. Springer Berlin Heidelberg.
\newblock ISBN 978-3-540-40971-7.

\bibitem[Stekolshchik(2008)]{RS}
Rafaek Stekolshchik.
\newblock \emph{Notes on Coxeter Transformations and the McKay Correspondence}.
\newblock Springer-Verlag, Berlin, Heidelberg, 2008.

\bibitem[Thomason(2016)]{Thomason}
R.~W. Thomason.
\newblock \emph{XX. Algebraic {K}-Theory of Group Scheme Actions}, pages
  539--563.
\newblock Princeton University Press, 2016.
\newblock \doi{doi:10.1515/9781400882113-021}.
\newblock URL \url{https://doi.org/10.1515/9781400882113-021}.

\bibitem[Weibel(2013)]{Weibel}
Charles~A. Weibel.
\newblock \emph{The {K}-book: an introduction to algebraic {K}-theory}, volume
  145 of \emph{Graduate Studies in Mathematics}.
\newblock American Mathematical Society, 07 2013.
\newblock ISBN 9780821891322.
\newblock URL \url{https://sites.math.rutgers.edu/~weibel/Kbook.html}.

\end{thebibliography}

% These are examples of the book titles as they can be put in. 
% You can use different labeling methods, if you prefer, 
% e.g., just enumerating the items  as [1], [2], etc.,
% instead of making them [ATLAS], [B], etc.
% Or use bibtex

%[ATLAS] J.H.Conway, R.T.Curtis, S.P.Norton, R.A.Parker, R.A.Wilson, \it ATLAS of Finite Groups \rm, Clarendon Press, Oxford, 1985.
%
%[B] A.Bereczky, Maximal Overgroups of Singer Elements in Classical Groups, Journal of Algebra 234 (2000), 187-206.

\end{document}